\documentclass[a4paper,12pt,reqno]{amsart}
\usepackage{hyperref}
\hypersetup{backref,colorlinks=true,allcolors=black}
\usepackage{mathrsfs}
\usepackage{enumitem}
\usepackage{dsfont}
\usepackage{mathtools}
\usepackage{xcolor}
\usepackage[final]{graphicx}
\usepackage{upgreek}
\usepackage{amsmath,amssymb,bbm,bm}
\usepackage{nameref}
\usepackage{comment}

\newtheorem{theorem}{Theorem}[section]
\newtheorem{proposition}[theorem]{Proposition}
\newtheorem{lemma}[theorem]{Lemma}
\newtheorem{corollary}[theorem]{Corollary}

\theoremstyle{definition}

\newtheorem{remark}[theorem]{Remark}

\newtheorem{innercustomthm}{Condition}
\newenvironment{customcondition}[1]
{\renewcommand\theinnercustomthm{#1}\innercustomthm}
{\endinnercustomthm}

\newcommand{\ignore}[1]{}
\usepackage{color}
\usepackage{xcolor}
\definecolor{crimson}{rgb}{0.7294,0.0666,0.0470}

\newcommand{\R}{\mathbb{R}}
\newcommand{\N}{\mathbb{N}}

\newcommand{\norm}[1]{\left\lVert #1 \right\rVert}
\newcommand{\abs}[1]{\left\vert #1 \right\rvert}

\newcommand{\E}[1]{\mathbb{E}{\left[ #1\right]}}

\DeclareMathOperator*{\plim}{\mathbb{P}-lim}

\DeclarePairedDelimiter\autobracket{(}{)}
\newcommand{\brac}[1]{\autobracket*{#1}}
\newcommand{\inner}[1]{\left\langle #1 \right\rangle}

\newcommand{\der}{{\mathbb{D}}}

\newcommand{\bracsq}[1]{\left[ #1\right]}

\theoremstyle{definition}

\makeatletter
\@addtoreset{proofcase}{theorem}
\makeatother

\theoremstyle{definition}

\makeatletter
\@addtoreset{proofstep}{theorem}
\makeatother

\usepackage{microtype}
\usepackage{parskip}

\title{Gradient estimates for semigroups associated with stochastic differential equations driven by cylindrical L\'{e}vy processes}

\author{Thanh~Dang$^1$}
\address{$^1$Department of Mathematics, Florida State University, 1017
  Academic Way, Tallahassee, FL 32306, USA}
\email{ycloud777@gmail.com}
\author{Lingjiong Zhu$^2$}\thanks{Lingjiong Zhu is partially supported by the grants NSF DMS-2053454, NSF DMS-2208303.} 
\address{$^2$Department of Mathematics, Florida State University, 1017
  Academic Way, Tallahassee, FL 32306, USA}
\email{zhu@math.fsu.edu}

\date{}

\begin{document}

\begin{abstract}
 Via a Bismut-Elworthy-Li formula from \cite{kulik2023gradient}, we derive uniform gradient estimates for transition semigroups associated with stochastic differential equations driven by a large class of cylindrical L\'{e}vy processes which includes the important case of cylindrical $\alpha$-stable processes. As the first application, we formulate a Stein's method for quantitative approximation of the invariant measure of these stochastic differential equations in Wasserstein distance. As the second and main application, we study Euler-Maruyama numerical schemes of stochastic differential equations driven by stable L\'{e}vy processes with i.i.d. stable components and obtain a uniform-in-time approximation error in Wasserstein distance. Our approximation error has a linear dependence on the stepsize, which is expected to be tight, as can be seen from an explicit calculation for the case of an Ornstein-Uhlenbeck process.
\end{abstract}

\maketitle

\section{Introduction}

In \cite{kulik2023gradient}, Kulik, Peszat and Priola consider the stochastic differential equation
\begin{align}
\label{sde_cylindricallevy}
    X_t^x=x+\int_0^t b(X_t^x)dt+Z_t,
\end{align}
where
$Z_t:=\brac{Z^1_t,\cdots,Z^d_t}^T$
is a $d$-dimensional cylindrical L\'{e}vy process, i.e. $\{Z^j:1\leq j\leq d\}$ is a family of independent real-valued L\'{e}vy processes. Denote $m$ as the L\'{e}vy measure of $Z$ then under quite general assumptions on $m$ (which includes the L\'{e}vy measure of cylindrical stable processes), they establish a Bismut-Elworthy-Li formula for the transition semigroups $\{P_t:t\geq 0 \}$ associated with \eqref{sde_cylindricallevy}:
\begin{align}
\label{bismut_intro}
    \nabla P_tf(x)= \E{f(X^x_t)G(x,t)}, \qquad f\in \mathcal{C}_b(\R^d),
\end{align}
where the random field $G(x,t)$ does not depend on $f$. Such integration-by-parts formula is a powerful tool in stochastic analysis and has been used to obtain Harnack inequalities, heat kernel estimates as well as strong Feller properties in \cite{arnaudon2006harnack,arnaudon2009gradient,da2009singular,liu2008harnack,wang2007harnack,zhang2010white,wang2012derivative} among others. While Bismut-Elworthy-Li formulas for semigroups of stochastic differential equations driven by L\'{e}vy noise has been studied by several authors, see for instance \cite{takeuchi2010bismut, zhang2013derivative,xuzhang2015gradient}, the case of L\'{e}vy noise with singular L\'{e}vy measure considered in \cite{kulik2023gradient} is arguably more technically challenging to study, as explained in the introduction and also at the end of the second section in \cite{kulik2023gradient}. More generally, there is a vast amount of literature devoted to the study of cylindrical L\'{e}vy process, see e.g. \cite{liu2022alpha,
barndorff2001multivariate,ye2018stochastic, pruitt1969sample,kulczycki2022weak,kulczycki2021strong,kulczycki2017transitioncylin,
bass2006systems,
friesen2021aniso,
bogdan2020heat,
zhang2014aniso,
wang2015harnack}. 

Relying on formula \eqref{bismut_intro}, we are able to derive in Proposition \ref{prop_semigroupestimate} uniform gradient estimates for the semigroups associated the solution to \eqref{sde_cylindricallevy}. Based on these estimates, our first application is to develop a Stein's method for approximation of the invariant measure $\nu$ of the process in  \eqref{sde_cylindricallevy}. Specifically, formula \eqref{bismut_intro} enables us to bound the second derivative of the solution to the Stein's equation, which leads to an upper bound of the Wasserstein distance between the law of a generic random variable $F$ and the invariant measure $\nu$ of \eqref{sde_cylindricallevy} (Theorem~\ref{theorem_steinmethodgeneralbound}). The idea of using Bismut-Elworthy-Li formula to study solution to the Stein's equation has appeared earlier in \cite{FSX19,Gor19,erdogdu2018global} in the context of approximating invariant measure of It\^{o} diffusion processes. 

Furthermore, in our second and main application of Proposition~\ref{prop_semigroupestimate}, we assume the cylindrical L\'{e}vy process to be an $\alpha$-stable L\'{e}vy process with i.i.d. stable components and consider the stochastic differential equation (SDE):
\begin{align}
\label{SDE_stable_intro}
X_t&=b(X_t)dt+\ dL^\alpha_t,\quad X_0=x\in\mathbb{R}^{d}.
\end{align}
Here $\{L^\alpha_t:t\geq 0\},1< \alpha< 2$ denotes a $d$-dimensional $\alpha$-stable L\'{e}vy process with independent and identically distributed (i.i.d.) $\alpha$-stable components and $b(\cdot):\mathbb{R}^{d}\rightarrow\mathbb{R}^{d}$ is the drift term. An Euler-Maruyama scheme of this equation is
\begin{align*}
Y_{m+1}&=Y_m+\eta b(Y_m)+\xi_m, \quad Y_0=x\in\mathbb{R}^{d}, 
\end{align*}
 where $\eta \in[0,1]$ is the stepsize and $\{\xi_m:m\geq 1\}$ are the independent increments of length $\eta$, i.e.
$\xi_{m}:= L^\alpha_{(m+1)\eta}-L^\alpha_{m\eta}$. 

Let us provide here some background and practical motivation for our second application. The study of the Euler-Maruyama scheme of SDE has a long history in the probability and numerical analysis literature. 
In particular, the approximation of Euler-Maruyama scheme of SDE driven by L\'{e}vy noise has been extensively
studied; see e.g. \cite{Janicki1996,Protter1997,Pamen2017,Mikulevicius2018,Kuhn2019,chenxu2023euler}. 
However, the majority of these works obtain an approximation error of Euler-Maruyama scheme
on a fixed time interval $[0,T]$, and the approximation error often grows to infinity
as $T\rightarrow\infty$. The recent applications of L\'{e}vy-driven SDE in the machine learning community
reignite the interest of studying the Euler-Maruyama scheme.
The L\'{e}vy-driven SDEs naturally appear in the following two scenarios
in the machine learning applications.

First, stochastic gradient descent (SGD) methods are one of the most popular algorithms
for solving many optimization problems that arise in machine learning, especially deep learning.
It has been numerically observed that the gradient noise often becomes heavy-tailed over iterations in deep learning practice 
\cite{csimcsekli2019heavy,pmlr-v97-simsekli19a,ht_sgd_quad}.
Recent theoretical studies \cite{ht_sgd_quad, hodgkinson2020multiplicative} showed that heavy tails 
can arise in SGD even under surprisingly simple settings such as linear regression with Gaussian data. 
To better understand the effect of heavy-tails in SGD,
one often approximates the gradient noise by $\alpha$-stable distributions, 
and L\'{e}vy-driven SDE arises naturally as the continuous limit as the stepsize goes to zero \cite{pmlr-v97-simsekli19a,nguyen2019first,simsekli2020hausdorff,raj2023algorithmic,raj2023general}.
Such an approximation allows studies on the correlation of heavy-tailedness with generalization performance 
which is of key interest in machine learning
\cite{pmlr-v97-simsekli19a,simsekli2020hausdorff,raj2023algorithmic,raj2023general}.

Second, fractional Langevin algorithms are proposed and studied in recent machine learning literature
that can target a given distribution. 
The algorithms are based on discretizations of L\'{e}vy-driven SDE
where the drift term involves fractional derivatives \cite{simcsekli2017fractional,nguyen19}.
Such heavy-tailed sampling algorithms can be useful in both
large-scale sampling problems such as Bayesian learning \cite{simcsekli2017fractional}, 
as well as non-convex optimization problems
that arise in machine learning \cite{nguyen19}.

The recent application of L\'{e}vy-driven SDE in machine learning 
calls for uniform-in-time approximation error of Euler-Maruyama scheme.
The motivation is that for an arbitrarily small target accuracy, 
the number of iterates for the stochastic algorithms of interest
can be arbitrarily large. Therefore, one would like to have a quantitative control
on the discretization error that is uniform-in-time and will not grow
to infinity as the number of iterates increases to infinity. See e.g. \cite{simcsekli2017fractional,nguyen19}.

In a seminal work, \cite{chenxu2023euler} studied uniform-in-time 
Euler-Maruyama approximation of SDE driven by rotationally invariant $\alpha$-stable L\'{e}vy processes.
They studied two discretization schemes, based on the $\alpha$-stable distributed noise
and the Pareto distributed noise, and proved the $1$-Wasserstein error bounds in the order
of $\mathcal{O}(\eta^{1-\epsilon})$ and $\mathcal{O}(\eta^{\frac{2}{\alpha}-1})$ respectively 
as the stepsize $\eta\rightarrow 0$,
where $\epsilon\in(0,1)$ can be arbitrarily small. They showed that the discretization error
in the Pareto scheme is tight in terms of the stepsize dependence.

In contrast to \cite{chenxu2023euler}, we consider uniform-in-time 
Euler-Maruyama approximation of SDE driven by $\alpha$-stable L\'{e}vy processes with i.i.d. components.
While rotationally invariant $\alpha$-stable L\'{e}vy-driven SDE has been sometimes used
to approximate heavy-tailed SGD in machine learning literature \cite{raj2023algorithmic,raj2023general},
$\alpha$-stable L\'{e}vy processes with i.i.d. components
are often preferred in some other settings in machine learning applications,
such as in the construction of fractional Langevin Monte Carlo algorithms
that can target a given Gibbs distribution \cite{simcsekli2017fractional,nguyen19,CWZGHS20}. The noise structure of the i.i.d. components
makes it relatively easy to design fractional Langevin Monte Carlo algorithms
in which each component in the drift term involves a one-dimensional fractional derivative \cite{simcsekli2017fractional,nguyen19,CWZGHS20}. 

The $1$-Wasserstein distance between two probability measures $\mu$ and $\nu$ on $\mathbb{R}^{d}$ is defined as (\cite{villani2008optimal})
\begin{equation*}
d_{\operatorname{Wass}}(\mu,\nu)=\sup_{h\in\text{Lip}(1)}\left|\int_{\mathbb{R}^{d}}h(x)\mu(dx)-\int_{\mathbb{R}^{d}}h(x)\nu(dx)\right|,
\end{equation*}
where $\text{Lip}(1)$ consists of the functions $h:\mathbb{R}^{d}\rightarrow\mathbb{R}$
that are $1$-Lipschitz. In Theorem \ref{theorem_eulerscheme}, we provide an upper bound of the convergence rate of the Euler-Maruyama scheme in the 1-Wasserstein distance. Specifically, under suitable conditions, we will show that
\begin{align}
\label{estimate_euler_intro}
    d_{\operatorname{Wass}}\left(\operatorname{Law}(X_{\eta N}),\operatorname{Law}(Y_N)\right)\leq \mathcal{C} \eta,
\end{align}
for a constant $\mathcal{C}$ that will be made explicit. Our Euler-Maruyama approximation \eqref{estimate_euler_intro} has a linear dependence on the stepsize,
which is expected to be tight (that can be seen 
through an explicit calculation for an Ornstein-Uhlenbeck process driven by a one-dimensional stable L\'{e}vy process). 

In addition and as our second Euler-Maruyama approximation, we consider in Appendix~\ref{appendix_eulerschemepareto} a second discretization scheme of Equation~\eqref{SDE_stable_intro} using the i.i.d. 
Pareto noise, inspired by \cite{chenxu2023euler}: 
\begin{align*}
    U_{m+1}=U_m+\eta b(U_m)+\sigma_{\eta,\alpha}\zeta_m, \quad U_0=x\in\mathbb{R}^{d}, 
\end{align*}
where $\sigma_{\eta,\alpha}$ is a proper scaling parameter 
and $\zeta_{m}$ are i.i.d. $d$-dimensional Pareto noises that we will specify later. We will obtain an approximation error under this regime as well.
In particular, we are able to show that
\begin{align}
\label{estimate_euler_intro_pareto}
    d_{\operatorname{Wass}}\left(\operatorname{Law}(X_{\eta N}),\operatorname{Law}(U_N)\right)\leq \mathcal{C}' \eta^{2/\alpha-1},
\end{align}
for a constant $\mathcal{C}'$ that will be made explicit.  
The dependence on the stepsize $\eta^{2/\alpha-1}$ is the same as in \cite{chenxu2023euler}, which is expected
to be tight as is indicated by the discussions in \cite{chenxu2023euler}.

The recent papers \cite{chenxu2023euler,chenxufollowuppaper,bally2023approximation} also study uniform-in-time Euler-Maruyama scheme for SDE driven by L\'{e}vy process. What distinguishes our paper from the aforementioned references is that while they assume L\'{e}vy processes whose L\'{e}vy measures are absolutely continuous with respect to the Lebesgue measure, we are able to handle cylindrical L\'{e}vy processes with singular L\'{e}vy measures by employing a Malliavin calculus framework recently established in \cite{kulik2023gradient}. 
In particular, the authors
of \cite{chenxu2023euler,chenxufollowuppaper} consider SDE driven by $d$-dimensional rotationally invariant $\alpha$-stable L\'{e}vy process $Z^\alpha$. Since $Z^\alpha$ has the representation
    \begin{align*}
   Z^\alpha_t=     B_{S_t}=\brac{B^1_{S_t},\ldots,B^d_{S_t}},
    \end{align*}
    where $\{B_t=\brac{B^1_t,\ldots,B^d_t}:t\geq 0\}$ is an $\R^d$-valued Brownian motion and $S$ is a one-dimensional $\alpha/2$-stable subordinator independent from $B$, they are able to perform a time change to transform their L\'{e}vy driven SDE into an It\^{o} diffusion SDE, after which they apply the classical Malliavin calculus on Wiener space. In the context of our paper, a $d$-dimensional $\alpha$-stable L\'{e}vy process with i.i.d. components $L$ can be written as 
    \begin{align*}
        \brac{L^1_t,\ldots, L^d_t}= \brac{B^1_{S^1_t},\ldots,B^d_{S^d_t}},
    \end{align*}
 where for each $i$, $B^i$ is a one-dimensional Brownian motion and $S^i$ is a one-dimensional $\alpha/2$-stable subordinator independent from $B^i$. When $d\geq 2$, the time change argument in \cite{chenxu2023euler,chenxufollowuppaper} does not work for SDE driven by $L$, and that is where the framework in \cite{kulik2023gradient} comes in.

The paper is organized as follows. We will first provide important definitions and notations in Section~\ref{section_notations}. Then we introduce a few properties for
SDEs driven by a general class of cylindrical L\'{e}vy processes in Section~\ref{sec:cylindrical}. 
In particular, Section~\ref{sec:cylindrical} includes the semigroup gradient estimates that leads to a Stein's bound. 
The application to Euler-Maruyama schemes whose innovations are increments of $\alpha$-stable L\'{e}vy processes 
with i.i.d. components will be provided in Section~\ref{section_eulerscheme}. The proofs of the results in Section~\ref{section_eulerscheme} are presented in Sections~\ref{sec:proofs:1} and \ref{section_proof_eulerscheme}.
In addition, we will also provide in Appendix~\ref{appendix_eulerschemepareto}  a result on Euler-Maruyama schemes whose innovations are multivariate Pareto distribution. The background on Malliavin calculus on Poisson space, additional technical proofs,
and the illustration of the tightness of the stepsize dependence in our second application through an Ornstein-Uhlenbeck process
will also be provided in subsequent sections of the Appendices.

\section{Definitions and notations}
\label{section_notations}

We collect here definitions and notations that appear throughout the paper. 

\begin{itemize}
    \item the operator norm of a linear map $T:\mathbb{R}^{d}\to\mathbb{R}^{d}$ is $\norm{T}_{\operatorname{op}}:=\sup_{v\in\mathbb{R}^{d}:\norm{v}=1}\norm{Tv}$. 
    \item the $1$-Wasserstein distance between two probability measures $\mu$ and $\nu$ on $\mathbb{R}^{d}$ is (\cite{villani2008optimal})
\begin{equation*}
d_{\operatorname{Wass}}(\mu,\nu)=\sup_{h\in\text{Lip}(1)}\left|\int_{\mathbb{R}^{d}}h(x)\mu(dx)-\int_{\mathbb{R}^{d}}h(x)\nu(dx)\right|,
\end{equation*}
where $\text{Lip}(1)$ consists of the functions $h:\mathbb{R}^{d}\rightarrow\mathbb{R}$
that are $1$-Lipschitz. 
\item $\{Z_t:t\geq 0\}$
is a $d$-dimensional cylindrical L\'{e}vy process, i.e. $Z_t:=\brac{Z^1_t,\cdots,Z^d_t}^T$ and $\{Z^j:1\leq j\leq d\}$ is a family of independent real-valued L\'{e}vy processes.
\item In Section~\ref{sec:cylindrical}, $\nu$ is the invariant probability measure of \eqref{sde_cylindricallevy}. In Section~\ref{section_eulerscheme}, $\nu$ is the invariant probability measure of \eqref{SDE_alphastable} which is a special case of \eqref{sde_cylindricallevy}. $W$ denotes a random variable distributed as $\nu$.  $m$ is the the L\'{e}vy measure of $Z$ and $m_j$ is the L\'{e}vy measure of $Z^j$. $\rho_j$ is the density of $m_j$ with respect to the Lebesgue measure.

\item $\pi,R,\kappa,\tau$ and $\Lambda$ are the constants in Condition~\ref{cond_levymeasure}. $\theta_1,\theta_2$ and $\theta_3$ are the constants in Condition~\ref{cond_driftb}. $\theta_4$ is the constant in Condition~\ref{cond_diss}.
\item In Section \ref{sec:cylindrical}, $\{P_t:t\geq 0\}$ is the semigroups of the solution of \eqref{sde_cylindricallevy}. Later in Section \ref{section_eulerscheme}, $\{P_t:t\geq 0\}$ is the semigroups of the solution of \eqref{SDE_alphastable} which is a special case of \eqref{sde_cylindricallevy}. 

\item For $u,v\in \R^d$, any Lipschitz function $f$ on $\R^d$ and $x\in \R^d$, we wite $\nabla_u P_t f(x):=\nabla P_tf(x)u$ and $\nabla_v\nabla_u P_tf(x):=v^T\nabla^2P_tf(x)u$. 

\item Let $x\in \R^{d}$. The Dirac delta function $\delta_x$ is understood in the distributional sense and $\delta_x P_t$ is the law of $X_t$ in \eqref{sde_cylindricallevy} given $X_0=x$ a.s.

\item $\{ L^\alpha_t:t\geq 0 \}$ is an $\alpha$-stable L\'{e}vy process that has $1<\alpha<2$ and i.i.d. $\alpha$-stable components. 
\item In Section~\ref{section_eulerscheme}, $\{Y_m:m\in \N\}$ is the Euler-Maruyama discretization of \eqref{SDE_alphastable} that has step size $\eta$ and $\alpha$-stable noise $\{\xi_m=L_{(m+1)\eta}-L_{m\eta}:m\in\N\}$. $\nu_\eta$ is the associated invariant probability measure and $\{Q_k:k\in\N \}$ the associated semigroups, that is $Q_kf(x)=\E{f(Y^x_k)}$ for $x\in R^d$ and any Lipschitz function $f$. 

\item $ V_{\lambda }(x):=\brac{1+\abs{x}^2}^{\lambda /2}$ with suitable constant $\lambda$ is the Lyapunov function associated with \eqref{SDE_alphastable}. 
\item  In Appendix~\ref{appendix_eulerschemepareto}, $\{U_k:k\in\N \}$ is the discretization of \eqref{SDE_alphastable} that has step size $\eta$ and Pareto noise $\{\zeta_k:k\in \N\}$. $\chi_\eta$ is the associated invariant probability measure and $\{T_k:k\in\N \}$ is the associated semigroups,  that is $T_kf(x)=\E{f(U^x_k)}$ for $x\in R^d$ and any Lipschitz function $f$. 
\end{itemize}

\section{SDEs driven by a class of cylindrical L\'{e}vy processes}\label{sec:cylindrical}

In this section, we consider the stochastic differential equation \eqref{sde_cylindricallevy} driven by a $d$-dimensional cylindrical L\'{e}vy process $Z_t:=\brac{Z^1_t,\cdots,Z^d_t}^T$, i.e. $\{Z^j:1\leq j\leq d\}$ is a family of independent real-valued L\'{e}vy processes. We will denote $m$ as the L\'{e}vy measure of $Z$ and $m_j$ as the L\'{e}vy measure of $Z^j$. 

Our main goal of this section is to obtain semigroup gradient estimates for the semigroups associated with the solution to \eqref{sde_cylindricallevy}
(Proposition~\ref{prop_semigroupestimate}). As the first application of the semigroup gradient estimates, we develop a Stein's method for approximation of the invariant measure $\nu$ of the process in  \eqref{sde_cylindricallevy} and provide a bound in Wasserstein distance between the law of a generic random variable $F$ and $\nu$ (Theorem~\ref{theorem_steinmethodgeneralbound}).

First, we will impose the following conditions on the  L\'{e}vy measure $m_j$ and the drift coefficient $b$. Our goal is to be able to adopt the Malliavin calculus developed in \cite{kulik2023gradient}.
\begin{customcondition}{H1}(on L\'{e}vy measure $m$) 
		\label{cond_levymeasure}
  \begin{enumerate}[label=\roman*)]
\item There exists some $\pi>0$ such that for any $1\leq j\leq d$, 
\begin{align}
\label{cond_pi}
    \liminf_{\epsilon\to 0}\epsilon^\pi m_j\brac{\abs{\xi}\geq \epsilon}\in (0,\infty].
\end{align}

\item There exists some $R\in (0,1]$ such that each $m_j$ restricted to $(-R,R)$ is absolutely continuous with respect to the Lebesgue measure. In particular, the density $\rho_j=dm_j/d\xi$ is in $\mathcal{C}^1\brac{(-R,R)\setminus \{0 \}}$. 

Moreover, there exists a constant $\kappa>1$ such that for all $j$,
\begin{align}
   M_j(\kappa)&:=   \int_{-R}^R\abs{\xi}^\kappa\rho_j(\xi)d\xi<\infty, \label{cond_intkappa}\\
    M_j(2\kappa-2)&:= \int_{-R}^R\abs{\xi}^{2\kappa-2}\rho_j(\xi)d\xi<\infty,\label{cond_int_2kappaminus2}\\
      \overline{M}_j(2\kappa)&:=\int_{-R}^R\abs{\xi}^{2\kappa}\brac{\frac{\rho_j'(\xi)}{\rho_j(\xi)}}^2\rho_j(\xi)d\xi<\infty.\label{cond_inttwokappa}
\end{align}

There is also a constant $\tau>2$ such that for all $j$,
\begin{align}
      \overline{M}_j(\tau \kappa):=  \int_{-R}^R\abs{\xi}^{\tau\kappa}\brac{\frac{\rho_j'(\xi)}{\rho_j(\xi)}}^\tau\rho_j(\xi)d\xi<\infty.\label{cond_inttaukappa}
\end{align}

\item There exists some $\Lambda>\frac{\tau}{\tau-1}$ such that for all $j$ and for any $\lambda$ in $[1,\Lambda)$,
\begin{align}
\label{cond_moment_largejump}
 \widetilde{M}_j(\lambda) := \int_{\abs{\xi}> R}\abs{\xi}^\lambda\rho_j(\xi)d\xi&<\infty. 
\end{align}
\end{enumerate}
	\end{customcondition}

\begin{customcondition}{H2}
		\label{cond_driftb} $b\in \mathcal{C}^2$ and there exist $\theta_1>0, \theta_2\geq 0,\theta_3\geq 0$ such that
\begin{align}
\label{originalconddirftb}
   \sup_{x\in\mathbb{R}^{d}} \abs{\nabla b(x)}\leq \theta_1,\qquad \sup_{x\in\mathbb{R}^{d}} \abs{\nabla^2 b(x)}\leq \theta_2,
\qquad
    \sup_{x\in\mathbb{R}^{d}} \abs{\nabla^3 b(x)}\leq \theta_3. 
\end{align}
\end{customcondition}

Next, let us denote the transition semigroups associated with \eqref{sde_cylindricallevy} by $\{P_t:t\geq 0\}$, that is $   P_tf(x)=\E{f(X^x_t)}$. Let $x\in \R^{d}$. $\delta_x P_t$ denotes the law of $X_t$ given $X_0=x$ a.s., where $\delta_{x}$ is the Dirac delta function. We say the semigroups have a \textit{Wasserstein decay rate} $\mathcal{R}:\R_{\geq 0}\to\R_{\geq 0}$ if 
\begin{align}\label{R:t:eqn}
d_{\operatorname{Wass}}\brac{\delta_xP_t,\delta_yP_t}\leq \mathcal{R}(t) d_{\operatorname{Wass}}\brac{\delta_x,\delta_y}, 
\end{align}
for every $x,y\in\R^d$ and $t\geq 0$.

We make the following assumption on the semigroups.     

\begin{customcondition}{H3} 
\label{cond_wassersteindecay}
$\{P_t:t\geq 0\}$ has a Wasserstein decay rate $\mathcal{R}$ that is a non-increasing and integrable function on $\R_{\geq 0}$.

	\end{customcondition}

\begin{remark}
    The class of cylindrical L\'{e}vy processes under consideration includes the important case that is cylindrical $\alpha$-stable processes (see Remark~\ref{remark_stablenoisebelongtoourclass}). 
    
\end{remark}
\begin{remark}
    One can see the combination of \eqref{cond_pi}, \eqref{cond_intkappa}, \eqref{cond_int_2kappaminus2}, \eqref{cond_inttwokappa} in Condition~\ref{cond_levymeasure} and \eqref{originalconddirftb} in Condition~\ref{cond_driftb} is the original assumption in \cite[Theorem 1]{kulik2023gradient}. Our assumption is therefore more restrictive than theirs, due to the appearance of the new constant $\tau>2$.  There is a technical reason why we need $\tau$ in this paper and further details are provided in Remark~\ref{remark_needtaugreaterthan2}. Another new feature is the presence of the constant $\Lambda$ in Condition~\ref{cond_levymeasure}, which will play a role in the proof of Proposition~\ref{prop_semigroupestimate}. 
    
\end{remark}

\begin{remark}
\label{remark_smalljump}
     Since $0<R\leq 1$, it holds for any $0<a\leq b$, we have:
\begin{align*}
    M_j(a)=\int_{-R}^R \abs{\xi_k}^a \rho_k(\xi_k)d\xi_k  \geq \int_{-R}^R \abs{\xi_k}^b \rho_k(\xi_k)d\xi_k=M_j(b).
\end{align*}
\end{remark}

We are now ready to state the gradient estimates for the semigroups associated with the solution to \eqref{sde_cylindricallevy}. Recall the operator norm of a linear map $T:\mathbb{R}^{d}\to\mathbb{R}^{d}$ is defined as 
\begin{align*}
\norm{T}_{\operatorname{op}}:=\sup_{v\in\mathbb{R}^{d}:\norm{v}=1}\norm{Tv}. 
\end{align*}

\begin{proposition}
\label{prop_semigroupestimate}
    Assume that the SDE \eqref{sde_cylindricallevy} satisfies Conditions~\ref{cond_levymeasure},~\ref{cond_driftb} and~\ref{cond_wassersteindecay}. Then for all $t\geq 0$ and any Lipschitz function $f$, we have for any $u,v\in \R^d$
\begin{align}
\label{estimate_gradsemigroup_prop}
   \sup_{x\in\R^d} \abs{\nabla_u P_tf(x)}&:=\sup_{x\in\R^d} \abs{\nabla P_tf(x)u}\nonumber\\
   &\leq \brac{\sup_{y\in\mathbb{R}^{d}} \norm{\nabla f(y)}_{\operatorname{op}}}\abs{u}\mathcal{R}(t),
\end{align}
where $\mathcal{R}(t)$ is given in \eqref{R:t:eqn} and 
   \begin{align}
   \label{estimate_secondgradientsemigroup_prop}
  \sup_{x\in\R^d}  \abs{\nabla_v\nabla_u P_tf(x)}&:=\sup_{x\in\R^d} \abs{v^T\nabla^2 P_tf(x)u }\nonumber\\
  &\leq C_2\abs{u}\abs{v}\brac{\sup_{y\in\mathbb{R}^{d}}\norm{\nabla f(y)}_{\operatorname{op}}}\phi(t),
\end{align}
where 
\begin{equation}\label{phi:t:eqn}
\phi(t)=\begin{cases} 
      1 &\text{if } 0\leq t\leq 1, \\
     \mathcal{R}(t-1) &\text{if } t>1,
   \end{cases}
\end{equation}
and the constant $C_2$ is defined in \eqref{def_constantC2}. Dimension dependence of $C_2$ is spelled out in Remark \ref{remark_C_2dimensiondependence}. 
\end{proposition}

For our first application of Proposition~\ref{prop_semigroupestimate}, we provide the following upper bound on the Wasserstein distance between the law of a generic random variable $F$ and $\nu$. The proof is based on Stein's method and will be presented in Section~\ref{section_steinmethod}.

\begin{theorem}
\label{theorem_steinmethodgeneralbound} (Stein's bound)
In addition to Conditions~\ref{cond_levymeasure},\ref{cond_driftb} and~\ref{cond_wassersteindecay}, let us assume that $\nu$ is the unique invariant measure  of  \eqref{sde_cylindricallevy} on $(\R^d,\mathcal{B}({\R^d})$ where $\mathcal{B}({\R^d})$ is the Borel sets in $\R^d$. Moreover, assume $\E{\abs{W}}<\infty$, where $W$ is a random variable distributed as $\nu$. 
Then it holds that
    \begin{align*}
        d_{\operatorname{Wass}}(\operatorname{Law}(F),\nu)\leq \sup_{h\in \mathbb{H}}\abs{\E{\mathcal{L}h(F) }},
    \end{align*}
    where $F$ denotes a generic random variable taking value in $\R^d$ and $h\in \mathbb{H}$ are twice differentiable functions that satisfy
    \begin{align*}
     \sup_{x\in \R^d}   \norm{\nabla h(x)}_{\operatorname{op}}&\leq \int_0^\infty \mathcal{R}(t)dt,
    \end{align*}
where $\mathcal{R}(t)$ is given in \eqref{R:t:eqn} and 
\begin{align*}
     \sup_{x\in \R^d}    \norm{\nabla^2h(x)}_{\operatorname{op}}&\leq C_2\brac{1+\int_0^\infty \mathcal{R}(t)dt},
    \end{align*}
where $C_2$ is the constant defined at \eqref{def_constantC2}. 
\end{theorem}

In Theorem~\ref{theorem_steinmethodgeneralbound}, 
we assumed the existence of an unique invariant measure $\nu$ of \eqref{sde_cylindricallevy}. Note that in the next section where the
cylindrical L\'{e}vy process in \eqref{sde_cylindricallevy} is  an $\alpha$-stable process, we will provide explicit conditions so that this assumption is satisfied.


\section{Euler-Maruyama scheme for SDEs driven by $\alpha$-stable L\'{e}vy process}\label{section_eulerscheme}

In this section, we specialize the  cylindrical L\'{e}vy process considered in the previous section to be an $\alpha$-stable L\'{e}vy process $\{L^\alpha_t:t\geq 0\}$ with $1< \alpha< 2$, where the components are i.i.d. $\alpha$-stable and are equipped with the L\'{e}vy measure ${p_\alpha}/{\abs{z}^{\alpha+1}}$, where $p_\alpha$ is the constant defined as:
\begin{equation}\label{p:alpha}
p_{\alpha}:=\alpha 2^{\alpha-1}\pi^{-1/2}\Gamma(\alpha/2+1/2)\Gamma(1-\alpha/2)^{-1}, 
\end{equation}
(see \cite[Example 2.4d]{schillingbottcher2013levy}).
The SDE in \eqref{sde_cylindricallevy} then becomes
\begin{align}
\label{SDE_alphastable}
X_t&=b(X_t)dt+ dL^\alpha_t,\quad X_0=x. 
\end{align}
We propose the following Euler-Maruyama schemes of the above equation.
\begin{align}
\label{equation_discreteY}
Y_{m+1}&=Y_m+\eta b(Y_m)+\xi_m, \quad Y_0=x, 
\end{align}
where $\{\xi_m:m\geq 1\}$ are increments of the stable process $L^\alpha$, i.e.
$\xi_m:= L^\alpha_{(m+1)\eta}-L^\alpha_{m\eta}$. 



As the second and main application of Proposition~\ref{prop_semigroupestimate}, we will derive estimates on the convergence rates of the Euler-Maruyama scheme in \eqref{equation_discreteY}.   

Let us denote the $j$-th component of $L^\alpha$ by $L^{\alpha,j}$; then its L\'{e}vy measure is defined as:
\begin{align*}
    m_j(B)=p_\alpha\int_{\R} \mathds{1}_{B}(x)\frac{1}{\abs{x}^{\alpha+1}}dx,
    \qquad\text{for any Borel set $B$ in $\mathbb{R}$}.
\end{align*}

\begin{remark}
\label{remark_stablenoisebelongtoourclass}
     To see that $m_j$ above  satisfies Condition~\ref{cond_levymeasure}, one can take $R=1$, $\pi=\Lambda=\alpha$ and let $\kappa$ be any number satisfying $\kappa>1+\frac{\alpha}{2}$, and $\tau$ be any number  satisfying $\tau>\max \left\{\frac{\alpha}{\alpha-1}, 2 \right\}$. 

Indeed, \cite[Remark 3 and Lemma 8]{kulik2023gradient} already pointed out in the case of a cylindrical $\alpha$-stable process, the assumptions \eqref{cond_pi}, \eqref{cond_intkappa}, \eqref{cond_int_2kappaminus2}, \eqref{cond_inttwokappa} in Condition~\ref{cond_levymeasure} are satisfied by letting $R=1,\pi=\alpha$ and $\kappa>1+\frac{\alpha}{2}$. What remains is to verify that \eqref{cond_inttaukappa} and \eqref{cond_moment_largejump} are satisfied by the previous choice of $R,\pi,\kappa$ and additionally $\Lambda=\alpha, \tau>\max \left\{\frac{\alpha}{\alpha-1}, 2 \right\}$. The fact that $\kappa>1+\frac{\alpha}{2}$ means $\frac{2}{\alpha}\geq \frac{1}{\kappa-1}$ and $2>\frac{\alpha}{\kappa-1}$. This implies $\tau\geq \frac{\alpha}{\kappa-1}$ and $\tau(\kappa-1)-\alpha>0$. Therefore, we deduce that
\begin{align*}
    \overline{M}_j(\tau \kappa)=\int_{-1}^1 \abs{\xi}^{\tau(\kappa-1)-\alpha-1}d\xi<\infty. 
\end{align*}
Finally, let us check \eqref{cond_moment_largejump} with $\Lambda=\alpha$. Assume $\lambda\in [1,\alpha)$ then $\lambda-\alpha-1<-1$ and hence
\begin{align*}
     \widetilde{M}_j(\lambda) =\int_{\abs{\xi}>1}\abs{\xi}^{\lambda-\alpha-1}d\xi<\infty. 
\end{align*}
\end{remark}

\begin{remark}
We can perform an exact simulation 
of the Euler-Maruyama scheme \eqref{equation_discreteY}.
To see this, notice that we can simulate multivariate stable distribution with i.i.d. components via exact simulation of one-dimensional stable distribution as follows.
Let us recall that an alpha stable random variable $X$ has the characteristic function
\begin{equation*}
\mathbb{E}[e^{itX}]=\exp\left(it\mu-|ct|^{\alpha}(1-i\beta\text{sgn}(t)\Phi)\right),
\end{equation*}
where $\Phi=\tan(\pi\alpha/2)$ when $\alpha\neq 1$ and $\Phi=-\frac{2}{\pi}\log|t|$ when $\alpha=1$.
Here $\alpha\in(0,2]$ is the stability parameter, $\beta\in[-1,1]$ is the skewness parameter, 
$c\in(0,\infty)$ is the scale parameter and $\mu\in(-\infty,\infty)$ is the location parameter. In order to simulate $X$, we first generate a random variable $U$ uniformly distributed on $(-\frac{\pi}{2},\frac{\pi}{2})$
and an independent random variable $W$ with mean $1$ (which can also be exactly simulated).
Then, for $\alpha\neq 1$, we compute
\begin{equation*}
X=(1+\zeta^{2})\frac{1}{2\alpha}\frac{\sin(\alpha(U+\xi))}{(\cos(U))^{\frac{1}{\alpha}}}\left(\frac{\cos(U-\alpha(U+\xi))}{W}\right)^{\frac{1-\alpha}{\alpha}},
\end{equation*}
and for $\alpha=1$, we compute 
\begin{equation*}
X=\frac{1}{\xi}\left(\left(\frac{\beta}{2}+\beta U\right)\tan U-\beta\log\left(\frac{\frac{\pi}{2}W\cos U}{\frac{\beta}{2}+\beta U}\right)\right),
\end{equation*}
where $\zeta=-\beta\tan(\pi\alpha/2)$ and $\xi=\frac{1}{\alpha}\arctan(-\zeta)$ for $\alpha\neq 1$ and $\xi=\pi/2$ for $\alpha=1$. This simulation method was first proposed in \cite{Chambers1976} and is known as the CMS method in the literature. 
Thus, one can simulate a one-dimensional alpha stable
distribution, and hence a vector of i.i.d. components of one-dimensional alpha stable
distributions that includes $\xi_{m}$ in 
the Euler-Maruyama scheme \eqref{equation_discreteY}.
\end{remark}

For the current section, we also need the following assumption. 
\begin{customcondition}{H3*}~
\label{cond_diss} for all $x,y\in\R$, there exists constants $\theta_4>0$ and $K\geq 0$ such that the drift coefficient $b$ satisfies
     \begin{align*}
 \inner{b(x)-b(y),x-y}\leq -\theta_4\abs{x-y}^2+K. 
\end{align*}
\end{customcondition}

\begin{remark}\label{remark:dissipativity}
    Condition~\ref{cond_driftb} implies for all $x,y\in \R^d$, 
    \begin{align*}
        \inner{b(x)-b(y),x-y}\leq \theta_1\abs{x-y}^2.
    \end{align*}
Meanwhile, Condition~\ref{cond_diss} implies for $\abs{x-y}^2\geq \frac{2K}{\theta_4}$,
\begin{align*}
    \inner{b(x)-b(y),x-y}\leq -\frac{\theta_4}{2}\abs{x-y}^2.
\end{align*}
Together, Condition~\ref{cond_driftb} and Condition~\ref{cond_diss} imply
\begin{align*}
 \inner{b(x)-b(y),x-y}\leq  \begin{cases} 
      \theta_1\abs{x-y}^2 &\text{ if } \abs{x-y}\leq L_0,\\
      -\frac{\theta_4}{2}\abs{x-y}^{2} &\text{ if }\abs{x-y}> L_0,
   \end{cases}
\end{align*}
where $L_0:= \sqrt{\frac{2K}{\theta_4}}$. The above expression is known in literature as \textit{distant dissipativity} condition. It implies Condition~\ref{cond_wassersteindecay} as well as existence of unique invariant measures associated with \eqref{SDE_alphastable} and \eqref{equation_discreteY}. These results will be shown in the upcoming lemmas.
\end{remark}

Let $\{e^i:1\leq i\leq d\}$ be the canonical basis of $\R^d$, i.e. $e^{i}$ is a $d$-dimensional vector with $1$ in its $i$-th coordinate and $0$ elsewhere. For $z=\brac{z_1,\ldots,z_d}\in\R^d$, we write
\begin{align*}
    z^i:=z_ie^i. 
\end{align*}

We define the fractional Laplacian operator as:
\begin{align}\label{L:0:operator:stable}
  L_0h(x)
  =\Delta^{\alpha/2}h(x) = \sum_{i=1}^d p_\alpha\int_{\R}\brac{h(x+z^i)-h(x)-\inner{\nabla h(x),z^i}\mathds{1}_{\{\abs{z_i}\leq 1\}} }\frac{1}{\abs{z_i}^{1+\alpha}}dz_i,
\end{align}
where $p_{\alpha}$ is defined in \eqref{p:alpha}.

Let us also define
\begin{align}
\label{def_Vlambda}
    V_{\lambda }(x):=\brac{1+\abs{x}^2}^{\lambda /2},
\end{align}
where $\lambda \in (1,\Lambda\wedge\kappa)$. The upcoming results state that under Conditions~\ref{cond_driftb} and~\ref{cond_diss}, $\{X_t:t\geq 0\}$ at \eqref{SDE_alphastable} and its Euler-Maruyama discretization $\{Y_k:k\geq 0\}$, $\{U_k:k\geq 0\}$ are ergodic. Their proofs are very similar to the proofs of analogous results in \cite{chenxu2023euler}, and are therefore relegated to the Appendix~\ref{section_proofexistenceofinvariantmeasure} .

\begin{lemma}
\label{lemma_invariantmeasure_continuousX}
    Assume Conditions~\ref{cond_driftb} and~\ref{cond_diss} hold for \eqref{SDE_alphastable}. Then any solution to \eqref{SDE_alphastable} admits an unique invariant measure $\nu$. Moreover for $1\leq \lambda<\Lambda$, there exist constants $C,C'>0$ such that
    \begin{align}
    \label{estimate_exponentialerg_X}
        \sup_{\abs{f}\leq V_\lambda} \abs{\E{f(X^x_t)}-\E{f(W)}}\leq CV_\lambda(x)e^{-C't}. 
    \end{align}

In addition, we have the moment estimate (uniform over $t\geq 0$)
\begin{align*}
    \E{\abs{X^x_t}^\lambda}\leq \E{V_\lambda\brac{X_t^x}}\leq C_3(\lambda)\brac{1+\abs{x}^2}^{\lambda/2},
\end{align*}
where
\begin{align*}
    C_3(\lambda):=\frac{2}{\theta_4}\Bigg(\lambda(\theta_4+K)+\theta_4^{1-\lambda}\abs{b(0)}^\lambda +\frac{2p_\alpha \lambda(3-\lambda)\sqrt{d}}{2(2-\alpha)}+\frac{2p_\alpha\lambda }{\alpha-\lambda}+\brac{\frac{\theta_4}{4}}^{1-\lambda} \brac{\frac{2p_\alpha}{\alpha-1}}^\lambda\Bigg)+1. 
\end{align*}
\end{lemma}

The following lemma establishes exponential ergodicity for the Euler-Maruyama discretiation scheme \eqref{equation_discreteY}.

\begin{lemma}
\label{lemma_invariantmeasure_discreteYk}
Assume Conditions~\ref{cond_driftb} and~\ref{cond_diss} hold. The Markov chain $\{Y_k:k\in\N\}$  admits a unique invariant measure $\nu_\eta$. Moreover, there exist constants $C,C'>0$ such that
  \begin{align}
  \label{estimate_exponentialerg_Y}
        \sup_{\abs{f}\leq V_1} \abs{\E{f(Y^x_k)}-\mathbb{E}_{X\sim \nu_\eta}\bracsq{f(X)}}\leq CV_1(x)e^{-C'k}. 
    \end{align}
In addition for $ 1\leq \lambda<\Lambda$, we have the moment estimate (uniform over $k\geq 0$) 
\begin{align*}
    \E{\abs{Y_k^x}^\lambda}\leq \E{V_\lambda\brac{Y_k^x}}\leq  C_4(\lambda)\left(1+\abs{x}^2\right)^{\lambda/2}, 
\end{align*}
where 
\begin{align*}
&C_4(\lambda):=1+\frac{2}{\theta_4}\Bigg[\frac{\theta_4\lambda}{2}\brac{\eta\frac{2\abs{b(0)}^2}{\theta_4}+2\eta^2\abs{b(0)}^2+1+2\eta K}+\frac{\lambda \abs{b(0)}^2}{\theta_4}\\
&\quad+2\lambda\eta\abs{b(0)}^2+\lambda K+2\lambda p_\alpha\brac{\frac{(3-\alpha)\sqrt{d}}{2(2-\alpha)}+\frac{1}{\alpha-\lambda} +\abs{b(0)}^{\lambda-1}+\frac{\E{\abs{L^\alpha_1}^{\lambda-1}}}{\alpha-1}}
\\
&\qquad\qquad\qquad\qquad\qquad\qquad\qquad\qquad\qquad\quad
+\brac{\frac{2p_\alpha(1+\theta_1^{\lambda-1})}{\alpha-1}}^\lambda\brac{\frac{2}{\theta_4}}^{\lambda-1}\Bigg]. 
\end{align*}
\end{lemma}

The next lemma provides Wasserstein decay rate of Equation \eqref{SDE_alphastable} which is our SDE driven by an $\alpha$-stable L\'{e}vy process with i.i.d. components. It is similar to the main theorem in \cite{jianwangwassersteindecay}, which considers SDE driven by a rotationally invariant $\alpha$-stable process. Apart from some minor technical differences, the proof of the next Lemma largely follows the idea of the aforementioned paper. The proof is somewhat long and is therefore relegated to the Appendix~\ref{section_proofresultbyjianwang} .

\begin{lemma}
\label{lemma_basedonjianwangpaper}
   Under Conditions~\ref{cond_driftb} and~\ref{cond_diss}, for every $x,y\in\R^d$ and $t\geq 0$, it holds that
   \begin{align*}
       d_{\operatorname{Wass}}(\operatorname{Law}(X^x_t),\operatorname{Law}(X^y_t))\leq \frac{2\brac{1-e^{-c_1L_0}}}{L_0}e^{-C_5t}\abs{x-y}, 
   \end{align*}
where 
\begin{align*}
    C_5:=- e^{-2c_1\sqrt{\frac{2K}{\theta_4}}} \min \Bigg\{  2\theta_1,\frac{\theta_4}{2}\brac{\frac{2K}{\theta_4}}^{\theta_4/2-1},\frac{c_1}{8\sqrt{2}} \brac{\frac{e^{-2c_1\sqrt{\frac{2K}{\theta_4}}}}{20}+1}\frac{\theta_4^{3/2}}{K^{1/2}} \brac{\frac{2K}{\theta_4}}^{\theta_4/2-1} \Bigg\},
\end{align*}
and 
\begin{align*}
    c_1:=\brac{\frac{\theta_1(2-\alpha)}{4p_\alpha}\brac{\frac{\theta_4}{2K}}^{\frac{1-\alpha}{2}}e^{-2\sqrt{\frac{2K}{\theta_4}}} }^{\frac{1}{\alpha-1}}. 
\end{align*}
\end{lemma}

Now, we are ready to state the main result of this section.
The following results provide the convergence rates of the Euler-Maruyama scheme. The proof is deferred to Section~\ref{section_proof_eulerscheme}.

\begin{theorem}
\label{theorem_eulerscheme}
Assume the stepsize $\eta$ satisfies $\eta\leq \min \left\{ 1,\frac{\theta_4}{8
\theta_1^2},\frac{1}{\theta_4}\right\} $. Then it holds that 
\begin{align}
\label{estimate_distancebetweenXandY}
d_{\operatorname{Wass}}(\operatorname{Law}(X_{\eta N}),\operatorname{Law}(Y_N))
\leq\mathcal{C}\eta,
\end{align}
where
\begin{align}
\mathcal{C}&:=\brac{1+\frac{2\brac{1-e^{-c_1L_0}}}{L_0C_5}+\frac{2C_2\brac{1-e^{-c_1L_0}}}{L_0C_5}+C_2}\nonumber
    \\&\quad\cdot\bigg(\brac{3\theta_1^2+\frac{4\theta_2d p_\alpha}{(2-\alpha)(\alpha-1)}} C_3(1)C_4(1)\brac{1+\abs{x}^2}^{1/2}+2\theta_1\E{\abs{L^\alpha_1}}+\abs{\Delta^{\alpha/2}b(0)}\bigg).\label{defn:constant:C}
\end{align}
This leads to 
\begin{align}
\label{estimate_distancebetweeninvariantmeasures}
d_{\operatorname{Wass}}\brac{\nu_\eta,\nu}
\leq\mathcal{C}\eta.
\end{align}
\end{theorem}


\begin{remark}
The uniform-in-time Euler-Maruyama approximation bounds in Theorem~\ref{theorem_eulerscheme} can be directly applied to many settings in machine learning,
such as fractional Langevin Monte Carlo algorithms \cite{simcsekli2017fractional,nguyen19} 
where a rigorous uniform-in-time approximation
analysis is lacking, and our results help bridge a gap between theory and practice.
\end{remark}

\begin{remark}
In Theorem~\ref{theorem_eulerscheme}, our approximation error bound has a linear dependence
on the stepsize $\eta$, and we expect it to be tight as can be seen through an explicit calculation
in the case of the Ornstein-Uhlenbeck process in Appendix~\ref{sec:OU} . 
\end{remark}

\begin{remark}
We will present the proof of Theorem~\ref{theorem_eulerscheme} in Section~\ref{section_proof_eulerscheme} which will follow the strategy in \cite{chenxu2023euler} and employ the classical Linderberg's principle plus the semigroup gradient estimates that is obtained in Proposition~\ref{prop_semigroupestimate}.

A natural question that one might come up is why our strategy to prove Theorem~\ref{theorem_eulerscheme}  is not via the Stein's bound proposed in Theorem~\ref{theorem_steinmethodgeneralbound}. The answer is that one could use Theorem~\ref{theorem_steinmethodgeneralbound} to deduce an upper bound like \eqref{estimate_distancebetweeninvariantmeasures}. In fact, in the case of stochastic differential equations with Brownian noise, Stein's method is applied to derive a similar result to \eqref{estimate_distancebetweeninvariantmeasures} in \cite[Theorem 4.1]{FSX19}. However, it is ultimately not clear to us whether Stein's method can lead to a bound such as \eqref{estimate_distancebetweenXandY}, and thus we follow the strategy in \cite{chenxu2023euler}.
\end{remark}


\section{Proof of Proposition~\ref{prop_semigroupestimate}}\label{sec:proofs:1}

In this section, we present the proof of Proposition~\ref{prop_semigroupestimate}, 
which provides semigroup gradient estimates for the semigroups associated with the solution to \eqref{sde_cylindricallevy}. We first present a technical lemma. 

\begin{lemma}
\label{lemma_boundB_tx}
For any Lipschitz function $f$, let:
    \begin{align*}
        \mathcal{B}_{t,x}(f):=\int_0^t\int_{\R^d} \left(P_{t-s}f(X^x_{s-}+\xi)-P_{t-s}f(X^x_{s-})\right)\widehat{N}(ds,d\xi),
    \end{align*}
    where $\widehat{N}$ is the compensated Poisson measure defined in Appendix~\ref{appendix_malliavin} .
    Then for any $0\leq t\leq 1$ and $1\leq \lambda<\Lambda$ with $\Lambda$ defined in Condition \eqref{cond_moment_largejump}, we have
\begin{align}
 \sup_{x\in\R^d}   \E{\abs{\mathcal{B}_{t,x}(f)}^\lambda}\leq C_1(\lambda) \brac{\sup_{y\in\mathbb{R}^{d}} \norm{\nabla f(y)}_{\operatorname{op}}}^\lambda,\label{holds:with:C1}
\end{align}
where the factor $C_{1}(\lambda)$ is defined as:
\begin{align}
\label{def_constantC1}
    C_1(\lambda):=\sup_{t\in[0,1]} \E{\abs{ \int_0^t\int_{\R^d}\abs{\xi}\widehat{N}
       (ds,d\xi)}^\lambda}.
\end{align}
\end{lemma}

\begin{proof}
We can compute that
    \begin{align*}
    \E{\abs{\mathcal{B}_{t,x}(f)}^\lambda}&=\E{ \abs{\int_0^t\int_{\R^d} \left(P_{t-s}f(X^x_{s-}+\xi)-P_{t-s}f(X^x_{s-})\right)\widehat{N}(ds,d\xi)}^\lambda}\\
   &\leq \E{\abs{ \int_0^t\int_{\R^d} \sup_{y\in\mathbb{R}^{d}}\norm{\nabla P_{t-s}f(y)}_{\operatorname{op}}\abs{\xi}\widehat{N}(ds,d\xi)}^\lambda}\\
       &\leq \E{\abs{ \int_0^t\int_{\R^d} {\brac{\sup_{y\in\mathbb{R}^{d}} \norm{\nabla f(y)}_{\operatorname{op}}}}e^{-\theta_4(t-s)}\abs{\xi}\widehat{N}
       (ds,d\xi)}^\lambda}\\
   &\leq \brac{{\sup_{y\in\mathbb{R}^{d}} \norm{\nabla f(y)}_{\operatorname{op}}}}^\lambda\E{\abs{ \int_0^t\int_{\R^d}\abs{\xi}\widehat{N}
       (ds,d\xi)}^\lambda}.
    \end{align*}
Recall $\rho$ is the intensity measure of $N$. Let $\psi(x)=\abs{x}$. Per \cite[Theorem 1.2.14, Part 1) of Theorem 2.3.7 and (2.9)]{applebaum2009levy}, the L\'{e}vy measure associated with the Poisson integral $\int_0^t\int_{\R^d}\abs{\xi}\widehat{N}
       (s,d\xi)$  has the form $t\rho_{\psi}(\cdot)$ where $\rho_{\psi}(A):=\rho(\psi^{-1}(A\cap [0,\infty)^{\otimes d}))$. Then due to  \eqref{cond_moment_largejump} in Condition~\ref{cond_levymeasure} which guarantees
$\int_{\abs{\xi}>R} \abs{\xi}^\lambda \rho_\psi(d\xi)<\infty$
and \cite[Theorem 2.5.2]{applebaum2009levy}, the Poisson integral above is indeed finite for all $t\geq 0$. Hence \eqref{holds:with:C1} holds with the constant
$C_1(\lambda)$ that is defined in \eqref{def_constantC1}. 
This completes the proof.
\end{proof}


Now, we are ready to prove Proposition~\ref{prop_semigroupestimate}.

\begin{proof}[Proof of Proposition~\ref{prop_semigroupestimate}]
For any Lipschitz function $f$, we have
\begin{align*}
    \left|P_tf(x)-P_tf(y)\right|&=\left|\E{f(X^x_t)-f(X^y_t)}\right|\\
    &\leq \brac{\sup_{y\in\mathbb{R}^{d}} \norm{\nabla f(y)}_{\operatorname{op}}} d_{\operatorname{Wass}}\brac{\delta_xP_t,\delta_yP_t}\\
    &\leq \brac{\sup_{y\in\mathbb{R}^{d}} \norm{\nabla f(y)}_{\operatorname{op}}} \mathcal{R}(t)d_{\operatorname{Wass}}\brac{\delta_x,\delta_y}\\
    &=\brac{\sup_{y\in\mathbb{R}^{d}} \norm{\nabla f(y)}_{\operatorname{op}}}\mathcal{R}(t)\abs{x-y}, 
\end{align*}
where we applied \eqref{R:t:eqn}, which gives the estimate \eqref{estimate_gradsemigroup_prop}.

Next,  per the Bismut-Elworthy-Li formula in Proposition~\ref{prop_bismut}, there exists a random vector $G$ such that
\begin{align*}
   \nabla_u P_tf(x)= \nabla_u \E{f(X^x_t)}=\E{f(X^x_t) \inner{G(x,t),u}}.
\end{align*}

By interchanging the derivative and the expectation, we get
\begin{align}
\label{step_secondgradientsemigroup}
    \nabla_v \nabla_u P_tf(x)&= \E{\nabla_v\brac{f(X^x_t) \inner{G(x,t),u}}}\nonumber\\
    &=\E{\nabla_v f(X_t^x) \inner{G(x,t),u}}+\E{f(X_t^x)v^T\nabla G(x,t)u}. 
\end{align}
To see that $\E{f(X^x_t) \inner{G(x,t),u}}$ is differentiable with respect to $x$ and that the interchange of derivative and expected value in \eqref{step_secondgradientsemigroup} is valid, we will rely on \cite[Theorem 16.8]{Billingsleyprobandmeasure} and the paragraph following the proof of the aforementioned theorem. By Condition~\ref{cond_driftb} and Lipschitz continuity of $f$, $f(X^x_t)$
and $ \nabla_v f(X_t^x)$ are continuous in $x $ almost surely. Furthermore by Lemma~\ref{lemma_GandnablaG}, $G(x,t)$ and $\nabla G(x,t)$ are differentiable and hence continuous in $x$. These facts imply the integrand in \eqref{step_secondgradientsemigroup}, which is given by
\begin{align*}
    J_x=\nabla_v f(X_t^x) \inner{G(x,t),u}+f(X_t^x)v^T\nabla G(x,t)u,
\end{align*}
are continuous in $x$. Let $K_x$ be a compact ball around $x$. The continuity implies that there exists $x_0\in K_x$ such that 
\begin{align*}
\sup_{y\in K_x} J_y=J_{x_0}. 
\end{align*}
We further claim that $J_x$ is integrable for every $x\in\R^d$, i.e.
\begin{align}
\label{step_Jxintegrable}
    \E{\abs{J_x}}<\infty.  
\end{align}
The proof of \eqref{step_Jxintegrable} will be postponed to the end. Then the aforementioned result from \cite{Billingsleyprobandmeasure} applies and it is possible to interchange the derivative and expected value in \eqref{step_secondgradientsemigroup}. 


Next, let us get back to Equation~\eqref{step_secondgradientsemigroup}. Regarding the first term on the right hand side of \eqref{step_secondgradientsemigroup}, we have
\begin{align*}
    \abs{\E{\nabla_v f(X_t^x) \inner{G(x,t),u}}}\leq \abs{u}\abs{v}\E{\abs{G(x,t)}}\sup_{y\in\mathbb{R}^{d}}\norm{\nabla f(y)}_{\operatorname{op}}.
\end{align*}
Now we study the second term on the right hand side of \eqref{step_secondgradientsemigroup}. An application of It\^{o}'s formula \cite[Theorem 4.4.7]{applebaum2009levy} to the function $(s,x)\mapsto P_{t-s}f(x)$ gives
\begin{align}
\label{equation_itoformulaapplied}
    f(X^x_t)&=P_tf(x)+\int_0^t\int_{\R^d} \left(P_{t-s}f(X^x_{s-}+\xi)-P_{t-s}f(X^x_{s-})\right)\widehat{N}(ds,d\xi)\nonumber\\
    &=P_tf(x)+\mathcal{B}_{t,x}(f).
\end{align}
Notice in particular that we can apply It\^{o}'s formula \cite[Theorem 4.4.7]{applebaum2009levy} as long as for any fixed $t$ and any Lipschitz function $f$, $P_tf(x)$ is twice continuously differentiable with respect to $x$. The latter is implied by \eqref{step_secondgradientsemigroup} and the paragraph following it.

This leads to
\begin{align*}
    \abs{\E{f(X_t^x)v^T\nabla G(x,t)u}}&= P_tf(x)v^T\E{\nabla G(x,t)}u+\E{\mathcal{B}_{t,x}(f)v^T\nabla G(x,t)u}\\
    &=\E{\mathcal{B}_{t,x}(f)v^T\nabla G(x,t)u},
\end{align*}
where we used $\E{\nabla G(x,t)}=0$ which is obtained by substituting $f\equiv 1$ into Equation \eqref{step_secondgradientsemigroup}
to obtain the last equality above. 

Next, \eqref{cond_moment_largejump} in Condition~\ref{cond_levymeasure} allows us to choose $\lambda$ which satisfies $\tau/(\tau-1)< \lambda<\Lambda $. Then
\begin{align}
\label{step_needtaugreaterthan2}
    \abs{\E{f(X_t^x)v^T\nabla G(x,t)u}}\leq \abs{u}\abs{v} \E{\abs{\mathcal{B}_{t,x}(f)}^\lambda}^{1/\lambda} \E{\abs{\nabla G(x,t)}^{\frac{\lambda}{\lambda-1}}}^{\frac{\lambda-1}{\lambda}}. 
\end{align}
By Lemma~\ref{lemma_GandnablaG}, $G(x,t)$ and $\nabla G(x,t)$ are $q$-integrable for $\tau>q\geq 1$ and our choice of $\lambda$ ensures $\tau> \lambda/(\lambda-1)$. The previous facts and Lemma~\ref{lemma_boundB_tx} allows us to deduce from Equation \eqref{step_secondgradientsemigroup} that for every $t\leq 1$,
\begin{align*}
    \abs{\nabla_v \nabla_u P_tf(x)}\leq \abs{u}\abs{v}C_2\sup_{y\in\mathbb{R}^{d}}\norm{\nabla f(y)}_{\operatorname{op}},
\end{align*}
where $C_2$ is a constant that is chosen large enough such that
\begin{align}
\label{def_constantC2}
  C_2\geq  \sup_{t\in [0,1],x\in\R^d} \E{\abs{G(x,t)}}+\sup_{t\in [0,1],x\in\R^d}C_1(\lambda)^{1/\lambda}  \E{\abs{\nabla G(x,t)}^{\frac{\lambda}{\lambda-1}}}^{\frac{\lambda-1}{\lambda}},
\end{align}
for some $\lambda$ satisfying $\frac{\tau}{\tau-1}< \lambda<\Lambda $, where the factor $C_1(\lambda)$ is defined in Lemma~\ref{lemma_boundB_tx}.

Now, for $t> 1$, we have
\begin{align*}
   \abs{ \nabla_v \nabla_u P_tf(x)}=\abs{\nabla_v \nabla_u P_1 P_{t-1}f(x)}
    &\leq  C_2\abs{u}\abs{v}\sup_{y\in\mathbb{R}^{d}}\norm{\nabla P_{t-1}f(y)}_{\operatorname{op}}\\
    &\leq C_2\abs{u}\abs{v}\brac{\sup_{y\in\mathbb{R}^{d}}\norm{\nabla f(y)}_{\operatorname{op}}}\mathcal{R}(t-1),
\end{align*}
where we applied \eqref{R:t:eqn}, which gives the estimate \eqref{estimate_secondgradientsemigroup_prop}. 

Finally, as the last step of the proof, we will prove \eqref{step_Jxintegrable}. It follows from the definition of $J_x$ and \eqref{equation_itoformulaapplied} that 
\begin{align*}
    \E{\abs{J_x}}&\leq \E{\abs{\nabla_v f(X_t^x)}\abs{G(x,t)}}\abs{u}+\E{\abs{\nabla G(x,t) }}\abs{u}\abs{v}\abs{P_tf(x)}
    \\
    &\qquad\qquad\qquad+\E{\abs{\mathcal{B}_{t,x}(f)}\abs{\nabla G(x,t) } }\abs{u}\abs{v}. 
\end{align*}
The first and second term on the right hand side are finite due to Lipschitz continuity of $f$ and Lemma~\ref{lemma_GandnablaG}. To see the last term is also finite, recall \eqref{cond_moment_largejump} in Condition~\ref{cond_levymeasure} which allows us to choose $\lambda$ such that $\tau/(\tau-1)< \lambda<\Lambda $. It follows that
\begin{align*}
    \E{\abs{\mathcal{B}_{t,x}(f)}\abs{\nabla G(x,t) } }\leq  \E{\abs{\mathcal{B}_{t,x}(f)}^\lambda}^{1/\lambda} \E{\abs{\nabla G(x,t)}^{\frac{\lambda}{\lambda-1}}}^{\frac{\lambda-1}{\lambda}}. 
\end{align*}
Again by Lemma~\ref{lemma_GandnablaG}, $\nabla G(x,t)$ is $q$-integrable for $\tau>q\geq 1$ and our choice of $\lambda$ ensures $\tau> \lambda/(\lambda-1)$. This, together with Lemma~\ref{lemma_boundB_tx}, indicates the right hand side of the above equation is finite. Thus, we have shown $J_x$ is integrable for every $x\in\R^d$. 
The proof is complete.
\end{proof}

\begin{remark}\label{remark_needtaugreaterthan2}
Here we explain the importance of the parameter $\tau$ in Condition~\ref{cond_levymeasure}, noting that this parameter does not appear in the earlier reference \cite{kulik2023gradient}, but is needed in this paper. 
Having $\tau>2$ is crucial if we want the results of Proposition~\ref{prop_semigroupestimate} to hold in the case that the L\'{e}vy process in \eqref{sde_cylindricallevy} is an $\alpha$-stable process $L^\alpha$ with $1< \alpha<2$. Lemma~\ref{lemma_GandnablaG} shows that the constant $\tau$ in Condition~\ref{cond_levymeasure}  determines $q$-integrability of $\nabla G(x,t)$, i.e. 
    \begin{align}
    \label{roleoftau_remark}
        \E{\abs{\nabla G(x,t)}^q}<\infty, \quad \tau>q\geq 1.
    \end{align}
     Now let us consider the calculation in \eqref{step_needtaugreaterthan2} where the H\"{o}lder's inequality is applied. The $\alpha$-stable process $L^\alpha$ with $1<\alpha<2$ has moments up to order $p<\alpha$. Then Lemma~\ref{lemma_boundB_tx} implies for a Lipschitz function $f$, $\mathcal{B}_{t,x}(f)$ is $p$-integrable for $p<\alpha$ and in particular $p<2$. This suggests $\nabla G(x,t)$ must be $q$-integrable for some $q>2$ in order for the right hand side of \eqref{step_needtaugreaterthan2} to remain bounded. In view of \eqref{roleoftau_remark}, requiring $\tau>2$ fulfills such a purpose.
\end{remark}

\begin{remark}
\label{remark_C_2dimensiondependence}
In this remark, we spell out the dimension dependence of $C_2$. 
    Since $\lambda$ satisfies $\tau/(\tau-1)< \lambda<\Lambda $, we deduce that $\lambda/(\lambda-1)<\tau$. Moreover,
\begin{align*}
    \frac{\lambda}{\lambda-1}>\frac{\tau/(\tau-1)}{\Lambda-1}=\frac{\tau}{(\tau-1)(\Lambda-1)}. 
\end{align*}
Based on this, if we set $q=\lambda/(\lambda-1)$ then finding $C_2$ that satisfies \eqref{def_constantC2} reduces to finding $C_2$ such that 
\begin{align}
\label{C_2_new}
    C_2\geq  \sup_{t\in [0,1],x\in\R^d} \E{\abs{G(x,t)}}+\sup_{t\in [0,1],x\in\R^d}C_1(\lambda)^{1/\lambda}  \E{\abs{\nabla G(x,t)}^q}^\frac{1}{q},
\end{align}
for some $q$ satisfying $\frac{\tau}{(\tau-1)(\Lambda-1)}<q<\tau$ and some $\lambda$ satisfying $\frac{\tau}{\tau-1}< \lambda<\Lambda $. While the explicit form of $C_2$ can be derived using Lemma~\ref{lemma_GandnablaG}, it is quite cumbersome; so we will focus only on the dimension dependence of $C_2$.

Via Appendix~\ref{appendix_malliavin}  and Jensen's inequality, 
\begin{align*}
    C_1(\lambda)\leq  \E{\abs{ \int_0^1\int_{\R^d}\abs{\xi}\widehat{N}
       (ds,d\xi)}^\lambda}
       &=\E{\abs{\sum_{i=1}^d\int_0^1\int_{\R^d}\abs{\xi_j}\widehat{N}_j(ds,d\xi_j) }^\lambda}\\
       &\leq d^{\lambda-1}\sum_{i=1}^d\E{\abs{\int_0^1\int_{\R^d}\abs{\xi_j}\widehat{N}_j(ds,d\xi_j) }^\lambda}.
\end{align*}
Hence, in terms of the dimension dependence, $\left(C_1(\lambda)\right)^{1/\lambda}$ is of the order $\mathcal{O}(d)$ as $d\rightarrow\infty$. Next due to the fact that $\frac{q}{2(\tau-q)}<1$,  in terms of the dimension dependence, our upper bound of the quantity $\sup_{t\in [0,1],x\in\R^d} \E{\abs{G(x,t)}}$ in Lemma~\ref{lemma_GandnablaG} is of the order
\begin{align*}
\mathcal{O}\left(d^{\frac{3}{2}+\tau\brac{\frac{1}{2(\tau-1)}\vee 1 }} +d^{\frac{q_0}{2(q_0-1)}}\right),
\end{align*}
as $d\rightarrow\infty$.
Meanwhile, our upper bound on $\sup_{t\in [0,1],x\in\R^d} \E{\abs{\nabla G(x,t)}^q}^\frac{1}{q}$ in Lemma~\ref{lemma_GandnablaG} is of the order 
\begin{align*}
\mathcal{O}\left(d^{\frac{4q_0}{q_0-q}+\frac{5}{2q}}+d^{\frac{q_0}{q}+\frac{7}{2q}-\frac{q_0}{q\tau}}+d^{\frac{\tau}{q(\tau-q)}+\frac{5}{2q}}+d^{\frac{\tau}{q}+\frac{5}{2q}}\right),
\end{align*}
as $d\rightarrow\infty$, 
with $1\leq q<q_0<\tau$. Since $\frac{q_0}{2(q_0-1)}\leq \frac{4q_0}{q_0-q}$ and based on \eqref{C_2_new}, we can choose $C_2$ of the order $\mathcal{O}\brac{d^\mathcal{B}}$
as $d\rightarrow\infty$, where 
\begin{align*}
    \mathcal{B}&:=\max\Bigg\{\frac{3}{2}+\tau\brac{\frac{1}{2(\tau-1)}\vee 1 },\frac{4q_0}{q_0-q}+\frac{5}{2q}+1,\frac{q_0}{q}+\frac{7}{2q}-\frac{q_0}{q\tau}+1,\\
    &\hspace{16em}\frac{\tau}{q(\tau-q)}+\frac{5}{2q}+1, \frac{\tau}{q}+\frac{5}{2q}+1 \Bigg\}.
\end{align*}
In particular, when the  cylindrical L\'{e}vy process in \eqref{sde_cylindricallevy} is an $\alpha$-stable L\'{e}vy process $\{L^\alpha_t:t\geq 0\}$ with $1< \alpha< 2$, we have $\tau\brac{\frac{1}{2(\tau-1)}\vee 1 }\geq \tau>\max \left\{\frac{\alpha}{\alpha-1}, 2 \right\}$ (see Remark \ref{remark_stablenoisebelongtoourclass}), so that $\mathcal{B}\to \infty$ as $\alpha\to 1^+$. In plain words, the dimension dependence of $C_2$ gets worse as $\alpha\to 1^+$. 

\end{remark}


\section{Stein's method and proof of Theorem~\ref{theorem_steinmethodgeneralbound}}
\label{section_steinmethod} 

In this section, we introduce Stein's method for approximation of the invariant measure $\nu$ of the process in  \eqref{sde_cylindricallevy} and present
the proof of Theorem~\ref{theorem_steinmethodgeneralbound}, 
which provides an upper bound on the Wasserstein distance between the law of a generic random variable $F$ and $\nu$.

Denote $\mathcal{L}$ the infinitesimal generator associated with \eqref{sde_cylindricallevy}. It is a linear operator from $\mathcal{C}^2$ to itself and takes the form
\begin{align}
\label{def_generator}
    \mathcal{L}h(x):= \inner{b(x),\nabla h(x)}+\mathcal{L}_0h(x),
\end{align}
where $\mathcal{L}_0$ is the non-local operator defined as:
\begin{align*}
    \mathcal{L}_0h(x):=\int_{\R^d}\brac{h(x+\xi)-h(x)-\inner{\xi,\nabla h(x)}\mathbbm{1}_{\{\abs{\xi}\leq R\}}}m(d\xi),
\end{align*}
which includes \eqref{L:0:operator:stable} as a special case.
Let $f$ be a Lipschitz function. A Stein's equation for approximation of the invariant measure $\nu$ of \eqref{sde_cylindricallevy} in the Wasserstein distance is 
\begin{align}
\label{equation_stein}
    \mathcal{L} h(x)=f(x)-\E{W},
\end{align}
where $W$ is a random variable distributed according to $\nu$.
Formally, it is easy to see  \begin{align*}
        h_f(x)=\int_0^\infty \left(\E{f(W)}-P_tf(x)\right)dt
    \end{align*}
is a solution to the Stein's equation \eqref{equation_stein}. To rigorously justify this, we need to study strong continuity property of the semigroups $\{P_t:t\geq 0\}$ associated with the SDE \eqref{sde_cylindricallevy}.   The next argument is motivated by the argument in \cite[Appendix B]{Gor19}. We introduce the function space
\begin{align}
\label{def_functionspaceS}
    \mathcal{S}:=\brac{1+\abs{x}^2}^{\lambda /2}\mathcal{C}_0,
\end{align}
where $\lambda $ is any constant which satisfies $1<\lambda <\Lambda$, and $\mathcal{C}_0$ is the set of continuous functions on $\R^d$ vanishing at $\infty$. The function space $\mathcal{S}$ is a Banach space equipped with the norm 
\begin{align*}
    \abs{h}_\mathcal{S}:=\sup_{x\in \R^d}\frac{\abs{f(x)}}{\brac{1+\abs{x}^2}^{\lambda /2}}.
\end{align*}
We also let $\mathcal{C}_c^2$ denote the set of compactly supported functions for which the first and second derivatives are also compactly supported. Finally, denote $\operatorname{Lip}(1)$ the set of $1$-Lipschitz functions on $\R^d$.  Our goal is to show
the following technical lemma:

\begin{lemma}
\label{lemma_semigroupstronglycontinuous}
    The semigroup $\{P_t:t\geq 0\}$ associated with \eqref{SDE_alphastable} is strongly continuous on the Banach space $\mathcal{S}$. Moreover, $\operatorname{Lip}(1)$ is a subspace in $\mathcal{S}$.
\end{lemma}
\begin{proof}
To justify strong continuity of $P_t$ on $\mathcal{S}$, we make the following claims. 

\begin{enumerate}[label=\roman*)]
\item $\brac{\mathcal{S},\abs{\cdot}_\mathcal{S}}$ is a Banach space.
\item The space $\mathcal{S}'=\brac{1+\abs{x}^2}^{\lambda /2}\mathcal{C}^2_c=\mathcal{C}^2_c$ equipped with $\abs{\cdot}_\mathcal{S}$-norm is dense in $\mathcal{S}$.
\item $P_t$ maps $\mathcal{S}$ into $\mathcal{S}$. 
\item $\operatorname{Lip}(1)\subset \mathcal{S}$, which implies any Lipschitz function can be approximated by elements in $\mathcal{S}'$. 
\item For any $f\in \mathcal{S}'$, it holds that $\lim_{t\to 0}\abs{P_tf-f}_\mathcal{S}= 0$. This implies for any $f\in \mathcal{S}$, we have $\lim_{t\to 0}\abs{P_tf-f}_\mathcal{S}= 0$
\end{enumerate}
Item i) is immediate since $\mathcal{C}_0$ equipped with the supremum norm is a Banach space. To show ii), we recall the Stone–Weierstrass theorem for $\R^d$ which is a locally compact space. The fact that $\mathcal{C}^2_c$ separates points in $\R$ and vanishes nowhere, and also that $\mathcal{C}^2_c$ is a sub-algebra in $\mathcal{C}_0$ imply $\mathcal{C}^2_c$ is dense in $\mathcal{C}_0$. This means $\mathcal{S}'$ is dense in $\mathcal{S}$.

Regarding to item iii), we follow the argument in \cite[Proof of Theorem 3.1.9]{applebaum2009levy} and use the formula
\begin{align*}
    P_tf(x)=\int_{\R^d} f(x+y)q_t(dy),
\end{align*}
where $q_t(\cdot)$ is the law of $X^0_t$. Next, assume $f\in \mathcal{S}'$, that is 
\begin{align*}
    f(x)=\brac{1+\abs{x}^2}^{\lambda /2}g(x),
\end{align*}
for some $g\in \mathcal{C}^2_c$. Since $\mathcal{S}'$ is dense in $\mathcal{S}$, if one can show $P_tf\in \mathcal{S}$ then item iii) follows. Notice that we have
\begin{align*}
    \lim_{\abs{x}\to \infty}\frac{\abs{P_tf(x)}}{\brac{1+\abs{x}^2}^{\lambda /2}}&\leq  \lim_{\abs{x}\to \infty}\int_{\R^d} \frac{\brac{1+(x+y)^2}^{\lambda /2}\abs{g(x+y)}}{\brac{1+\abs{x}^2}^{\lambda /2}}q_t(dy)\\
    &= \int_{\R^d} \lim_{\abs{x}\to \infty}\frac{\brac{1+(x+y)^2}^{\lambda /2}\abs{g(x+y)}}{\brac{1+\abs{x}^2}^{\lambda /2}}q_t(dy)=0, 
\end{align*}
where we applied the dominated convergence theorem to interchange limit and integral in the above calculation. Notice $g\in \mathcal{C}^2_c$ is bounded, so we have
\begin{align*}
    \frac{\brac{1+(x+y)^2}^{\lambda /2}\abs{g(x+y)}}{\brac{1+\abs{x}^2}^{\lambda /2}}\leq C\abs{y}^\lambda,
\end{align*}
for some constant $C$ independent of $x$. Then if we know for a fixed $t$ and $\lambda$, 
\begin{align}
\label{momentbound_X_tcylindricallevynoise}
    \E{\abs{X^0_t}^\lambda}=\int_{\R^d}\abs{y}^\lambda q_t(dy)<\infty,
\end{align}
then the dominated convergence theorem can be applied. 

Next, let us show \eqref{momentbound_X_tcylindricallevynoise} holds. \cite[Theorem 2.5.2]{applebaum2009levy} and \eqref{cond_moment_largejump} in Condition~\ref{cond_levymeasure} imply there exists a positive function $C(\lambda,t)$ which is finite for every $\lambda,t$ and is such that $\E{\abs{Z_t}^\lambda}<C(\lambda,t)$. This and Condition~\ref{cond_driftb} provide us with
\begin{align*}
    \E{\abs{X^0_t}^\lambda}\leq C(\lambda,t)+\int_0^t\theta_1\E{\abs{X_s}^\lambda}ds,
\end{align*}
and by Gronwall's inequality,
\begin{align*}
    \E{\abs{X^0_t}^\lambda}\leq C(\lambda,t)e^{\theta_1t}.
\end{align*}
This proves \eqref{momentbound_X_tcylindricallevynoise}.

Next, item iv) is true since a Lipschitz function $f$ has sub-linear growth, and hence
\begin{align*}
    \abs{f}_{\mathcal{S}}=\sup_{x\in\R^d}\frac{\abs{f(x)}}{ \brac{1+\abs{x}^2}^{\lambda /2}}<\infty. 
\end{align*}
For v), it is sufficient to consider only $t\in[0,1]$. We apply It\^{o}'s formula \cite[Theorem 4.4.7]{applebaum2009levy} to $f\in \mathcal{S}'$ to get
\begin{align}
\label{equation_itoformulaforstrongcontinuity}
    \E{f(X_t^x)}-f(x)&=\E{\int_0^t \nabla f(X_s)b(X_s^x)ds}+\E{\int_0^t\int_{\abs{\xi}\geq R} \left(f\brac{X_{s-}^x+\xi}-f(X_{s-}^x)\right)m(d\xi)ds}\nonumber\\
    &\qquad\qquad+\frac{1}{2}\E{\int_0^t\int_{\abs{\xi}< R}\nabla^2f(a(X_{s},\xi))\xi^T\xi m(d\xi)ds },
\end{align}
where $a(X_{s},\xi)$ is some element in between $X_s$ and $X_s+\xi$. Next, let us each term on the right hand side of \eqref{equation_itoformulaforstrongcontinuity}.  

By \eqref{cond_moment_largejump} in Condition~\ref{cond_levymeasure}, there exists a constant $C$ such that 
\begin{align*}
   &\abs{ \E{\int_0^t\int_{\abs{\xi}\geq R} \left(f\brac{X_{s-}^x+\xi}-f(X_{s-}^x)\right)m(d\xi)ds}}\\
   &\leq \brac{\sup_{y\in\R^d} \norm{\nabla f(y)}_{\operatorname{op}}}\abs{\int_0^t\int_{\abs{\xi}\geq R} \abs{\xi} m(d\xi)ds }\leq \brac{\sup_{y\in\R^d} \norm{\nabla f(y)}_{\operatorname{op}}} Ct.
\end{align*}

Moreover, $m$ is a L\'{e}vy measure so that $\int_{\R^d}(\abs{\xi}^2\wedge 1)m(d\xi)<\infty$. This, and the assumption $R\in (0,1]$, imply that there exists some constant $C'$ such that 
\begin{align*}
   \abs{ \E{\int_0^t\int_{\abs{\xi}< R}\nabla^2f(a(X_{s},\xi))\xi^T\xi m(d\xi)ds }}\leq \brac{\sup_{y\in\R^d} \norm{\nabla^2 f(y)}_{\operatorname{op}}}C't. 
\end{align*}

Finally, we consider the first term on the right hand side of \eqref{equation_itoformulaforstrongcontinuity} restricted to $t\in[0,1]$. By the estimate in \eqref{momentbound_X_tcylindricallevynoise}, one can define
\begin{align*}
    C'':=\sup_{s\in[0,1]}\E{\abs{X^x_s}}<\infty. 
\end{align*}
Then Condition~\ref{cond_driftb} implies that 
\begin{align*}
    \abs{\E{\int_0^t \nabla f(X_s)b(X_s^x)ds}}&\leq \brac{\sup_{y\in\R^d} \norm{\nabla f(y)}_{\operatorname{op}}}\theta_1\int_0^t \E{\abs{X_s}}ds\\
    &\leq \brac{\sup_{y\in\R^d} \norm{\nabla f(y)}_{\operatorname{op}}}\theta_1 C'' t. 
\end{align*}
We combine \eqref{equation_itoformulaforstrongcontinuity} and the previous estimates to get
\begin{align*}
   \lim_{t\to 0}\abs{P_tf-f}_\mathcal{S}= \lim_{t\to 0}\sup_{x\in\mathbb{R}^{d}}\frac{\abs{\E{f(X_t^x)}-f(x)}}{\brac{1+\abs{x}^2}^{\lambda /2}}=0. 
\end{align*}
This completes the proof.

\end{proof}

The next result is a corollary of Proposition~\ref{prop_semigroupestimate}. This corollary combined with the Stein's equation at \eqref{equation_stein} will yield Theorem~\ref{theorem_steinmethodgeneralbound}. 

\begin{corollary}
\label{coro_semigroupestimate}
     Assume Conditions~\ref{cond_levymeasure},\ref{cond_driftb} and~\ref{cond_wassersteindecay} hold. Then, for any Lipschitz function $f$, the function
    \begin{align*}
        h_f(x):=\int_0^\infty \left(\E{f(W)}-P_tf(x)\right)dt,
    \end{align*}
    solves the Stein's equation \eqref{equation_stein}. 
    Moreover, $h_f$ is twice differentiable and satisfies
    \begin{align*}
      \sup_{x\in \R^d}  \abs{\nabla_u h_f(x)}&\leq \brac{\int_0^\infty \mathcal{R}(t)dt}\brac{\sup_{y\in\mathbb{R}^{d}}\norm{\nabla f(y)}_{\operatorname{op}}} \abs{u},\\
        \sup_{x\in \R^d}   \abs{\nabla_v\nabla_u h_f(x)}&\leq C_2\brac{1+\int_0^\infty \mathcal{R}(t)dt}\brac{\sup_{y\in\mathbb{R}^{d}}\norm{\nabla f(y)}_{\operatorname{op}}}\abs{u}\abs{v},
    \end{align*}
where $\mathcal{R}(t)$ is given in \eqref{R:t:eqn} and the constant $C_2$ is defined in \eqref{def_constantC2}. 
\end{corollary}

\begin{proof}
  Since $W\sim \nu$ is the invariant measure of \eqref{sde_cylindricallevy}, $\E{P_tf(W)}=\E{f(W)}$ so that 
  \begin{align*}
     \E{f(W)-P_tf(x)}=\E{P_tf(W)-P_tf(x)}\leq \E{\abs{W-x}}\mathcal{R}(t),
  \end{align*}
  where we applied \eqref{R:t:eqn}.
Since $\int_0^\infty \mathcal{R}(t)dt<\infty$ by Condition~\ref{cond_driftb}, $h_f$ is well-defined. We proceed to show that $h_f$ solves \eqref{equation_stein} along the line of \cite[Proof of Theorem 5]{Gor19}. Note that 
Lemma~\ref{lemma_semigroupstronglycontinuous} and \cite[Proposition 1.5]{ethierkurtzbook} provide us with
\begin{align}
\label{equation_ethierprelimit}
    f(x)-P_tf(x)=\mathcal{L}\brac{\int_0^t \left(P_sf(x)-\E{f(W)}\right)ds},
\end{align}
where $\mathcal{L}$ is defined at \eqref{def_generator}. 
Let $u(x,t):=\int_0^t \left(P_sf(x)-\E{f(W)}\right)ds$. Then $u(x,t)$ is Lipschitz with respect to its time variable, and moreover $\{u(x,t) \}_{t\geq 0}$ is a Cauchy sequence in the function space $\mathcal{S}$ defined in \eqref{def_functionspaceS}. Indeed, for any $0\leq r\leq t$, we have
\begin{align*}
    \abs{u(x,t)-u(x,r)}&=\int_r^t \left(P_sf(x)-\E{f(W)}\right)ds
    \\
&\leq \E{\abs{W-x}}\int_r^t  \mathcal{R}(s)ds\\
&\leq  \brac{\E{\abs{W}}+\abs{x}}\mathcal{R}(0)\abs{t-r},
\end{align*}
where we applied \eqref{R:t:eqn}, which leads to
\begin{align*}
    \abs{u(x,t)-u(x,r)}_{\mathcal{S}} \leq \sup_{x\in\R^d} \frac{\E{\abs{W}}+\abs{x}}{\brac{1+x^2}^{\lambda /2}}\mathcal{R}(0)\abs{t-r}. 
\end{align*}
Thus, $u(x,t)$ is a Cauchy sequence in $\mathcal{S}$ and 
$
    \lim_{t\to\infty} u(x,t)=\int_0^\infty P_sf(x)-\E{f(W)}ds. 
$
Now let us take limit $t\to\infty$ on both sides of \eqref{equation_ethierprelimit}. Since $P_t$ is strongly continuous on $\mathcal{S}$ per Lemma \eqref{lemma_semigroupstronglycontinuous}, \cite[Corollary 1.6]{ethierkurtzbook} tells us that the generator $\mathcal{L}$ is closed on its domain. This implies
\begin{align*}
    f(x)-\E{f(W)}&=\lim_{t\to\infty} \left(f(x)-P_tf(x)\right)
    \\
    &=\mathcal{L}\brac{\lim_{t\to\infty} \int_0^t \left(P_sf(x)-\E{f(W)}\right)ds}
    \\
    &=\mathcal{L}\brac{\int_0^\infty \left(P_sf(x)-\E{f(W)}\right)ds}. 
\end{align*}
Therefore $h_f$ solves the Stein's equation \eqref{equation_stein}. Finally, derivative estimates of $h_f$ follow directly from Proposition ~\ref{prop_semigroupestimate}, which completes the proof.
\end{proof}


\section{Proofs of Theorem~\ref{theorem_eulerscheme}}
\label{section_proof_eulerscheme}
In this section, we will use the following notations for a Lipschitz function $f$. 
\begin{align*}
    P_tf(x)=\E{f\brac{X^x_t}},\qquad Q_kf(x)=\E{f\brac{Y^x_k}}. 
\end{align*}

We introduce the next three technical lemmas from \cite{chenxu2023euler}.

\begin{lemma}(\cite[Lemma 2.6]{chenxu2023euler})
    \label{lemma_gradientdifferencefromchenxupaper}
Assume $h$ is a function from $\R^d$ to $\R$ and satisfies
\begin{align}\label{h:assumption}
    \sup_{z\in \R^d} \norm{\nabla h(z)}_{\operatorname{op}}<\infty,\qquad \sup_{z\in \R^d} \norm{\nabla^2h(z)}_{\operatorname{op}}<\infty. 
\end{align}
Then for any $\beta\in [1,2]$ and $x,y\in \R^d$, we have
    \begin{align*}
        \abs{\nabla h(x)-\nabla h(y)}\leq \brac{2\sup_{z\in \R^d} \norm{\nabla h(z)}_{\operatorname{op}}+\sup_{z\in \R^d} \norm{\nabla^2h(z)}_{\operatorname{op}}}\abs{x-y}^{\beta-1}. 
    \end{align*}
\end{lemma}

\begin{lemma}(\cite[Lemma 2.4]{chenxu2023euler})
\label{lemma_fractionallaplacianfromchenxupaper}
Assume $h$ is a function from $\R^d$ to $\R$ that satisfies \eqref{h:assumption}. Then, it holds that:
    \begin{align*}
        \abs{\Delta^{\alpha/2}h(x)-\Delta^{\alpha/2}h(y)}\leq \frac{4dp_\alpha \brac{{\sup_{z\in \R^d} \norm{\nabla^2h(z)}_{\operatorname{op}}}}  }{(2-\alpha)(\alpha-1)}\abs{x-y}^{2-\alpha}. 
    \end{align*}
\end{lemma}

\begin{lemma}(\cite[Lemma 2.2]{chenxu2023euler})
\label{lemma_someestimatesfromchenxueuler}
   For all $t\in (0,1]$ and $\beta\in [1,\alpha)$, it holds that
  \begin{align*}
      \E{\abs{X^x_t-x}^\beta}&\leq \brac{2\theta_1^\beta C_3(\beta)(1+\abs{x}^2)^{\beta/2}+2\E{\abs{L^\alpha_1}^\beta}}t^{\beta/\alpha},\\
      \E{\abs{X^x_\eta-Y^x_1}^\beta}&\leq \theta_1 \brac{2\theta_1^\beta C_3(\beta)(1+\abs{x}^2)^{\beta/2}+2\E{\abs{L^\alpha_1}^\beta}}\eta^{\beta+\frac{\beta}{\alpha}},
  \end{align*}
where the constant $C_3$ (as a function of $\beta$) is defined in Lemma~\ref{lemma_invariantmeasure_continuousX}. 
\end{lemma}


The next lemma quantifies how well the one-step iterate of our discretization schemes tracks the original SDE \eqref{SDE_alphastable}. While our Lemma~\ref{lemma_semigroupfirststep} contains an analogous statement to \cite[Lemma 2.7]{chenxu2023euler}, we get a better dependence on the stepsize $\eta$, that is $\eta^2$. 
\begin{lemma}
    \label{lemma_semigroupfirststep}
Assume $h:\R^d \to \R$ is a function satisfying 
\begin{align*}
    \sup_{y\in\mathbb{R}^{d}}\norm{\nabla h(y)}_{\operatorname{op}}<\infty,\quad \sup_{y\in\mathbb{R}^{d}}\norm{\nabla^2 h(y)}_{\operatorname{op}}<\infty. 
\end{align*}
Then for all $x\in \R^d$ and $\eta\in(0,1)$, it holds that
  \begin{align*}
   & \abs{P_\eta h(x)-Q_1h(x)}\\
   &\leq \brac{3\sup_{y\in \R^d}\norm{\nabla h(y)}_{\operatorname{op}} +\sup_{y\in \R^d}\norm{\nabla^2 h(y)}_{\operatorname{op}}}\bigg(\brac{6\theta_1+\theta_1^2+\frac{4\theta_2d p_\alpha}{(2-\alpha)(\alpha-1)}} \\
     &\hspace{12em}C_3(1)(1+\abs{x}^2)^{1/2}+6\E{\abs{L^\alpha_1}}+\abs{\Delta^{\alpha/2}b(0)}\bigg)\eta^{2}. 
\end{align*}
\end{lemma}
\begin{proof}
 By a Taylor's expansion, we can write
\begin{align*}
    &\E{h(X^x_\eta)}-\E{h(Y^x_1)}\\
    &=\E{\inner{\nabla h(Y^x_1), X^x_\eta-Y^x_1}}+\int_0^1 \inner{\nabla h\brac{Y^x_1+r(X^x_\eta-Y^x_1)}-\nabla h(Y^x_1), X^x_\eta-Y^x_1 } dr\\
    &=\mathcal{A}_1+\mathcal{A}_2. 
\end{align*}
Regarding the term $\mathcal{A}_1$, It\^{o} formula says
\begin{align}
\label{step_decompose_onestepdifference}
    \mathcal{A}_1&=\E{\inner{\nabla h(Y^x_1),\int_0^\eta \brac{b(X^x_s)-b(x)}ds }}\nonumber\\
    &=\inner{\nabla h(Y^x_1), \int_0^\eta \int_0^s \E{\inner{\nabla b(X^x_r), b(X^x_r)}+\Delta^{\alpha/2} b(X^x_r)}  drds}.
\end{align}
Condition \ref{cond_driftb} and Lemma \ref{lemma_invariantmeasure_continuousX} imply that
\begin{align*}
    \E{\abs{\inner{\nabla b(X^x_r), b(X^x_r)}}}&\leq \theta_1^2 \E{\abs{X^x_r}}\leq \theta_1^2 C_3(1)(1+\abs{x}^2)^{1/2}. 
\end{align*}
Moreover, Condition \ref{cond_driftb} and Lemma \ref{lemma_fractionallaplacianfromchenxupaper} indicate
\begin{align*}
\abs{\Delta^{\alpha/2} b(X^x_r)-\Delta^{\alpha/2}b(0) }\leq \frac{4\theta_2 d p_\alpha}{(2-\alpha)(\alpha-1)}\abs{X^x_r}^{2-\alpha}
\end{align*}
so that
\begin{align*}
    \E{\abs{\Delta^{\alpha/2}b(X^x_r) }}&\leq \frac{4\theta_2d p_\alpha}{(2-\alpha)(\alpha-1)}\E{\abs{X^x_r}^{2-\alpha}}+\abs{\Delta^{\alpha/2}b(0)}\\
    &\leq \frac{4\theta_2d p_\alpha}{(2-\alpha)(\alpha-1)}C_3(1)\brac{1+\abs{x}^2}^{1-\frac{\alpha}{2}}+\abs{\Delta^{\alpha/2}b(0)}. 
\end{align*}
In the last line, we have applied Lemma \ref{lemma_invariantmeasure_continuousX} and the fact that $\E{\abs{X^x_r}^{2-\alpha}}\leq \E{\abs{X^x_r}}^{2-\alpha}$. 
Therefore, we can deduce from \eqref{step_decompose_onestepdifference} that 
\begin{align*}
    \abs{\mathcal{A}_1}&\leq \brac{\sup_{y\in \R^d}\norm{\nabla h(y)}_{\operatorname{op}}}\cdot\Bigg(\theta_1^2 C_3(1)(1+\abs{x}^2)^{1/2}+\abs{\Delta^{\alpha/2}b(0)}\\
    &\hspace{7em}+\frac{4\theta_2d p_\alpha}{(2-\alpha)(\alpha-1)}C_3(1)\brac{1+\abs{x}^2}^{1-\frac{\alpha}{2}}\Bigg)\eta^2\\
    &\leq \brac{\sup_{y\in \R^d}\norm{\nabla h(y)}_{\operatorname{op}}}\cdot\brac{\brac{\theta_1^2 +\frac{4\theta_2d p_\alpha}{(2-\alpha)(\alpha-1)}}C_3(1)(1+\abs{x}^2)^{1/2}+\abs{\Delta^{\alpha/2}b(0)}}\eta^2. 
    \end{align*} 

Next let us deal with the term $\mathcal{A}_2$. Assume $\beta \in [1,\alpha)$ then Lemma \ref{lemma_gradientdifferencefromchenxupaper} and Lemma \ref{lemma_someestimatesfromchenxueuler} imply
\begin{align*}
    \abs{\mathcal{A}_2}&\leq \int_0^1 \E{\brac{2\sup_{y\in \R^d}\norm{\nabla h(y)}_{\operatorname{op}} +\sup_{y\in \R^d}\norm{\nabla^2 h(y)}_{\operatorname{op}}}r^{\beta-1}\abs{X^x_\eta-Y^x_1}^\beta }dr\\
   &\leq  \brac{2\sup_{y\in \R^d}\norm{\nabla h(y)}_{\operatorname{op}} +\sup_{y\in \R^d}\norm{\nabla^2 h(y)}_{\operatorname{op}}} \\
   &\hspace{5em}\cdot\frac{\theta_1}{\beta}\brac{2\theta_1^\beta C_3(\beta)(1+\abs{x}^2)^{\beta/2}+2\E{\abs{L^\alpha_1}^\beta}}\eta^{\beta+\frac{\beta}{\alpha}}\\
      &\leq  \brac{2\sup_{y\in \R^d}\norm{\nabla h(y)}_{\operatorname{op}} +\sup_{y\in \R^d}\norm{\nabla^2 h(y)}_{\operatorname{op}}} \\
   &\hspace{5em}\cdot\frac{\theta_1}{\beta}\brac{2\theta_1^\beta C_3(\beta)(1+\abs{x}^2)^{\beta/2}+2\E{\abs{L^\alpha_1}^\beta}}\eta^{2}. 
\end{align*}
To get the last line, we choose $\beta$ close enough to $\alpha$ such that $\beta+\frac{\beta}{\alpha}>2$ and hence $\eta^2>\eta^{\beta+\frac{\beta}{\alpha}}$. Since the above bound on $\mathcal{A}_2$ holds for any $\beta\in [1,\alpha)$, we can set $\beta=1$. Combining this with our previous bound on $\mathcal{A}_1$ yields

\begin{align*}
 &\E{h(X^x_\eta)}-\E{h(Y^x_1)}\\
     &\leq \brac{3\sup_{y\in \R^d}\norm{\nabla h(y)}_{\operatorname{op}} +\sup_{y\in \R^d}\norm{\nabla^2 h(y)}_{\operatorname{op}}}\bigg(3\theta_1^2 C_3(1)(1+\abs{x}^2)^{1/2}+2\theta_1\E{\abs{L^\alpha_1}}\\
     &\hspace{12em}+\frac{4\theta_2d p_\alpha}{(2-\alpha)(\alpha-1)}C_3(1)\brac{1+\abs{x}^2}^{1-\frac{\alpha}{2}}+\abs{\Delta^{\alpha/2}b(0)}\bigg)\eta^{2}\\
      &\leq \brac{3\sup_{y\in \R^d}\norm{\nabla h(y)}_{\operatorname{op}} +\sup_{y\in \R^d}\norm{\nabla^2 h(y)}_{\operatorname{op}}}\bigg(\brac{3\theta_1^2+\frac{4\theta_2d p_\alpha}{(2-\alpha)(\alpha-1)}} \\
     &\hspace{13em}C_3(1)(1+\abs{x}^2)^{1/2}+2\theta_1\E{\abs{L^\alpha_1}}+\abs{\Delta^{\alpha/2}b(0)}\bigg)\eta^{2},
\end{align*}
noting that $1-\frac{\alpha}{2}<\frac{1}{2}$. 

\end{proof}


Now, we are finally ready to prove Theorem \ref{theorem_eulerscheme}.

\begin{proof}[Proofs of Theorem \ref{theorem_eulerscheme}]
The proof will follow the strategy in \cite[Proof of Theorem 1.2]{chenxu2023euler} (see also \cite{chenxu23amarkovapproach}). We start with the following decomposition that is in the spirit of the classical Lindeberg's principle. 
\begin{align*}
    &\E{h(X^x_{\eta N})}-\E{h(Y^x_N)}\\
    &=P_{\eta N}h(x)-Q_Nh(x)=\sum_{i=1}^N Q_{i-1}\brac{P_\eta-Q_1}P_{(N-i)\eta} h(x). 
\end{align*}
It follows that
 \begin{align}
    \label{estimate_semigroupsdecomposition}
        &d_{\operatorname{Wass}}(X_{\eta N},Y_N)\nonumber\\
        &=\sup_{h\in \operatorname{Lip(1)}}\abs{P_{N\eta}h(x)-Q_N h(x)}\nonumber\\
        &\leq \sup_{h\in \operatorname{Lip(1)}}\abs{Q_{N-1}\brac{P_\eta-Q_1}h(x)}+\sup_{h\in \operatorname{Lip(1)}}\sum_{i=1}^{N-1}\abs{Q_{i-1}(P_\eta-Q_1)P_{(N-i)\eta}h(x) }\nonumber\\
        &=\mathcal{B}_1+\mathcal{B}_2. 
    \end{align}
Let us first bound $\mathcal{B}_1$. Via the same calculation as the one for $\mathcal{A}_1$ at \eqref{step_decompose_onestepdifference}, we get 
\begin{align*}
    &\abs{\brac{P_\eta-Q_1}h(x)}\\
    &\leq \brac{\sup_{y\in\R^d} \norm{\nabla h(y)}_{\operatorname{op}}}\E{\abs{X^x_\eta-Y^x_1}}\\
   & \leq \brac{\sup_{y\in\R^d} \norm{\nabla h(y)}_{\operatorname{op}}} \abs{\int_0^\eta b(X^x_s)-b(x)dx}\\
     &\leq \brac{\sup_{y\in \R^d}\norm{\nabla h(y)}_{\operatorname{op}}}\cdot\brac{\brac{\theta_1^2 +\frac{4\theta_2d p_\alpha}{(2-\alpha)(\alpha-1)}}C_3(1)(1+\abs{x}^2)^{1/2}+\abs{\Delta^{\alpha/2}b(0)}}\eta^2.  
\end{align*}
Hence by Lemma \ref{lemma_invariantmeasure_discreteYk}, 
\begin{align*}
    \mathcal{B}_1&\leq \sup_{h\in \operatorname{Lip(1)}}\brac{\sup_{y\in \R^d}\norm{\nabla h(y)}_{\operatorname{op}}}\\
    &\qquad\cdot\brac{\brac{\theta_1^2 +\frac{4\theta_2d p_\alpha}{(2-\alpha)(\alpha-1)}}C_3(1)\E{(1+\abs{Y_{N-1}^x}^2)^{1/2}}+\abs{\Delta^{\alpha/2}b(0)}}\eta^2\\
    &\leq\brac{\brac{\theta_1^2 +\frac{4\theta_2d p_\alpha}{(2-\alpha)(\alpha-1)}}C_3(1)C_4(1)(1+\abs{x}^2)^{1/2}+\abs{\Delta^{\alpha/2}b(0)}}\eta^2.  
\end{align*}
Next we consider $\mathcal{B}_2$ in \eqref{estimate_semigroupsdecomposition}. Lemma \ref{lemma_semigroupfirststep} implies for any $x\in\R^d$ and $1\leq i\leq N-1$, 
\begin{align}
    \label{equation_step_i>1}
     & \sup_{h\in \operatorname{Lip(1)}}\abs{(P_\eta-Q_1)P_{(N-i)\eta} h(x)}\nonumber\\
       &\leq \sup_{h\in \operatorname{Lip(1)}}\brac{3\sup_{y\in \R^d}\norm{\nabla P_{(N-i)\eta}h(y)}_{\operatorname{op}} +\sup_{y\in \R^d}\norm{\nabla^2 P_{(N-i)\eta}h(y)}_{\operatorname{op}}}\nonumber\\
     &\qquad\cdot\bigg(\brac{3\theta_1^2+\frac{4\theta_2d p_\alpha}{(2-\alpha)(\alpha-1)}} C_3(1)(1+\abs{x}^2)^{1/2}+2\theta_1\E{\abs{L^\alpha_1}}+\abs{\Delta^{\alpha/2}b(0)}\bigg)\eta^{2}\nonumber \\
     &\leq \sup_{h\in \operatorname{Lip(1)}}\sup_{y\in \R^d}\norm{\nabla h(y)}_{\operatorname{op}}\brac{3 \mathcal{R}\brac{(N-i)\eta} +C_2\phi\brac{(N-i)\eta}}\nonumber\\
     &\qquad\cdot\bigg(\brac{3\theta_1^2+\frac{4\theta_2d p_\alpha}{(2-\alpha)(\alpha-1)}} C_3(1)(1+\abs{x}^2)^{1/2}+2\theta_1\E{\abs{L^\alpha_1}}+\abs{\Delta^{\alpha/2}b(0)}\bigg)\eta^{2}.
\end{align}
The last line is due to Proposition \ref{prop_semigroupestimate}. In particular, Lemma \ref{lemma_basedonjianwangpaper} tells us the Wasserstein decay rate $\mathcal{R}(t)$  in the aforementioned Proposition is
\begin{align*}
\mathcal{R}(t)=    \frac{2\brac{1-e^{-c_1L_0}}}{L_0}\exp\brac{-C_5t}. 
\end{align*}

Moreover, we observe that
\begin{align*}
   \sum_{i=1}^{N-1}\mathcal{R}\brac{(N-i)\eta}&= \frac{2\brac{1-e^{-c_1L_0}}}{L_0}\sum_{i=1}^{N-1}\exp\brac{-C_5(N-i)\eta}\\
   &\leq\frac{2\brac{1-e^{-c_1L_0}}}{L_0} \exp\brac{- C_5 N\eta}\int_1^{N}\exp\brac{\eta C_5x}dx\leq \frac{2\brac{1-e^{-c_1L_0}}}{L_0}\frac{1}{C_5 \eta}. 
\end{align*}

When $N\geq \frac{1}{\eta}$, the definition of $\phi$ in Proposition \ref{prop_semigroupestimate}  implies that 
\[\phi\brac{(N-x)\eta}= \begin{cases} 
      \frac{2\brac{1-e^{-c_1L_0}}}{L_0}\exp\brac{-C_5\brac{(N-x)\eta-1}} &  0\leq x\leq N-\frac{1}{\eta}, \\
      1 &  \text{ otherwise}
   \end{cases}
\]
Hence

\begin{align*}
    \sum_{i=1}^{N-1}\phi\brac{(N-i)\eta}
    &\leq \int_1^{N}\phi\brac{(N-x)\eta}dx\\
    &=\int_1^{N-\frac{1}{\eta}}\phi\brac{(N-x)\eta}dx+\int^{N}_{N-\frac{1}{\eta}}\phi\brac{(N-x)\eta}dx\\
    &=\frac{2\brac{1-e^{-c_1L_0}}}{L_0}\frac{1}{C_5\eta}\brac{1-\exp\brac{-C_5\eta(N-1)-1} }+\frac{1}{\eta}\\
    &=\frac{2\brac{1-e^{-c_1L_0}}}{L_0C_5}\frac{1}{\eta}+\frac{1}{\eta}
    \leq \brac{\frac{2\brac{1-e^{-c_1L_0}}}{L_0C_5}+1}\frac{1}{\eta}. 
\end{align*}

By combining the previous calculations and \eqref{equation_step_i>1}, we arrive at
\begin{align*}
    \mathcal{B}_2& \leq \brac{\frac{2\brac{1-e^{-c_1L_0}}}{L_0C_5}+\frac{2C_2\brac{1-e^{-c_1L_0}}}{L_0C_5}+C_2}\\
    &\qquad \cdot\bigg(\brac{3\theta_1^2+\frac{4\theta_2d p_\alpha}{(2-\alpha)(\alpha-1)}} C_3(1)\E{(1+\abs{Y_{i-1}}^2)^{1/2}}+2\theta_1\E{\abs{L^\alpha_1}}+\abs{\Delta^{\alpha/2}b(0)}\bigg)\eta\\
    &\leq \brac{\frac{2\brac{1-e^{-c_1L_0}}}{L_0C_5}+\frac{2C_2\brac{1-e^{-c_1L_0}}}{L_0C_5}+C_2}
    \\&\qquad \cdot\bigg(\brac{3\theta_1^2+\frac{4\theta_2d p_\alpha}{(2-\alpha)(\alpha-1)}} C_3(1)C_4(1)\brac{1+\abs{x}^2}^{1/2}+2\theta_1\E{\abs{L^\alpha_1}}+\abs{\Delta^{\alpha/2}b(0)}\bigg)\eta. 
\end{align*}
The last line is a consequence of Lemma \ref{lemma_invariantmeasure_discreteYk}. Now by summing up the bounds on $\mathcal{B}_1$ and $\mathcal{B}_2$, we obtain the stated estimate on $d_{\operatorname{Wass}}(X_{\eta N},Y_N)$ at \eqref{estimate_distancebetweenXandY}.

Finally by the triangle inequality,
\begin{align*}
    d_{\operatorname{Wass}}\brac{\nu_n,\nu}&\leq d_{\operatorname{Wass}}\brac{\nu_n,\operatorname{Law}(Y_N)}+d_{\operatorname{Wass}}\brac{\operatorname{Law}(Y_N), \operatorname{Law}(X_{\eta N})}+d_{\operatorname{Wass}}\brac{\operatorname{Law}(X_{\eta N}), \nu}\\
    &\leq d_{\operatorname{Wass}}\brac{\nu_n,\operatorname{Law}(Y_N)}+ C_11\eta^{1+\frac{1}{\alpha}-\frac{1}{\beta}}+d_{\operatorname{Wass}}\brac{\operatorname{Law}(X_{\eta N}), \nu}. 
\end{align*}
Letting $N\to \infty$ and notice that Lemmas \ref{lemma_invariantmeasure_continuousX} and \ref{lemma_invariantmeasure_discreteYk} imply  
\begin{align*}
\lim_{N\to\infty}d_{\operatorname{Wass}}\brac{\nu_n,\operatorname{Law}(Y_N)}=\lim_{N\to\infty}d_{\operatorname{Wass}}\brac{\operatorname{Law}(X_{\eta N}), \nu}=0. 
    \end{align*}
Thus we have deduced \eqref{estimate_distancebetweeninvariantmeasures}. 
\end{proof}



\section*{Acknowledgements}
The authors are grateful to Peng Chen, Alexei Kulik, Enrico Priola, Jian Wang, Lihu Xu and Xicheng Zhang 
for helpful comments and discussions.
Lingjiong Zhu is partially supported by the grants NSF DMS-2053454, NSF DMS-2208303.

\bibliographystyle{alpha}
\bibliography{refs}

\newpage

\begin{center}
\Large \bf Appendices: Gradient estimates for semigroups associated with stochastic differential equations driven by cylindrical L\'{e}vy processes
\end{center}

\appendix

The Appendices are organized as follows:
\begin{itemize}
\item In Appendix~\ref{appendix_eulerschemepareto}, we provide our result on Euler-Maruyama scheme with Pareto noise.  
    \item In Appendix~\ref{section_proofexistenceofinvariantmeasure}, we provide the technical proofs of Lemmas \ref{lemma_invariantmeasure_continuousX}, \ref{lemma_invariantmeasure_discreteYk} in the main paper and Lemma~\ref{lemma_invariantmeasure_discretepareto} from Appendix~\ref{appendix_eulerschemepareto}.
    \item In Appendix~\ref{section_proofresultbyjianwang}, we provide the technical proof of Lemma~\ref{lemma_basedonjianwangpaper} in the main paper.
        \item In Appendix~\ref{appendix_malliavin}, we provide some technical background on Malliavin calculus on Poisson space. 
    \item In Appendix~\ref{section_proofoflemmaGandnablaG}, we present the technical proof of Lemma~\ref{lemma_GandnablaG} from Appendix~\ref{appendix_malliavin}.
    \item Finally, in Appendix~\ref{sec:OU}, we present an explicit calculation for the Euler scheme of a one-dimensional Ornstein-Uhlenbeck process driven by an $\alpha$-stable L\'{e}vy process, and we show that the approximation error bound has a linear dependence on the stepsize $\eta$, as in Theorem~\ref{theorem_eulerscheme}.
\end{itemize}

\section{Euler-Maruyama scheme with Pareto noise}
\label{appendix_eulerschemepareto}

In this appendix, for the purpose of completeness, we include here a second discretization of Equation \eqref{SDE_alphastable} which will employ i.i.d. $\R^d$-valued Pareto random variables $\zeta_m,m\in\N$: 
\begin{align}
\label{equation_discretePareto}
    U_{m+1}=U_m+\eta b(U_m)+\frac{\eta^{1/\alpha}}{\sigma}\zeta_m, \quad U_0=x, 
\end{align}
where $\sigma:=(\alpha/2p_\alpha)^{1/\alpha}$, 
$p_{\alpha}$ is a constant defined in \eqref{p:alpha}
and $\zeta_m,m\in\N$ are i.i.d. $\R^d$-valued Pareto random variables and the components $(\zeta_m^k)_{k=1}^{d}$ of each $\zeta_m$ are i.i.d. with density
\begin{align*}
    \frac{\alpha}{2\abs{z}^{\alpha+1}}\mathds{1}_{(1,\infty)}\brac{\abs{z}}. 
\end{align*} 

Using sum of multivariate Pareto distribution to simulate multivariate stable distribution  has been proposed in \cite{davydov2002two,davydov1999limit,chenxu2023euler} due to the following facts: 1) the classical stable central limit theorem (see e.g. \cite{Xu19,CheXia19multivariate}) says that under suitable scaling, a sum of Pareto distributions converges to a stable limit; 2) unlike an $\alpha$-stable distribution whose density does not admit an analytic form, a Pareto distribution has an explicit density formula which can be more analytically tractable.


We state here the main result for the Euler-Maruyama scheme
using the Pareto noise in \eqref{equation_discretePareto}. For a Lipschitz function $f$, we will write $T_kf(x)=\E{f\brac{U^x_k}}$.

\begin{theorem}\label{theorem_eulerscheme_pareto}
Assume the stepsize $\eta$ satisfies $\eta\leq \min \left\{ 1,\frac{\theta_4}{8
\theta_1^2},\frac{1}{\theta_4}\right\} $. Then the Markov chain $\{U_k:k\in\N\}$  admits a unique invariant measure $\chi_\eta$ and it holds that 
\begin{align}
\label{estimate_distancebetweenXandU}
d_{\operatorname{Wass}}\brac{\operatorname{Law}\brac{X_{\eta N}},\operatorname{Law}\brac{U_N}}
\leq\mathcal{C}'\eta^{2/\alpha-1},
\end{align}
where
\begin{align}
\mathcal{C}'&:=\Bigg[ \Bigg(\frac{2\theta_1}{1+\frac{1}{\alpha}}\brac{\theta_1 C_3(1)\brac{\left(1+\abs{x}^2\right)^{1/2}+\frac{2C_7}{\theta_4}}+\E{\abs{L^\alpha_1}}}
+\frac{d p_\alpha}{ \sigma^{\alpha}}\nonumber\\
&+\frac{2d\alpha p_\alpha \E{\abs{L^\alpha_1}^{2-\alpha}}}{(2-\alpha)(\alpha-1)}\Bigg)\cdot\brac{\frac{2\brac{1-e^{-c_1L_0}}}{L_0}\frac{1}{C_5 }+C_2\brac{\frac{2\brac{1-e^{-c_1L_0}}}{L_0C_5}+1}} \nonumber\\
&\hspace{8em}+\frac{2\theta_1}{1+\frac{1}{\alpha}}\brac{\theta_1 C_3(1)\brac{\left(1+\abs{x}^2\right)^{1/2}+\frac{2C_7}{\theta_4}}+\E{\abs{L^\alpha_1}}}\Bigg].\label{defn:constant:C:prime}
\end{align}
Recall $\nu$ is the unique invariant measure of \eqref{SDE_alphastable}. Then it also holds that 
\begin{align}
\label{estimate_distancebetweeninvariantmeasures_pareto}
d_{\operatorname{Wass}}\brac{\chi_\eta,\nu}
\leq\mathcal{C}'\eta^{2/\alpha-1},
\end{align} 
where $\mathcal{C}'$ is defined in \eqref{defn:constant:C:prime}.
\end{theorem}

\begin{remark}
In Theorem~\ref{theorem_eulerscheme_pareto}, our approximation error bound is of the order $\eta^{2/\alpha-1}$,
which is also expected to be tight based on the discussions in \cite{chenxu2023euler}.
\end{remark}

What follows are the supporting lemmas for Theorem \ref{theorem_eulerscheme_pareto} and its proof at the end of this section. The first lemma establishes exponential ergodicity for the Euler-Maruyama discretiation scheme with Pareto noise \eqref{equation_discretePareto}. The proof is in Appendix \ref{section_proofexistenceofinvariantmeasure}. 

\begin{lemma}
\label{lemma_invariantmeasure_discretepareto}
Assume Conditions~\ref{cond_driftb} and~\ref{cond_diss} hold. The Markov chain $\{U_k:k\in\N\}$  admits a unique invariant measure $\chi_\eta$. Moreover, there exist constants $C,C'>0$ such that
  \begin{align}
  \label{estimate_exponentialerg_pareto}
        \sup_{\abs{f}\leq V_1} \abs{\E{f(U^x_k)}-\mathbb{E}_{\chi_\eta}\bracsq{f(X)}}\leq CV_1(x)e^{-C'k}. 
    \end{align}
In addition, we have the moment estimate (uniform over $k\geq 0$) 
\begin{align*}
    \E{\abs{U_k^x}}\leq \E{V_1\brac{U_k^x}}\leq  \left(1+\abs{x}^2\right)^{1/2}+\frac{2C_7}{\theta_4}, 
\end{align*}
where 
\begin{align*}
    C_7&:=\frac{d\alpha}{\sigma}\brac{\frac{1}{(2-\alpha)\sigma}+\frac{1}{\alpha-1}}\\
    &\qquad\qquad+\frac{\theta_4}{2}\brac{\eta\frac{2\abs{b(0)}^2}{\theta_4}+2\eta^2\abs{b(0)}^2+1+2\eta K}+\frac{ \abs{b(0)}^2}{\theta_4}+2\eta \abs{b(0)}^2+ K.
\end{align*}
\end{lemma}

\begin{proof}
Performing the same calculation as the one in the proof for Lemma~\ref{lemma_invariantmeasure_discreteYk}, we will arrive at 
\begin{align*}
    \mathbb{E}\bracsq{ V_1(U_1)|U_0=x}
    &\leq \brac{1-\frac{\theta_4\eta}{2}}\abs{x}+\frac{\eta^{1/\alpha}}{\sigma}\E{\abs{\zeta_1}}+\sqrt{2K\eta}+\eta\abs{b(0)}+1. 
\end{align*}

Consequently,
\begin{align*}
    \mathbb{E}\bracsq{ V_1(U_1)|U_0=x}\leq CV_1(x)+C'\mathds{1}_{A}(x),
\end{align*}
where
\begin{align*}
    C:=1-\frac{\theta_4\eta}{2}<1, \qquad C':=1+\frac{\theta_4\eta}{2}+\frac{\eta^{1/\alpha}}{\sigma}\E{\abs{\zeta_1}}+\sqrt{2K\eta}+\eta\abs{b(0)},
\end{align*}
and the compact set $A$ is given by:
\begin{align*}
    A:=\left\{ x\in \R^d:\abs{x}\leq \frac{2\brac{\frac{\eta^{1/\alpha}}{\sigma}\E{\abs{\zeta_1}}+\sqrt{2K\eta}+\eta\abs{b(0)}+1}}{\brac{1-\frac{\theta_4\eta}{2}}}\right\}. 
\end{align*}
Now one can follow \cite[Appendix A]{lihuxu2022central} to show $\{U_n:n\in\N\}$ is an irreducible Markov chain. Then via \cite[Theorem 6.3]{meyntweediestability_i}, our Markov chain is indeed ergodic and satisfies
\eqref{estimate_exponentialerg_pareto}.

Next we will obtain the moment estimate for $U_m$. We can compute that:
    \begin{align}
        V_1(U_{m+1})
        &=V_1\brac{U_m+\eta b(U_m)}+V_1\brac{U_m+\eta b(U_m)+\frac{\eta^{1/\alpha}}{\sigma}\zeta_m}-V_1\brac{U_m+\eta b(U_m)}\nonumber\\
        &=V_1\brac{U_m}+\int_0^\eta \inner{\nabla V_1(U_m+sb(U_m)),b(U_m)}ds
        \nonumber\\
        &\qquad\qquad\qquad+\brac{V_1\brac{U_m+\eta b(U_m)+\frac{\eta^{1/\alpha}}{\sigma}\zeta_m}-V_1\brac{U_m+\eta b(U_m)}}\nonumber
        \\
         &=V_1\brac{U_m}+\int_0^\eta \inner{\nabla V_1(U_m+sb(U_m)),b(U_m)}ds
        \nonumber\\
        &\qquad\qquad\qquad\qquad\qquad+ \int_0^{\frac{\eta^{1/\alpha}}{\sigma}} \inner{\nabla V\brac{U_m+\eta b(U_m)+r \zeta_m},\zeta_m }dr \nonumber
        \\
        &=:V_1\brac{U_m}+\mathcal{M}+\mathcal{N}. \label{M:N:eqn}
    \end{align}
The term $\mathcal{M}$ can be bounded in the same way as $\mathcal{A}$ in the proof of Lemma \ref{lemma_invariantmeasure_discreteYk}, yielding
\begin{align*}
    \abs{\mathcal{M}}\leq -\frac{\theta_4}{2}\eta V_1(U_m)+C(\eta)\eta,
\end{align*}
where
\begin{align*}
    C(s):=\frac{\theta_4}{2}\brac{s\frac{2\abs{b(0)}^2}{\theta_4}+2\eta^2\abs{b(0)}^2+1+2\eta K}+\frac{ \abs{b(0)}^2}{\theta_4}+2s\abs{b(0)}^2+ K.
\end{align*}
To bound the term $\mathcal{N}$ in \eqref{M:N:eqn}, we repeat the argument in \cite{chenxu2023euler}. Let $u\in \R^d$ then
\begin{align*}
    &\E{\inner{\nabla V\brac{u+\eta b(u)+r \zeta_m},\zeta_m }}\\
    &=\frac{\alpha}{2}\sum_{i=1}^d\int_{\abs{z_i}\geq 1}\inner{\nabla V\brac{u+\eta b(u)+rz^i}-\nabla V\brac{u+\eta b(u)}\mathds{1}_{(0,\eta^{1/\alpha})}(\abs{z_i}),z^i  }\\
    &\hspace{30em}\frac{dz_i}{\abs{z_i}^{\alpha+1}}\\
    &=\frac{\alpha}{2}\sum_{i=1}^d \int_{1\leq \abs{z_i}\leq \eta^{-1/\alpha}} \int_0^r \inner{\nabla^2 V\brac{u+\eta b(u)+sz^i},z^i\brac{z^i}^T }\frac{dsdz_i}{\abs{z_i}^{\alpha+1}}\\
    &\qquad\qquad+\frac{\alpha}{2}\sum_{i=1}^d \int_{ \abs{z_i}>\eta^{-1/\alpha}} \inner{\nabla V\brac{u+\eta b(u)+rz^i},z^i }\frac{dz_i}{\abs{z_i}^{\alpha+1}}. 
\end{align*}
Furthermore \eqref{estimate_gradientVlambda} says $|\nabla^2 V_1(x)|\leq 2$ and $|\nabla V^1(x)|\leq 1$ for all $x$, so that 
\begin{align*}
    &\abs{\E{\inner{\nabla V\brac{u+\eta b(u)+r \zeta_m},\zeta_m }}}\\
    &\leq \frac{\alpha}{2} \sum_{i=1}^d\brac{\int_{\abs{z_i}\leq \eta^{-1/\alpha}}\int_0^r 2\abs{z_i}^2\frac{dz_ids}{\abs{z_i}^{\alpha+1}}+\int_{\abs{z_i}\geq \eta^{-1/\alpha}} \frac{dz_i}{\abs{z_i}^{\alpha+1}} }\\
    &=\frac{2d\alpha}{2-\alpha}r\eta^{1-2/\alpha}+\frac{d\alpha}{\alpha-1}\eta^{1-1/\alpha}.
\end{align*}
Since $Y_m$ is independent from $\zeta_m$, we can write
\begin{align*}
\abs{\mathcal{N}} &\leq \int_0^{\frac{\eta^{1/\alpha}}{\sigma}} \frac{2d\alpha}{2-\alpha}r\eta^{1-2/\alpha}+\frac{d\alpha}{\alpha-1}\eta^{1-1/\alpha}dr\\
&= \frac{d\alpha}{\sigma}\brac{\frac{1}{(2-\alpha)\sigma}+\frac{1}{\alpha-1}}\eta. 
\end{align*}
Let us set 
\begin{align*}
    C_7&:=\frac{d\alpha}{\sigma}\brac{\frac{1}{(2-\alpha)\sigma}+\frac{1}{\alpha-1}}\\
    &\qquad+\frac{\theta_4}{2}\brac{\eta\frac{2\abs{b(0)}^2}{\theta_4}+2\eta^2\abs{b(0)}^2+1+2\eta K}+\frac{ \abs{b(0)}^2}{\theta_4}+2\eta \abs{b(0)}^2+ K.
\end{align*}
Then by combining the bounds on $\mathcal{M}$ and $\mathcal{N}$, we get
\begin{align*}
    \abs{V_1(U_{m+1})} \leq \brac{1-\frac{\theta_4}{2}\eta} V_1(U_m)+C_7\eta. 
\end{align*}
Performing the above procedure iteratively will lead to
\begin{align*}
    \abs{V_1(U_{m+1})}&\leq \brac{1-\frac{\theta_4}{2}\eta}^{m+1} V_1(x)+C_7\eta\sum_{i=0}^m \brac{1-\frac{\theta_4}{2}\eta}^j\\
    &\leq V_1(x)+\frac{2C_7}{\theta_4}.
\end{align*}
This completes the proof. 
\end{proof}

The following lemma quantify how well the one-step iterate of our discretization schemes tracks the original SDE \eqref{SDE_alphastable}. The result is analogous to \cite[Lemma 2.5]{chenxu2023euler}.

\begin{lemma}
\label{lemma_semigroupfirststep_pareto}
Assume $h:\R^d \to \R$ is a function satisfying 
\begin{align*}
    \sup_{y\in\mathbb{R}^{d}}\norm{\nabla h(y)}_{\operatorname{op}}<\infty \text{ and } \sup_{y\in\mathbb{R}^{d}}\norm{\nabla^2 h(y)}_{\operatorname{op}}<\infty. 
\end{align*}
Then for all $x\in \R^d$ and $\eta\in(0,1)$, it holds that
    \begin{align*}
&\abs{P_\eta h(x)-T_1h(x) }\\
&\leq \brac{\frac{\theta_1}{1+1/\alpha}\brac{2\theta_1 C_3(1)(1+\abs{x}^2)^{1/2}+2\E{\abs{L^\alpha_1}}}
+d p_\alpha \sigma^{-\alpha}+\frac{4d\alpha p_\alpha \E{\abs{L^\alpha_1}^{2-\alpha}}}{2(2-\alpha)(\alpha-1)}}\\
&\hspace{15em}\cdot\brac{\sup_{y\in \R^d}\norm{\nabla h(y)}_{\operatorname{op}}+\sup_{y\in \R^d}\norm{\nabla^2 h(y)}_{\operatorname{op}}}\eta^{2/\alpha}. 
\end{align*}
\end{lemma}

\begin{proof}
We start with the decomposition
 \begin{align*}
        \E{h\brac{X^x_\eta}-h\brac{U^x_1}}&=\E{h\brac{x+\int_0^\eta b(X^x_r)dr+L^\alpha_\eta}-h\brac{x+\eta b(x)+\frac{\eta^{1/\alpha}}{\sigma}\zeta_1 } }\\
        &=\mathcal{A}_1+\mathcal{A}_2,
 \end{align*}
where 
\begin{align*}
    \mathcal{A}_1:=\E{h\brac{x+\int_0^\eta b(X^x_r)dr+L^\alpha_\eta }-h\brac{x+\eta b(x)+L^\alpha_\eta } },
\end{align*}
and 
\begin{align*}
    \mathcal{A}_2:=\E{h\brac{x+\eta b(x)+L^\alpha_\eta}-h\brac{x+\eta b(x)} }
    -\E{h\brac{x+\eta b(x)+\frac{\eta^{1/\alpha}}{\sigma}\zeta_1 }-h\brac{x+\eta b(x)} }. 
\end{align*}
First, $\abs{\mathcal{A}_1}$ can be bounded with Condition \ref{cond_driftb} and Lemma \ref{lemma_someestimatesfromchenxueuler} as follows.
\begin{align*}
    \abs{\mathcal{A}_1}&\leq \brac{\sup_{y\in \R^d}\norm{\nabla h(y)}_{\operatorname{op}}} \E{\abs{\int_0^\eta b(X^x_r)dr-\eta b(x) } }\\
    &\leq\brac{\sup_{y\in \R^d}\norm{\nabla h(y)}_{\operatorname{op}}}\int_0^\eta \E{\abs{b(X^x_r)-b(x)}} dr\\
    &\leq \theta_1 \brac{\sup_{y\in \R^d}\norm{\nabla h(y)}_{\operatorname{op}}} \int_0^\eta \E{\abs{X^x_r-x}}dr\\
    &\leq \theta_1\brac{2\theta_1 C_3(1)(1+\abs{x}^2)^{1/2}+2\E{\abs{L^\alpha_1}}}\brac{\sup_{y\in \R^d}\norm{\nabla h(y)}_{\operatorname{op}}}\int_0^\eta r^{1/\alpha} dr\\
    &\leq \frac{\theta_1}{1+1/\alpha}\brac{2\theta_1 C_3(1)(1+\abs{x}^2)^{1/2}+2\E{\abs{L^\alpha_1}}}\brac{\sup_{y\in \R^d}\norm{\nabla h(y)}_{\operatorname{op}}}\eta^{1+1/\alpha}. 
\end{align*}
Next, let us consider $\mathcal{A}_2$. Dynkin's formula says that
\begin{align*}
    \E{h\brac{x+\eta b(x)+L^\alpha_\eta}-h\brac{x+\eta b(x)} }=\int_0^\eta \E{\Delta^{\alpha/2}h\brac{x+\eta b(x)+L^\alpha_r} }dr. 
\end{align*}
Regarding the second term in $\mathcal{A}_2$, we have
\begin{align*}
    &\E{h\brac{x+\eta b(x)+\frac{\eta^{1/\alpha}}{\sigma}\zeta_1 }-h\brac{x+\eta b(x)} }\\
    &=\frac{\eta^{1/\alpha}}{\sigma}\E{\int_0^1 \inner{\nabla h\brac{x+\eta b(x)+\frac{\eta^{1/\alpha}}{\sigma}t\zeta_1 },\zeta_1  } dt}\\
    &=\frac{\eta^{1/\alpha}}{\sigma}\sum_{i=1}^d\int_{\abs{z_i}\geq 1}\int_0^1 \inner{\nabla h\brac{x+\eta b(x)+\frac{\eta^{1/\alpha}}{\sigma}t z^i },z^i  } \frac{1}{2\abs{z_i}^{\alpha+1}}dtdz_i
    \end{align*}
By performing the change of variable $z^i\mapsto \frac{\eta^{1/\alpha}}{\sigma}z^i$, we arrive at 

    \begin{align*}
     &\E{h\brac{x+\eta b(x)+\frac{\eta^{1/\alpha}}{\sigma}\zeta_1 }-h\brac{x+\eta b(x)} }\\
     &=\frac{\eta^{1/\alpha}}{\sigma}\sum_{i=1}^d\int_{\abs{z_i}\geq \sigma^{-1}\eta^{1/\alpha}}\int_0^1 \inner{\nabla h\brac{x+\eta b(x)+t z^i },\frac{\sigma}{\eta^{1/\alpha}}z^i  } \frac{1}{2\abs{z_i}^{\alpha+1}}\brac{\frac{\eta^{1/\alpha}}{\sigma}}^{\alpha+1}dt\frac{\sigma}{\eta^{1/\alpha}}dz_i\\
    &=\frac{\alpha\eta}{2\sigma^\alpha}\sum_{i=1}^d\int_{\abs{z_i}\geq  \sigma^{-1}\eta^{1/\alpha} }\int_0^1 \inner{\nabla h\brac{x+\eta b(x)+t z^i },z^i  } \frac{1}{\abs{z_i}^{\alpha+1}}dtdz_i\\
    &=\eta \Delta^{\alpha/2} h\brac{x+\eta b(x)}-\mathcal{R},
\end{align*}
where 
\begin{align*}
   & \mathcal{R}
    :=\eta p_\alpha \sum_{i=1}^d \int_{\abs{z_i}<\sigma^{-1}\eta^{1/\alpha}} \int_0^1 \inner{\nabla h\brac{x+\eta b(x)+t z^i }-\nabla h\brac{x+\eta b(x) },z^i }\frac{1}{\abs{z_i}^{\alpha+1}}dtdz_i. 
\end{align*}

Note that in the above equation, we have used the following formula of the fractional Laplacian operator $\Delta^{\alpha/2},\alpha\in (1,2)$ (which is just \eqref{L:0:operator:stable} after an application of the mean value theorem). 
\begin{align*}
    \Delta^{\alpha/2}h(x)&= p_\alpha \sum_{i=1}^d \int_{\R} \int_0^1\brac{\inner{\nabla h\brac{x+\frac{\eta^{1/\alpha}}{\sigma}t z^i },z^i }-\inner{\nabla h\brac{x },z^i }\mathds{1}_{\{\abs{z_i}\leq 1\}} }\frac{1}{\abs{z_i}^{\alpha+1}}dtdz_i,\\
    & = p_\alpha \sum_{i=1}^d \int_{\R} \int_0^1\brac{\inner{\nabla h\brac{x+\frac{\eta^{1/\alpha}}{\sigma}t z^i },z^i }-\inner{\nabla h\brac{x },z^i }}\frac{1}{\abs{z_i}^{\alpha+1}}dtdz_i,
\end{align*}
where $p_\alpha$ is defined in \eqref{p:alpha}.
The previous expansion of terms in $\mathcal{A}_2$ leads to
\begin{align*}
    \abs{\mathcal{A}_2}\leq \abs{\mathcal{R}}+\abs{\int_0^\eta \E{\Delta^{\alpha/2}h\brac{x+\eta b(x)+L^\alpha_r}}dr-\eta \Delta^{\alpha/2}h\brac{x+\eta b(x)}   }.
\end{align*}
Furthermore, by the mean value theorem, we have
\begin{align*}
    \abs{\mathcal{R}}&\leq \eta p_\alpha \brac{\sup_{y\in \R^d}\norm{\nabla^2 h(y)}_{\operatorname{op}}}\sum_{i=1}^d \int_{\abs{z_i}<\sigma^{-1}\eta^{1/\alpha}} \frac{1}{\abs{z_i}^{\alpha-1}}dz_i=d p_\alpha \sigma^{-\alpha}\eta^2. 
\end{align*}
Moreover via Lemma \ref{lemma_fractionallaplacianfromchenxupaper} and self-similarity of $L^\alpha$, we get
\begin{align*}
    &\abs{\int_0^\eta \E{\Delta^{\alpha/2}h\brac{x+\eta b(x)+L^\alpha_r}}dr-\eta \Delta^{\alpha/2}h\brac{x+\eta b(x)}   }\\
    &\leq \int_0^\eta \E{\abs{\Delta^{\alpha/2}h\brac{x+\eta b(x)+L^\alpha_r}-\Delta^{\alpha/2}h\brac{x+\eta b(x)} }}dr\\
    &\leq \frac{4dp_\alpha  }{(2-\alpha)(\alpha-1)}\brac{\sup_{y\in \R^d}\norm{\nabla^2 h(y)}_{\operatorname{op}}} \E{\abs{L^\alpha_1}^{2-\alpha}}\int_0^\eta  r^{2/\alpha-1}dr\\
        &=\frac{4d\alpha p_\alpha \E{\abs{L^\alpha_1}^{2-\alpha}}}{2(2-\alpha)(\alpha-1)}\brac{\sup_{y\in \R^d}\norm{\nabla^2 h(y)}_{\operatorname{op}}} \eta^{2/\alpha}. 
\end{align*}
 Combining the previous estimates yields the desired result. 
\end{proof}


\begin{proof}[Proof of Theorem \ref{theorem_eulerscheme_pareto}]
We proceed similarly as in the proof of Theorem~\ref{theorem_eulerscheme}. First, 
\begin{align*}
    \E{h(X^x_{\eta N})}-\E{h(U^x_N)}
    =P_{\eta N}h(x)-T_Nh(x)=\sum_{i=1}^N T_{i-1}\brac{P_\eta-T_1}P_{(N-i)\eta} h(x). 
\end{align*}
It follows that
 \begin{align}
    \label{estimate_semigroupsdecomposition_pareto}
        &d_{\operatorname{Wass}}(\operatorname{Law}(X_{\eta N}),\operatorname{Law}(U_N))
        \nonumber\\
        &=\sup_{h\in \operatorname{Lip(1)}}\abs{P_{N\eta}h(x)-T_N h(x)}\nonumber\\
        &\leq \sup_{h\in \operatorname{Lip(1)}}\abs{T_{N-1}\brac{P_\eta-T_1}h(x)}+\sup_{h\in \operatorname{Lip(1)}}\sum_{i=1}^{N-1}\abs{T_{i-1}(P_\eta-T_1)P_{(N-i)\eta}h(x) }\nonumber\\&=:\mathcal{C}_1+\mathcal{C}_2. 
    \end{align}
Let us first bound $\mathcal{C}_1$. Notice Lemma \ref{lemma_semigroupfirststep_pareto} cannot be applied to bound $\brac{P_\eta-T_1}h(x)$ since we do not know if $\nabla^2 h(y)$ exists. However, $\brac{P_\eta-T_1}h(x)$ can still be bounded in the same way the quantity $\mathcal{A}_1$ in the proof of Lemma \ref{lemma_semigroupfirststep_pareto} is bounded. This yields
\begin{align*}
\abs{P_\eta h(x)-T_1h(x) }\leq \frac{\theta_1}{1+1/\alpha}\brac{2\theta_1 C_3(1)(1+\abs{x}^2)^{1/2}+2\E{\abs{L^\alpha_1}}}\brac{\sup_{y\in \R^d}\norm{\nabla h(y)}_{\operatorname{op}}}\eta^{1+1/\alpha}.
\end{align*}
Note also that $\eta^{1+1/\alpha}\leq \eta^{2/\alpha-1}$. Thus,
\begin{align*}
    \mathcal{C}_1 &\leq   \frac{\theta_1}{1+1/\alpha}\brac{2\theta_1 C_3(1)\E{\brac{1+\abs{U_{N-1}}^2}^{1/2}}+2\E{\abs{L^\alpha_1}}}\cdot\brac{\sup_{h\in \operatorname{Lip(1)}}\sup_{y\in \R^d}\norm{\nabla h(y)}_{\operatorname{op}}}\eta^{2/\alpha-1}\\
    &\leq   \frac{\theta_1}{1+1/\alpha}\brac{2\theta_1 C_3(1)\brac{\left(1+\abs{x}^2\right)^{1/2}+\frac{2C_7}{\theta_4}}+2\E{\abs{L^\alpha_1}}}\eta^{2/\alpha-1},
\end{align*}
the last line being a consequence of the moment estimate in Lemma \ref{lemma_invariantmeasure_discretepareto}.

Next we consider $\mathcal{C}_2$ in \eqref{estimate_semigroupsdecomposition_pareto}. Lemma \ref{lemma_semigroupfirststep_pareto} implies for any $x\in\R^d$ and $1\leq i\leq N-1$, 
\begin{align}
    \label{equation_step_i>1_pareto}
     & \sup_{h\in \operatorname{Lip(1)}}\abs{(P_\eta-T_1)P_{(N-i)\eta} h(x)}\nonumber\\
&\leq \brac{\frac{\theta_1}{1+1/\alpha}\brac{2\theta_1 C_3(1)(1+\abs{x}^2)^{1/2}+2\E{\abs{L^\alpha_1}}}
+d p_\alpha \sigma^{-\alpha}+\frac{4d\alpha p_\alpha \E{\abs{L^\alpha_1}^{2-\alpha}}}{2(2-\alpha)(\alpha-1)}}\nonumber\\
&\hspace{1em}\cdot\brac{\sup_{h\in \operatorname{Lip(1)}}\sup_{y\in \R^d}\norm{\nabla P_{(N-i)\eta}h(y)}_{\operatorname{op}}+\sup_{h\in \operatorname{Lip(1)}}\sup_{y\in \R^d}\norm{\nabla^2 P_{(N-i)\eta}h(y)}_{\operatorname{op}}}\eta^{2/\alpha}\nonumber \\
&\leq \brac{\frac{\theta_1}{1+1/\alpha}\brac{2\theta_1 C_3(1)(1+\abs{x}^2)^{1/2}+2\E{\abs{L^\alpha_1}}}
+d p_\alpha \sigma^{-\alpha}+\frac{4d\alpha p_\alpha \E{\abs{L^\alpha_1}^{2-\alpha}}}{2(2-\alpha)(\alpha-1)}}\nonumber\\
&\cdot\brac{\sup_{h\in \operatorname{Lip(1)}}\sup_{y\in \R^d}\norm{\nabla h(y)}_{\operatorname{op}}\mathcal{R}\brac{(N-i)\eta}+\sup_{h\in \operatorname{Lip(1)}}\sup_{y\in \R^d}\norm{h(y)}_{\operatorname{op}}C_2\phi\brac{(N-i)\eta} }\eta^{2/\alpha},
\end{align}
where the last line is due to Proposition \ref{prop_semigroupestimate}. We recall from \eqref{phi:t:eqn} that 
\begin{equation*} 
\phi(t)=\begin{cases} 
      1 &\text{if } 0\leq t\leq 1, \\
     \mathcal{R}(t-1) &\text{if } t>1,
   \end{cases}
\end{equation*}
and Lemma \ref{lemma_basedonjianwangpaper}  tells us $\mathcal{R}(t)$ can be taken as $\frac{2\brac{1-e^{-c_1L_0}}}{L_0}\exp\brac{-C_5t}$. 

We further observe that
\begin{align*}
   \sum_{i=1}^{N-1}\mathcal{R}\brac{(N-i)\eta}&= \frac{2\brac{1-e^{-c_1L_0}}}{L_0}\sum_{i=1}^{N-1}\exp\brac{-C_5(N-i)\eta}\\
   &\leq\frac{2\brac{1-e^{-c_1L_0}}}{L_0} \exp\brac{- C_5 N\eta}\int_1^{N}\exp\brac{\eta C_5x}dx\\
   &\leq \frac{2\brac{1-e^{-c_1L_0}}}{L_0}\frac{1}{C_5 \eta}. 
\end{align*}

For $N\geq \frac{1}{\eta}$, $\phi\brac{(N-x)\eta}$ equals $\frac{2\brac{1-e^{-c_1L_0}}}{L_0}\exp\brac{-C_5\brac{(N-x)\eta-1}}$ for $0\leq x\leq N-\frac{1}{\eta}$ and $1$ otherwise. Hence
\begin{align*}
    \sum_{i=1}^{N-1}\phi\brac{(N-i)\eta}
    &\leq \int_1^{N}\phi\brac{(N-x)\eta}dx\\
    &=\int_1^{N-\frac{1}{\eta}}\phi\brac{(N-x)\eta}dx+\int^{N}_{N-\frac{1}{\eta}}\phi\brac{(N-x)\eta}dx\\
    &=\frac{2\brac{1-e^{-c_1L_0}}}{L_0}\frac{1}{C_5\eta}\brac{1-\exp\brac{-C_5\eta(N-1)-1} }+\frac{1}{\eta}\\
    &=\frac{2\brac{1-e^{-c_1L_0}}}{L_0C_5}\frac{1}{\eta}+\frac{1}{\eta}
    \leq \brac{\frac{2\brac{1-e^{-c_1L_0}}}{L_0C_5}+1}\frac{1}{\eta}. 
\end{align*}

By combining the above calculation and Lemma \ref{lemma_invariantmeasure_discretepareto}, we can deduce from \eqref{equation_step_i>1_pareto} that 
\begin{align*}
    \mathcal{C}_2&\leq \sum_{i=1}^{N-1}\Bigg(\frac{2\theta_1}{1+\frac{1}{\alpha}}\brac{\theta_1 C_3(1)\E{\brac{1+\abs{U_{i-1}}^2}^{1/2}}+\E{\abs{L^\alpha_1}}}
+\frac{d p_\alpha}{ \sigma^{\alpha}}\\
&\qquad+\frac{2d\alpha p_\alpha \E{\abs{L^\alpha_1}^{2-\alpha}}}{(2-\alpha)(\alpha-1)}\Bigg)\cdot\Bigg(\mathcal{R}\brac{(N-i)\eta}+C_2\phi\brac{(N-i)\eta} \Bigg)\eta^{2/\alpha}\\
&\leq \sum_{i=1}^{N-1}\Bigg(\frac{2\theta_1}{1+\frac{1}{\alpha}}\brac{\theta_1 C_3(1)\brac{\left(1+\abs{x}^2\right)^{1/2}+\frac{2C_7}{\theta_4}}+\E{\abs{L^\alpha_1}}}
+\frac{d p_\alpha}{ \sigma^{\alpha}}\nonumber\\
&\hspace{5em}+\frac{2d\alpha p_\alpha \E{\abs{L^\alpha_1}^{2-\alpha}}}{(2-\alpha)(\alpha-1)}\Bigg)\Bigg(\mathcal{R}\brac{(N-i)\eta}+C_2\phi\brac{(N-i)\eta} \Bigg)\eta^{2/\alpha}\\
&\leq \Bigg(\frac{2\theta_1}{1+\frac{1}{\alpha}}\brac{\theta_1 C_3(1)\brac{\left(1+\abs{x}^2\right)^{1/2}+\frac{2C_7}{\theta_4}}+\E{\abs{L^\alpha_1}}}
+\frac{d p_\alpha}{ \sigma^{\alpha}}\\&
+\frac{2d\alpha p_\alpha \E{\abs{L^\alpha_1}^{2-\alpha}}}{(2-\alpha)(\alpha-1)}\Bigg)\cdot\Bigg(\frac{2\brac{1-e^{-c_1L_0}}}{L_0}\frac{1}{C_5 }+C_2\brac{\frac{2\brac{1-e^{-c_1L_0}}}{L_0C_5}+1} \Bigg)\nonumber\\
&\hspace{30em}\eta^{2/\alpha-1}.
\end{align*}
Summing up the bounds on $\mathcal{C}_1$ and $\mathcal{C}_2$ which appear on the right hand side of \eqref{estimate_semigroupsdecomposition_pareto} yields the desired estimate on $d_{\operatorname{Wass}}\brac{\operatorname{Law}\brac{X_{\eta N}},\operatorname{Law}\brac{U_N}}$.

Finally, like in the proof of Theorem \ref{theorem_eulerscheme}, the bound on $ d_{\operatorname{Wass}}\brac{\chi_\eta,\nu}$ can be deduced from the triangle inequality
\begin{align*}
    d_{\operatorname{Wass}}\brac{\chi_\eta,\nu}\leq d_{\operatorname{Wass}}\brac{\chi_\eta,\operatorname{Law}(U_N)}+d_{\operatorname{Wass}}\brac{\operatorname{Law}(U_N), \operatorname{Law}(X_{\eta N})}+d_{\operatorname{Wass}}\brac{\operatorname{Law}(X_{\eta N}), \nu}. 
\end{align*}
This completes the proof.
\end{proof}


\section{Proofs of Lemma~\ref{lemma_invariantmeasure_continuousX} and Lemma~\ref{lemma_invariantmeasure_discreteYk}}
\label{section_proofexistenceofinvariantmeasure}

In this Appendix, we provide the proofs of Lemma~\ref{lemma_invariantmeasure_continuousX} and Lemma~\ref{lemma_invariantmeasure_discreteYk} from the main paper.

\begin{proof}[Proof of Lemma~\ref{lemma_invariantmeasure_continuousX}]
The proof follows the same line as \cite[Proof of Proposition 1.5]{chenxu2023euler}. See also \cite[Lemma 3.1]{zhang2022ergodicity} for a similar argument.  

First, we recall the function $V_{\lambda}$ that is defined in \eqref{def_Vlambda}:
\begin{align*}
    V_{\lambda }(x)=\brac{1+\abs{x}^2}^{\lambda /2},
\end{align*}
where $\lambda \in (1,\Lambda\wedge\kappa)$. 
Since
\begin{align*}
    \nabla V_{\lambda }(x)=\frac{\lambda x}{(1+\abs{x}^2)^{\frac{2-\lambda}{2}}}, \qquad \nabla^2 V_{\lambda }(x)=\frac{\beta I_{d\times d}}{(1+\abs{x}^2)^{1-\frac{\lambda}{2}}}+\frac{\lambda(\lambda-2)xx^T}{(1+\abs{x}^2)^{2-\frac{\lambda}{2}}},
\end{align*}
we have
\begin{align}
\label{estimate_gradientVlambda}
    \abs{\nabla V_{\lambda }(x)}\leq \lambda \abs{x}^{\lambda -1},\qquad  \abs{\nabla^2 V_{\lambda }(x)}\leq \lambda (3-\lambda )\sqrt{d}. 
\end{align}
This leads to
\begin{align}
    \inner{b(x),\nabla V_{\lambda }(x)}&=\frac{\lambda \inner{b(x)-b(0)+b(0),x}}{\brac{1+\abs{x}^2}^{\frac{2-\lambda }{2}}}
    \nonumber\\
    &\leq \frac{\lambda \brac{-\theta_4\brac{\abs{x}^2+1}+\theta_4+K+\abs{b(0)}\abs{x}}}{\brac{1+\abs{x}^2}^{\frac{2-\lambda }{2}}}\nonumber\\
    &\qquad\qquad\qquad-\lambda\theta_4 V_\lambda(x)+ \frac{\lambda\brac{ \theta_4+K+\abs{b(0)}\abs{x}}}{\brac{1+\abs{x}^2}^{\frac{2-\lambda }{2}}}\nonumber\\
    &\leq \lambda \brac{-\theta_4 V_{\lambda }(x)+\theta_4+K+\abs{b(0)}\abs{x}^{\lambda -1}}\nonumber\\
    &\leq -\lambda\theta_4 V_{\lambda }(x) +\lambda (\theta_4+K)+(\lambda-1)\theta_4x^\lambda +\frac{\abs{b(0)}^\lambda}{\theta_4^{\lambda-1}}\nonumber\\
    &\leq-\lambda\theta_4 V_{\lambda }(x) +\lambda (\theta_4+K)+(\lambda-1)\theta_4 V_\lambda(x) +\frac{\abs{b(0)}^\lambda}{\theta_4^{\lambda-1}}\nonumber\\
    &\leq -\theta_4V_{\lambda }(x)+\lambda(\theta_4+K) +\frac{\abs{b(0)}^\lambda}{\theta_4^{\lambda-1}}.\label{six:lines}
\end{align}
The second line and the fourth line in \eqref{six:lines} are respectively due to Condition~\ref{cond_diss} and the simple fact that ${1+\abs{x}^2}\geq \abs{x}^2$. To get the fifth line in \eqref{six:lines}, we apply Young's inequality which says
\begin{align*}
\abs{x}^{\lambda -1}\leq a\abs{x}^{\lambda }+b, 
\end{align*}
where the constant $a$ can be any positive value and $b$ depends on $a$. The sixth line in \eqref{six:lines} is a consequence of $\abs{x}^\lambda\leq V_\lambda (x)$. 

Based on \eqref{estimate_gradientVlambda}, we also have
\begin{align}
\label{estimate_needforanotherproofofY}
    &\Delta^{\alpha/2}V_\lambda (x)\nonumber=  \sum_{i=1}^d\frac{p_\alpha}{d}\int_R \brac{V_\lambda (x+z^i)-V_\lambda(x)-\inner{\nabla V_\lambda(x),z^i}\mathds{1}_{\{\abs{z^i}\leq 1\}} }\frac{1}{\abs{z^i}^{1+\alpha}}dz_i\\
    &\quad= \sum_{i=1}^d\frac{p_\alpha}{d}\int_{\abs{z^i}\leq 1}\int_0^1\int_0^r \inner{\nabla^2V_\lambda(x+sz^i),z^i(z^i)^T}dsdr\frac{1}{\abs{z^i}^{1+\alpha}}dz_i \nonumber\\
    &\qquad\qquad\qquad+ \sum_{i=1}^d\frac{p_\alpha}{d}\int_{\abs{z^i}> 1}\int_0^1\inner{\nabla V_\lambda(x+rz^i),z^i}dr\frac{1}{\abs{z^i}^{1+\alpha}}dz_i \nonumber\\
    &\quad\leq \frac{p_\alpha \lambda(3-\lambda)\sqrt{d}}{2d}\sum_{i=1}^d\int_{\abs{z_i}\leq 1}\frac{\abs{z_i}^2}{\abs{z_i}^{1+\alpha}}dz_i+\frac{p_\alpha \lambda}{d}\sum_{i=1}^d\int_{\abs{z_i}>1}\frac{\abs{x}^{\lambda-1}\abs{z_i}+\abs{z_i}^\lambda }{\abs{z_i}^{1+\alpha}}dz_i \nonumber\\
    &\quad=\frac{2 p_\alpha \lambda(3-\lambda)\sqrt{d}}{2(2-\alpha)}+2p_\alpha \lambda\brac{\frac{\abs{x}^{\lambda-1}}{\alpha-1}-\frac{1}{\alpha-\lambda}}.
\end{align}
Then by Young's inequality,
\begin{align*}
    \abs{\Delta^{\alpha/2}V_\lambda (x)}\leq \frac{2 p_\alpha \lambda(3-\lambda)\sqrt{d}}{2(2-\alpha)}+\frac{2p_\alpha \lambda}{\alpha-\lambda}+\brac{\frac{\theta_4}{4}}^{1-\lambda} \brac{\frac{2p_\alpha}{\alpha-1}}^\lambda+\frac{\theta_4}{4}V_\lambda(x). 
\end{align*}

Combining the previous calculations, we get
\begin{align}
\label{estimate_generatorL}
    \mathcal{L}V_\lambda(x)\leq -\frac{\theta_4}{2}V_\lambda(x)+C\mathds{1}_{A}(x),
\end{align}
where
\begin{align}
\label{def_constantCformomentofX}
    C=\lambda(\theta_4+K)+\theta_4^{1-\lambda}\abs{b(0)}^\lambda +\frac{2p_\alpha \lambda(3-\lambda)\sqrt{d}}{2(2-\alpha)}+\frac{2p_\alpha\lambda }{\alpha-\lambda}+\brac{\frac{\theta_4}{4}}^{1-\lambda} \brac{\frac{2p_\alpha}{\alpha-1}}^\lambda, 
\end{align}
and the compact set $A$ is 
\begin{align*}
    A:=\left\{x\in\R^d: \abs{x}\leq (4\theta_4^{-1}C)^{1/\lambda}\right\}. 
\end{align*}

Therefore by \cite[Theorem 5.1]{meyn1993stability_iii}, the solution to Equation \eqref{SDE_alphastable} admits an invariant measure $\nu$. Furthermore, \cite[Theorem 6.1]{meyn1993stability_iii} implies \eqref{estimate_exponentialerg_X}. 

Finally we will derive the moment estimate on $X^x_t$. By Dynkins's formula, 
\begin{align*}
    \E{V_{\lambda}(X^x_t)}=V_\lambda(x)+\int_0^t \E{\mathcal{L}V_\lambda(X^x_s)}ds.
\end{align*}
The estimate in \eqref{estimate_generatorL} implies
\begin{align*}
    \frac{d}{dt}\E{V_{\lambda}(X^x_t)}\leq -\frac{\theta_4}{2}\E{V_{\lambda}(X^x_t)}+C.
\end{align*}
This differential inequality is equivalent to
\begin{align*}
    \frac{d}{dt}\brac{e^{\frac{\theta_4 t}{2}} \E{V_{\lambda}(X^x_t)}}\leq Ce^{\frac{\theta_4 t}{2}}.
\end{align*}
Integrating both sides from $0$ to $t$ gives
\begin{align*}
e^{\frac{\theta_4}{2}t} \E{V_{\lambda}(X^x_t)}-(1+x^2)^{\lambda/2}\leq \frac{2C}{\theta_4}\brac{e^{\frac{\theta_4 t}{2}}-1},
\end{align*}
and hence
\begin{align*}
    \E{V_{\lambda}(X^x_t)}\leq \frac{2C}{\theta_4}+e^{-\frac{\theta_4 t}{2}} \brac{1+\abs{x}^2}^{\lambda/2}\leq \brac{\frac{2C}{\theta_4}+1}\brac{1+\abs{x}^2}^{\lambda/2},
\end{align*}
where the constant $C$ is defined in \eqref{def_constantCformomentofX}. 
The proof is complete.
\end{proof}


\begin{proof}[Proof of Lemma~\ref{lemma_invariantmeasure_discreteYk}]
The proof follows the same line as \cite[Proposition 1.7 and Lemma 1.8]{chenxu2023euler}. We repeat it here for reader's convenience. 

To show exponential ergodicity, we will rely on \cite[Theorem 6.3]{meyntweediestability_i}. Denote $p(\eta,x)$ the density function of $L^\alpha_\eta$. Since $V_1(y)\leq \abs{y}+1$ and $Y_1=x+\eta b(x) +L^\alpha_\eta$, it follows that
\begin{align*}
    \mathbb{E}\bracsq{ V_1(Y_1)|Y_0=x}&\leq \int_{\R^d}\brac{\abs{y}+1}p\brac{\eta,y-x-\eta b(x)}dy\\
    &=\int_{\R^d}\brac{\abs{z+x+\eta b(x)}+1}p(\eta,z)dz\\
    &\leq \E{\abs{L^\alpha_\eta}} +\abs{x+\eta(b(x)-b(0))}+\eta\abs{b(0)}+1. 
\end{align*}
Notice Condition~\ref{cond_driftb} and Condition~\ref{cond_diss} imply
\begin{align*}
    \abs{x+\eta(b(x)-b(0))}^2&=\abs{x}^2+2\eta\inner{b(x)-b(0),x}+\eta^2\abs{b(x)-b(0)}^2\\
    &\leq (1-2\theta_4\eta+\theta_1^2\eta^2)\abs{x}^2+2K\eta. 
\end{align*}
Then, since $\eta\leq \min \left\{ 1,\frac{\theta_4}{8
\theta_1^2},\frac{1}{\theta_4}\right\} $, we have
\begin{align*}
    \mathbb{E}\bracsq{ V_1(Y_1)|Y_0=x}
    &\leq\brac{1-2\theta_4\eta+\theta_1^2\eta^2}^{1/2}\abs{x}+\eta^{1/\alpha}\E{\abs{L^\alpha_1}}+\sqrt{2K\eta}+\eta\abs{b(0)}+1\\
    &\leq (1-\theta_4\eta)\abs{x}+\eta^{1/\alpha}\E{\abs{L^\alpha_1}}+\sqrt{2K\eta}+\eta\abs{b(0)}+1\\
    &\leq \brac{1-\frac{\theta_4\eta}{2}}\abs{x}+\eta^{1/\alpha}\E{\abs{L^\alpha_1}}+\sqrt{2K\eta}+\eta\abs{b(0)}+1. 
\end{align*}
Observe that whenever we have $A(x)\leq C\abs{x} +C'$ for some positive constants $C,C'$, then we can write 
\begin{align*}
    A(x)\leq C\abs{x}+C'\mathds{1}_{\{C\abs{x}\leq 2C'\}}(x). 
\end{align*}

Consequently, we arrive at the estimate
\begin{align*}
    \mathbb{E}\bracsq{ V_1(Y_1)|Y_0=x}\leq CV_1(x)+C'\mathds{1}_{A}(x),
\end{align*}
where
\begin{align*}
    C:=1-\frac{\theta_4\eta}{2}<1, \qquad C':=1+\frac{\theta_4\eta}{2}+\eta^{1/\alpha}\E{\abs{L^\alpha_1}}+\sqrt{2K\eta}+\eta\abs{b(0)},
\end{align*}
and the compact set $A$ is given by:
\begin{align*}
    A:=\left\{ x\in \R^d:\abs{x}\leq \frac{2\brac{\eta^{1/\alpha}\E{\abs{L^\alpha_1}}+\sqrt{2K\eta}+\eta\abs{b(0)}+1}}{\brac{1-\frac{\theta_4\eta}{2}}}\right\}. 
\end{align*}
Now one can follow \cite[Appendix A]{lihuxu2022central} to show $\{Y_n:n\in\N\}$ is an irreducible Markov chain. Then via \cite[Theorem 6.3]{meyntweediestability_i}, our Markov chain is indeed ergodic and satisfies
\eqref{estimate_exponentialerg_Y}. 


Our next step is to show the moment estimate for $Y_n$. We can compute that
    \begin{align}
    \label{step_Vlambda_decompose}
       V_\lambda(Y_{m+1})
       &=V_\lambda\brac{Y_m+\eta b(Y_m)}+V_\lambda\brac{Y_m+\eta b(Y_m)+\eta^{1/\alpha}\xi_m}-V_\lambda\brac{Y_m+\eta b(Y_m)}\nonumber\\
        &=V_\lambda\brac{Y_m}+\int_0^\eta \inner{\nabla V_\lambda(Y_m+sb(Y_m)),b(Y_m)}ds
        \nonumber\\
        &\qquad\qquad\qquad+\brac{V_\lambda\brac{Y_m+\eta b(Y_m)+\eta^{1/\alpha}\xi_m}-V_\lambda\brac{Y_m+\eta b(Y_m)}}\nonumber
        \\
        &=:V_\lambda\brac{Y_m}+\mathcal{A}+\mathcal{B}. 
    \end{align}
Let us first consider the terms $\mathcal{A}$ on the right hand side of \eqref{step_Vlambda_decompose}. Since $\nabla V_\lambda(x)=\lambda x(1+\abs{x})^{(\lambda-2)/2}$, Condition~\ref{cond_diss} implies that
\begin{align*}
    \mathcal{A}&\leq \int_0^\eta \frac{\lambda\inner{Y_m,b(Y_m)}+\lambda s \abs{b(Y_m)}^2 }{\brac{1+\abs{Y_m+sb(Y_m)}^2}^{(2-\lambda)/2}}ds\\
    &\leq \int_0^\eta \frac{-\theta_4 \lambda\abs{Y_m}^2+\lambda K+\lambda\abs{b(0)}\abs{Y_m}  +\lambda s \abs{b(Y_m)}^2}{\brac{1+\abs{Y_m+sb(Y_m)}^2}^{(2-\lambda)/2}} ds.
\end{align*}
Condition~\ref{cond_driftb} and the fact that $\eta\leq \min\brac{1,\frac{\theta_4}{8\theta_1^2},\frac{1}{\theta_4}}$ imply 
that for any $0\leq s\leq\eta$:
\begin{align*}
    &-\theta_4 \lambda\abs{Y_m}^2+\lambda\abs{b(0)}\abs{Y_m}  +\lambda s \abs{b(Y_m)}^2\\
    &\leq -\frac{\theta_4 \lambda}{2} \lambda\abs{Y_m}^2+\frac{\lambda \abs{b(0)}^2}{\theta_4}+2\lambda s\abs{b(0)}^2+\lambda K. 
\end{align*}
Similarly,
\begin{align*}
   1\leq 1+\abs{Y_m+sb(Y_m)}^2
    &=\abs{Y_m}^2+2s\inner{Y_m,b(Y_m)}+s^2b(Y_m)^2+1\\
   &\leq \abs{Y_m}^2+s\frac{2\abs{b(0)}^2}{\theta_4}+2\eta^2\abs{b(0)}^2+1+2\eta\lambda K. 
\end{align*}
Therefore,
\begin{align*}
    &\frac{-\theta_4 \lambda\abs{Y_m}^2+\lambda K+\lambda\abs{b(0)}\abs{Y_m}  +\lambda s \abs{b(Y_m)}^2}{\brac{1+\abs{Y_m+sb(Y_m)}^2}^{(2-\lambda)/2}}\\
    &\leq -\frac{\theta_4\lambda}{2}\frac{\abs{Y_m}^2}{\brac{\abs{Y_m}^2+s\frac{2\abs{b(0)}^2}{\theta_4}+2\eta^2\abs{b(0)}^2+1+2\eta K}^{(2-\lambda)/2} }\\
    &\hspace{14em}+\frac{\lambda \abs{b(0)}^2}{\theta_4}+s2\lambda\abs{b(0)}^2+\lambda K\\
    &\leq -\frac{\theta_4\lambda}{2}\brac{\abs{Y_m}^2+s\frac{2\abs{b(0)}^2}{\theta_4}+2\eta^2\abs{b(0)}^2+1+2\eta K}^{\lambda/2} +C(s)
    \\
    &\leq \frac{\theta_4 \lambda}{2}V_\lambda(Y_m)+C(s),
\end{align*}
where 
\begin{align*}
    C(s):=\frac{\theta_4\lambda}{2}\brac{s\frac{2\abs{b(0)}^2}{\theta_4}+2\eta^2\abs{b(0)}^2+1+2\eta K}+\frac{\lambda \abs{b(0)}^2}{\theta_4}+22\lambda\abs{b(0)}^2+\lambda K.
\end{align*}

This leads to
\begin{align*}
    \mathcal{A}\leq -\frac{\theta_4\lambda}{2}\eta V_\lambda(Y_m)+C(\eta)\eta. 
\end{align*}

Now we will bound the term $\mathcal{B}$ which appears on the right hand side of \eqref{step_Vlambda_decompose}. Dynkin's formula, the estimate at \eqref{estimate_needforanotherproofofY} and Condition~\ref{cond_driftb} imply that
\begin{align*}
    &\abs{\E{V_\lambda\brac{y+\eta b(y)+L^\alpha_\eta}-V_\lambda\brac{y+\eta b(y)} }}\\
    &=\abs{\int_0^\eta\E{\Delta^{\alpha/2} V_\lambda\brac{y+\eta b(y)+L^\alpha_s}}ds }\\
    &=\int_0^\eta\frac{2 p_\alpha \lambda(3-\lambda)\sqrt{d}}{2(2-\alpha)}+2p_\alpha \lambda\brac{\frac{\E{\abs{y+\eta b(y)+L^\alpha_s}^{\lambda-1}}}{\alpha-1}-\frac{1}{\alpha-\lambda}}ds\\
    &\leq 2\lambda p_\alpha\brac{\frac{(3-\alpha)\sqrt{d}\eta}{2(2-\alpha)}+\frac{\eta}{\alpha-\lambda}+\frac{1+\theta_1^{\lambda-1}\eta}{\alpha-1}\abs{y}^{\beta-1}+\eta\abs{b(0)}^{\lambda-1}+\frac{\E{\abs{L^\alpha_1}^{\lambda-1}}\eta}{\alpha-1}} .
\end{align*}
Then by Young's inequality,
\begin{align*}
    &\abs{\E{V_\lambda\brac{Y_m+\eta b(Y_m)+\xi_m}-V_\lambda\brac{Y_m+\eta b(Y_m)} }}\\
    & \leq 2\lambda p_\alpha\cdot\brac{\frac{(3-\alpha)\sqrt{d}\eta}{2(2-\alpha)}+\frac{\eta}{\alpha-\lambda}+\frac{1+\theta_1^{\lambda-1}\eta}{\alpha-1}\E{\abs{Y_k}^{\beta-1}}+\eta\abs{b(0)}^{\lambda-1}+\frac{\E{\abs{L^\alpha_1}^{\lambda-1}}\eta}{\alpha-1}}\\
    &\leq \frac{\theta_4(\lambda-1)\eta}{2}V_\lambda(Y_m)+C'\eta,
\end{align*}
where 
\begin{align*}
    C'&:=2\lambda p_\alpha\brac{\frac{(3-\alpha)\sqrt{d}}{2(2-\alpha)}+\frac{1}{\alpha-\lambda} +\abs{b(0)}^{\lambda-1}+\frac{\E{\abs{L^\alpha_1}^{\lambda-1}}}{\alpha-1}}\\
    &\hspace{11em}+\brac{\frac{2p_\alpha(1+\theta_1^{\lambda-1})}{\alpha-1}}^\lambda\brac{\frac{2}{\theta_4}}^{\lambda-1}. 
\end{align*}
We deduce from \eqref{step_Vlambda_decompose} that
\begin{align*}
    \E{V_\lambda (Y_{m+1})}\leq \brac{1-\frac{\theta_4\eta}{2}}\E{V_\lambda(Y_m)}+(C(\eta)+C')\eta. 
\end{align*}
By doing the previous step inductively, we get 
\begin{align*}
    \E{V_\lambda (Y_{m+1})}&\leq \brac{1-\frac{\theta_4\eta}{2}}^{m+1}{V_\lambda(x)}+(C(\eta)+C')\eta\sum_{j=0}^k\brac{1-\frac{\theta_4\eta}{2}}^j\\
    &\leq V_\lambda(x)+\frac{2(C(\eta)+C')}{\theta_4}. 
\end{align*}
Finally, since $V_\lambda(x)\leq 1+\abs{x}^\lambda$, we obtain
\begin{align*}
    \E{\abs{Y_m}^\lambda}\leq \E{V_\lambda(Y_m)}\leq C_4(1+\abs{x}^\lambda),
\end{align*}
where
\begin{align*}
   C_4:={1+\frac{2(C(\eta)+C')}{\theta_4}}. 
\end{align*}
This completes the proof.
\end{proof}



\section{Proof of Lemma~\ref{lemma_basedonjianwangpaper}} 
\label{section_proofresultbyjianwang}

In this Appendix, we present the proof of Lemma~\ref{lemma_basedonjianwangpaper}.
We will adapt the argument in the reference \cite{jianwangwassersteindecay} 
to the setting of cylindrical stable L\'{e}vy processes. 

Let us recall from Remark~\ref{remark:dissipativity} 
that Condition~\ref{cond_driftb} and~\ref{cond_diss} imply the following \textit{distant dissipativity} condition:
  \begin{align}
  \label{conditionfromjianwang}
 \inner{b(x)-b(y),x-y}\leq  \begin{cases} 
      \theta_1\abs{x-y}^2 &\text{ if } \abs{x-y}\leq L_0, \\
      -\frac{\theta_4}{2}\abs{x-y}^{2} &\text{ if }\abs{x-y}> L_0,
   \end{cases}
\end{align}
where $L_0:=\sqrt{\frac{2K}{\theta_4}}$. 
We recall that $\{e^i:1\leq i\leq d\}$ are the canonical basis of $\R^d$, i.e. $e^{i}$ is a $d$-dimensional vector with $1$ in its $i$-th coordinate and $0$ elsewhere. For $z=\brac{z_1,\ldots,z_d}\in\R^d$, we can therefore write
\begin{align*}
    z^i=z_ie^i. 
\end{align*}
The generator of the process in Equation~\eqref{SDE_alphastable} in the main paper has the form:
\begin{align}
\label{def_generatorL}
    \mathcal{L}f(x)
    =\sum_{i=1}^d \int_{\R}\brac{f(x+z^i)-f(x)-\inner{\nabla f(x),z^i}\mathds{1}_{\{\abs{z_i}\leq 1\}} }\frac{p_\alpha}{\abs{z_i}^{1+\alpha}}dz_i+\inner{b(x),\nabla f(x)},
\end{align}
where $p_\alpha$ is defined in \eqref{p:alpha}.

We will rely on $\mathcal{L}$ to define a new operator which acts on elements in $C^2_b\brac{\R^{2d},\R^{d}}$. For $x,y\in \R^d$ such that $\abs{x-y}\leq L_0$ and any $a\in (0,1/2)$, let us define
\begin{align*}
    \widetilde{\mathcal{L}}f(x,y)
    &:=\frac{1}{2}\sum_{i=1}^d\bigg(  \int_{\{\abs{z_i}\leq a\abs{x_i-y_i}\}}\brac{f\brac{x+z^i,y-z^i}-f(x,y) }\frac{p_\alpha}{\abs{z_i}^{1+\alpha}}dz_i \\
    &\qquad\qquad\qquad+ \int_{\{\abs{z_i}\leq a\abs{x_i-y_i}\}}\brac{f\brac{x-z^i,y+z^i}-f(x,y)}\frac{p_\alpha}{\abs{z_i}^{1+\alpha}}dz_i\bigg)\\
    &+\sum_{i=1}^d\int_{\{\abs{z_i}>a\abs{x_i-y_i}\}}\big(f(x+z^i,y+z^i)-f(x,y)-\inner{\nabla_x f(x,y)+\nabla_y f(x,y),z^i}\\
    &\hspace{25em}\mathds{1}_{\{\abs{z_i}\leq 1\}} \big)\frac{p_\alpha}{\abs{z_i}^{1+\alpha}}dz_i\\
    &\hspace{16.5em}+\inner{b(x),\nabla_x f(x,y)}+\inner{b(y),\nabla_y f(x,y)}. 
\end{align*}
Meanwhile for any $x,y\in \R^d$ such that $\abs{x-y}>L_0$, we define
\begin{align*}
    \widetilde{\mathcal{L}}f(x,y)
    &:= \sum_{i=1}^d\int_{\R}\big(f\brac{x+z^i,y+z^i}-f(x,y)-\inner{\nabla_x f(x,y),z^i}\mathds{1}_{\{\abs{z_i}\leq 1\}} \\
    &+\inner{\nabla_y f(x,y),z^i}\mathds{1}_{\{\abs{z_i}\leq 1\}}\big)\frac{p_\alpha}{\abs{z_i}^{1+\alpha}}dz_i+\inner{b(x),\nabla_x f(x,y)}+\inner{b(y),\nabla_y f(x,y)}.
\end{align*}

Next, we show that $\widetilde{\mathcal{L}}$ coincides with $\mathcal{L}$ on $C^2_b(\R^d)$.

\begin{lemma}
\label{lemma_estimatetildeL}
$\widetilde{\mathcal{L}}$ is the coupling generator of $\mathcal{L}$, that is,
    \begin{align*}
    {\widetilde{\mathcal{L}}f(x)=\mathcal{L}f(x), \quad f\in C^2_b(\R^d). }
\end{align*}
\end{lemma}

\begin{proof}
The case when $\abs{x-y}>L_0$ is immediate. When $\abs{x-y}\leq L_0$, we have
\begin{align*}
    \widetilde{\mathcal{L}}f(x)
    &=\frac{1}{2}\sum_{i=1}^d\bigg(  \int_{\{\abs{z_i}\leq a\abs{x_i-y_i}\}}\big(f\brac{x+z^i}-f(x)-\inner{\nabla f(x),z^i}\mathds{1}_{\{\abs{z_i}\leq 1\}}\big)\frac{p_\alpha}{\abs{z_i}^{1+\alpha}}dz_i \\
    &\qquad+ \int_{\{\abs{z_i}\leq a\abs{x_i-y_i}\}}\big(f\brac{x-z^i}-f(x)+\inner{\nabla f(x),z^i}\mathds{1}_{\{\abs{z_i}\leq 1\}}\big)\frac{p_\alpha}{\abs{z_i}^{1+\alpha}}dz_i\bigg)\\
    &\qquad+\sum_{i=1}^d\int_{\{\abs{z_i}>a\abs{x_i-y_i}\}}\big(f(x+z^i)-f(x)-\inner{\nabla f(x),z^i}\mathds{1}_{\{\abs{z_i}\leq 1\}} \big)\frac{p_\alpha}{\abs{z_i}^{1+\alpha}}dz_i\\
    &\qquad\qquad\qquad\qquad\qquad\qquad\qquad\qquad\qquad\qquad\qquad\qquad\qquad+\inner{b(x),\nabla f(x)}. 
\end{align*}
By applying change of variable $z_i\mapsto -z_i$ to the integrals in the second line, we conclude that $\widetilde{\mathcal{L}}f(x)=\mathcal{L}f(x)$. 
\end{proof}

In the next step, we follow the argument in \cite[Section 2.2]{jianwangwassersteindecay}. The operator $\widetilde{\mathcal{L}}$ allows us to construct via a martingale problem a pair 
\begin{align}
\label{def_couplingprocess}
    \{X'(t),Y'(t):t\geq 0\},
\end{align}
that is a non-explosive coupling process of $\{X(t):t\geq 0\}$. This pair satisfies $X'_t=Y'_t$ for any $t> T$ where $T$ is some coupling time. Moreover, the generator of the pair $\{X'(t),Y'(t):0\leq t\leq T\}$ is $\widetilde{\mathcal{L}}$.

In the upcoming result, we prove an estimate similar to \cite[Proposition 3.1]{jianwangwassersteindecay}. Let us define the function
\begin{equation*}
    \psi(r)=\begin{cases}
        1-e^{-c_1r} &\text{ if } r\in [0,2L_0],\\
        Ae^{c_2(r-2L_0)}+B(r-2L_0)^2+(1-e^{-2c_1L_0}-A) &\text{ if } r\in (2L_0,\infty),
    \end{cases}
\end{equation*}
where $c_1$ is a positive constant greater than $1$ that will be determined later, and we also require $c_2\geq 20c_1$, which means
\begin{align}
\label{cond_c_1c_2}
    \log \frac{2(c_1+c_2)}{c_2}\geq 2.1, 
\end{align}
and moreover, 
\begin{align*}
    A:=\frac{c_1}{c_2}e^{-2L_0c_1},\qquad B:=-\frac{(c_1+c_2)c_1}{2}e^{-2L_0c_1}. 
\end{align*}
Set 
\begin{align*}
    \phi(r):=\psi\brac{\sqrt{r}}.
\end{align*}
Then for any $r\in (0,4(L_0)^2)$, we can compute that:
\begin{align*}
     \phi'(r)&=\frac{c_1}{2}\frac{e^{-c_1\sqrt{r}}}{r^{1/2}}>0,\\
     \phi''(r)&=-\frac{1}{4}c_1e^{-c_1\sqrt{r}}\brac{\frac{c_1}{r}+\frac{1}{r^{3/2}}}<0,\\
     \phi'''(r)&=\frac{1}{4}c_1e^{-c_1\sqrt{r}}\brac{\frac{c_1^2}{2r^{3/2}} +\frac{c_1}{r^2}+\frac{c_1}{2r^{3/2}}+\frac{3}{2r^{5/2}}}>0. 
\end{align*}

\begin{lemma}
\label{lemma_generatorestimate_fromjianwang}
Let us assume the condition in \eqref{conditionfromjianwang} holds. For any $x,y\in \R^d$, it holds that
    \begin{align*}
        \widetilde{\mathcal{L}}\psi\brac{\abs{x-y}}\leq -C_5 \psi\brac{\abs{x-y}},
    \end{align*}
where 
\begin{align*}
    C_5:=- e^{-2c_1\sqrt{\frac{2K}{\theta_4}}} \min \Bigg\{  &2\theta_1,\frac{\theta_4}{2}\brac{\frac{2K}{\theta_4}}^{\theta_4/2-1} ,\\
    &\qquad\qquad\frac{c_1}{8\sqrt{2}} \brac{\frac{e^{-2c_1\sqrt{\frac{2K}{\theta_4}}}}{20}+1}\frac{\theta_4^{3/2}}{K^{1/2}} \brac{\frac{2K}{\theta_4}}^{\theta_4/2-1} \Bigg\},
\end{align*}
with
\begin{align*}
    c_1:=\brac{\frac{\theta_1(2-\alpha)}{4p_\alpha}\brac{\frac{\theta_4}{2K}}^{\frac{1-\alpha}{2}}e^{-2\sqrt{\frac{2K}{\theta_4}}} }^{\frac{1}{\alpha-1}}. 
\end{align*}
\end{lemma}
\begin{proof}

In the first part of the proof, we consider the case $x,y\in\R^d$ and $\abs{x-y}\leq L_0$.
We can compute that
\begin{align}
\label{equation_LpsilessthanL_0}
     \widetilde{\mathcal{L}}\psi\brac{\abs{x-y}}
     &=\frac{1}{2}\sum_{i=1}^d\int_{\{\abs{z_i}\leq a\abs{x_i-y_i}\}}\Bigg(\phi\brac{\sum_{1\leq n\neq i \leq d}\brac{x_n-y_n}^2+\brac{x_i-y_i-2z_i}^2}\nonumber\\
     &+\phi\brac{\sum_{1\leq n\neq i \leq d}\brac{x_n-y_n}^2+\brac{x_i-y_i+2z_i}^2}-2\phi\brac{\sum_{n}\brac{x_n-y_n}^2 } \Bigg)\nonumber\\
     &\qquad\qquad\qquad\qquad\qquad\cdot\frac{p_\alpha}{\abs{z_i}^{1+\alpha}}dz_i\nonumber+\psi'\brac{\abs{x-y}}\frac{\inner{b(x)-b(y),x-y}}{\abs{x-y}}
     \nonumber
     \\
     &=:\eta_1+\eta_2.
\end{align}

Regarding the integrands of the integrals in the above equation, we define:
\begin{align*}
\Gamma_i
&=\phi\brac{\sum_{1\leq n\neq i \leq d}\brac{x_n-y_n}^2+\brac{x_i-y_i-2z_i}^2}\\
     & \qquad\qquad +\phi\brac{\sum_{1\leq n\neq i \leq d}\brac{x_n-y_n}^2+\brac{x_i-y_i+2z_i}^2}-2\phi\brac{\sum_{n}\brac{x_n-y_n}^2 },
\end{align*}
for any $i=1,2,\ldots,d$. 
We observe that
\begin{align*}
\Gamma_i
     &=\phi\brac{\sum_{n=1}^d\brac{x_n-y_n}^2+ 4z_i^2+4\abs{x_i-y_i}\abs{z_i}}\\
     & \qquad\qquad+\phi\brac{\sum_{n=1}^d\brac{x_n-y_n}^2+ 4z_i^2-4\abs{x_i-y_i}\abs{z_i}}-2\phi\brac{\sum_{n}\brac{x_n-y_n}^2 } \\
     &=\int_0^{4\abs{x_i-y_i}\abs{z_i}}\phi'\brac{\sum_{n=1}^d\brac{x_n-y_n}^2+ 4z_i^2+s }\\
     &\hspace{11em}-\phi'\brac{\sum_{n=1}^d\brac{x_n-y_n}^2+ 4z_i^2+s-4\abs{x_i-y_i}\abs{z_i} }ds\\
     &=\int_0^{4\abs{x_i-y_i}\abs{z_i}} \int_0^{4\abs{x_i-y_i}\abs{z_i}}
     \phi''\brac{\sum_{n=1}^d\brac{x_n-y_n}^2+ 4z_i^2+s+t-4\abs{x_i-y_i}\abs{z_i} }\\
     &\hspace{32em}dtds.
 \end{align*}    
     
 Since $\phi'''(r)>0$ on $\brac{0,4(L_0)^2}$ and $\abs{z_i}\leq a\abs{x_i-y_i}$ for every $i$, it follows that 
  \begin{align*}   
     \Gamma_i&\leq 16\phi''\brac{\sum_{n=1}^d\brac{x_n-y_n}^2+ 4z_i^2+4\abs{x_i-y_i}\abs{z_i} }
     \abs{x_i-y_i}^2z_i^2\\
     &= 16\phi''\brac{\sum_{1\leq n\neq i \leq d}\brac{x_n-y_n}^2+\brac{\abs{x_i-y_i}+2\abs{z_i}}^2 } \abs{x_i-y_i}^2z_i^2\\
     &\leq 16\phi''\brac{\sum_{1\leq n\neq i \leq d}\brac{x_n-y_n}^2+(1+2a)^2\abs{x_i-y_i}^2 } \abs{x_i-y_i}^2z_i^2\\
     &\leq 16 \phi''\brac{(1+2a)^2\abs{x-y}^2 }\abs{x_i-y_i}^2z_i^2.
\end{align*}
Note that in the last line, we have $a\in (0,1/2)$ and $\abs{x-y}<L_0$ so that $(1+2a)^2\abs{x-y}^2< 4(L_0)^2$. 
Consequently, the quantity $\eta_1$ in \eqref{equation_LpsilessthanL_0} can be bounded as follows. 
\begin{align}
   \eta_1&\leq 8p_\alpha \phi''\brac{(1+2a)^2\abs{x-y}^2_2}\sum_{i=1}^d\abs{x_i-y_i}^2\int_{\{\abs{z_i}\leq a\abs{x_i-y_i}\}}\abs{z_i}^{1-\alpha}dz_i\nonumber\\
  & \leq  - \frac{2p_\alpha c_1a^{2-\alpha}}{2-\alpha}e^{-c_1(1+2a)\abs{x-y}}\brac{\frac{c_1}{(1+2a)^2\abs{x-y}^2}+\frac{1}{(1+2a)^3\abs{x-y}^3} }\cdot\sum_{i=1}^d\abs{x_i-y_i}^{4-\alpha}\nonumber\\
  & \leq  - \frac{2p_\alpha c_1a^{2-\alpha}}{2-\alpha}e^{-c_1(1+2a)\abs{x-y}}\frac{c_1}{(1+2a)^2\abs{x-y}^2} d^{\frac{\alpha-2}{2}}\abs{x-y}^{4-\alpha}\nonumber\\
  &\leq - \frac{2p_\alpha c_1L_0^{1-\alpha}a^{2-\alpha}e^{-2c_1aL_0}}{2-\alpha}c_1e^{-c_1\abs{x-y}}\abs{x-y}.\label{eta:1:upper:bound}
   \end{align}
The second line in \eqref{eta:1:upper:bound} is due to $\phi''(r)=(-1/4)c_1e^{-c_1\sqrt{r}}\brac{r^{-1}+r^{-3/2}}$ on $(0,2L_0]$. The third line in \eqref{eta:1:upper:bound} is a consequence of the relation $\abs{x-y}\leq d^{(2-\alpha)/(2(4-\alpha))}\abs{x-y}_{4-\alpha}$. The last line in \eqref{eta:1:upper:bound} is due to $\abs{x-y}\leq L_0$ and $1-\alpha<0$.

Next by the condition at \eqref{conditionfromjianwang} and $\psi'(r)=c_1e^{-c_1r}$, we have
\begin{align*}
    \eta_2&\leq \psi'\brac{\abs{x-y}}  \theta_1\abs{x-y}\leq \theta_1c_1e^{-c_1\abs{x-y}}\abs{x-y}.
\end{align*}

Now by choosing {$c_1$ sufficiently large; for instance, we can take: 
\begin{align*}
    c_1:=\brac{\frac{\theta_1(2-\alpha)}{4p_\alpha}\brac{\frac{\theta_4}{2K}}^{\frac{1-\alpha}{2}}e^{-2\sqrt{\frac{2K}{\theta_4}}} }^{\frac{1}{\alpha-1}},
\end{align*}}
and let $a=1/c_1$, it follows that
\begin{align*}
    \widetilde{\mathcal{L}}\psi\brac{\abs{x-y}}&\leq \brac{-\frac{4p_\alpha L_0^{1-\alpha} e^{-2L_0}}{2-\alpha}c_1^{\alpha-1} +\theta_1 } c_1e^{-c_1\abs{x-y}}\abs{x-y}\\
    &=-2\theta_1 c_1 e^{-c_1\abs{x-y}}\abs{x-y}\\
    &\leq -2\theta_1 c_1e^{-c_1L_0}\abs{x-y}. 
\end{align*}
The last inequality is a consequence of $\abs{x-y}\leq L_0$. Moreover, notice that $\psi''(r)<0$ on $[0,2L_0]$ which implies $\psi(r)\leq \psi'(0^+)r=c_1r$. Hence
\begin{align}
\label{estimate_tildeL_part1}
    \widetilde{\mathcal{L}}\psi\brac{\abs{x-y}}\leq -2\theta_1 e^{-c_1L_0}\abs{x-y}=-2\theta_1 e^{-c_1L_0}\psi\brac{\abs{x-y}}. 
\end{align}

In the second part of the proof, we consider $x,y\in\R^d$ such that $\abs{x-y}>L_0$. When $2L_0\geq \abs{x-y}>L_0$, we have
\begin{align}
\label{estimate_tildeL_part2}
    \widetilde{\mathcal{L}}\psi\brac{\abs{x-y}}\nonumber&=\psi'\brac{\abs{x-y}}\frac{\inner{b(x)-b(y),x-y}}{\abs{x-y}}\nonumber\\
    &\leq -\frac{\theta_4}{2} \psi'\brac{\abs{x-y}}\abs{x-y}^{\theta_0-1}\nonumber\\
    &\leq -\frac{\theta_4}{2}L_0^{\theta_0-2} c_1e^{-c_1\abs{x-y}}\abs{x-y}\nonumber\\
    &\leq - \frac{\theta_4}{2}L_0^{\theta_0-2}e^{-c_12L_0}\psi\brac{\abs{x-y}}. 
\end{align}
The second line in \eqref{estimate_tildeL_part2} is due to the condition at \eqref{conditionfromjianwang} . The third line in \eqref{estimate_tildeL_part2} is due to $\psi'(r)=c_1e^{-c_1r}$ and $\abs{x-y}>L_0$. The last line in \eqref{estimate_tildeL_part2} is obtained by noticing $\psi(r)\leq c_1r$ on $[0,2L_0]$ and $\abs{x-y}<2L_0$. 

Finally, we consider the case $\abs{x-y}> 2L_0$. When $r>2L_0$, it has been shown in \cite[p. 1609]{jianwangwassersteindecay} that assumption  \eqref{cond_c_1c_2} implies for $r>2L_0$, 
\begin{align*}
    q(r)=\frac{1}{2}Ac_2e^{c_2(r-2L_0)}+2B(r-2L_0)\geq 0,
\end{align*}
and hence
\begin{align*}
    \psi'(r)=\frac{1}{2}Ac_2e^{c_2(r-2L_0)}+q(r)>\frac{1}{2}Ac_2e^{c_2(r-2L_0)}>0 . 
\end{align*}
Therefore, we can apply the condition at \eqref{conditionfromjianwang} to get for $\abs{x-y}>2L_0$
\begin{align*}\widetilde{\mathcal{L}}\psi\brac{\abs{x-y}}\nonumber&=\psi'\brac{\abs{x-y}}\frac{\inner{b(x)-b(y),x-y}}{\abs{x-y}}\nonumber\\
    &\leq -\frac{\frac{\theta_4}{2} Ac_2}{2}e^{c_2(\abs{x-y}-2L_0)} \abs{x-y}^{\theta_0-1}\nonumber\\
    &\leq -\frac{\frac{\theta_4}{2} Ac_2L_0^{\theta_0-2}}{2} e^{c_2(\abs{x-y}-2L_0)} \abs{x-y}.
\end{align*}
Now observe that $B<0$ and $0<A<c_1/c_2$, so that when $r>2L_0$, we have
\begin{align*}
    \psi(r)&=Ae^{c_2(r-2L_0)}+B(r-2L_0)^2+(1-e^{-2c_1L_0}-A)\\
    &\leq {e^{c_2(r-2L_0)}+1}   \leq {\frac{A+1}{2L_0}}re^{c_2(r-2L_0)}.
\end{align*}
 The previous calculations imply for $\abs{x-y}>2L_0$
\begin{align}
    \label{estimate_tildeL_part3}
    \widetilde{\mathcal{L}}\psi\brac{\abs{x-y}}\leq -\frac{(A+1) \theta_4 Ac_2L_0^{\theta_0-2}}{8L_0} \psi\brac{\abs{x-y}}.
\end{align}

Now we combine the estimates at \eqref{estimate_tildeL_part1}, \eqref{estimate_tildeL_part2} and \eqref{estimate_tildeL_part3}. To make things simpler, let us set $c_2=20c_1$ so that $A=(1/20)e^{-2c_1L_0}$, and recall that $L_0=\sqrt{2K/\theta_4}$. Hence, we conclude that for any $x,y\in\R^d$, 
\begin{align*}
    \widetilde{\mathcal{L}}\psi\brac{\abs{x-y}}
    &\leq - e^{-2c_1\sqrt{\frac{2K}{\theta_4}}} \min \Bigg\{ 2\theta_1,\frac{\theta_4}{2}\brac{\frac{2K}{\theta_4}}^{\theta_4/2-1} ,\\
    &\hspace{4em}\frac{c_1}{8\sqrt{2}} \brac{\frac{e^{-2c_1\sqrt{\frac{2K}{\theta_4}}}}{20}+1}\frac{\theta_4^{3/2}}{K^{1/2}} \brac{\frac{2K}{\theta_4}}^{\theta_4/2-1} \Bigg\}\psi\brac{\abs{x-y}}. 
\end{align*}
This completes the proof.
\end{proof}


\begin{proof}[Proof of Lemma \ref{lemma_basedonjianwangpaper}]
    The proof makes use of the coupling process at \eqref{def_couplingprocess} plus  Lemma~\ref{lemma_estimatetildeL} and Lemma~\ref{lemma_generatorestimate_fromjianwang}, and is exactly the same as the proof of \cite[Theorem 1.2]{jianwangwassersteindecay}. A careful reading of the proof of the aforementioned Theorem will reveal that in their Equation (3.4),
    \begin{align*}
        C(p)= \frac{1-e^{-c_1L_0}}{(L_0)^p}. 
    \end{align*}
Based on this, we deduce the constant $C_2$ which appears in \cite[p. 1613]{jianwangwassersteindecay} when $p=1$ is given by
\begin{align*}
    C_2=2C(1)= \frac{2\brac{1-e^{-c_1L_0}}}{L_0}. 
\end{align*}
The proof is complete.
\end{proof}

	\section{Malliavin calculus on Poisson space}
	\label{appendix_malliavin}
In this Appendix, we introduce the Malliavin calculus developed in \cite{kulik2023gradient}. The aforementioned paper adapts ideas of the classical work \cite{Bichtelerbook, basscranston1986malliavin, norris1988integration} to the setting of cylindrical L\'{e}vy processes.  Unless specified otherwise, Condition~\ref{cond_levymeasure} and Condition~\ref{cond_driftb} are the standing assumptions throughout the Appendix. 

For a $d\times d$ matrix $A$ with entries $a_{ij},1\leq i,j\leq d$, let us write $\abs{A}:=\sqrt{\sum_{i=1}^d \abs{a_{ij}}^2}$ which is the Frobenius norm of $A$.

Recall the setup of Section \ref{sec:cylindrical}. For each $j$, let $N_j$ be a Poisson random measure on $E:= \R\times [0,\infty)$ with intensity measure $m_j(d\xi)ds$, with $m_j$ being specified in Condition~\ref{cond_levymeasure}. Let $N$ be the Poisson random measure that is the product of $N_j$ with the intensity measure: 
\begin{align*}
    m(d\xi)ds=\prod_{j=1}^d m_j(d\xi_j)ds.
\end{align*}
Then due to independence of $Z^j$'s, we have the representation 
\begin{align*}
    Z_t=\int_0^t\int_{\R^d}\xi \overline{N}(d\xi,ds),
\end{align*}
where
\begin{align*}
    \overline{N}(d\xi, ds)&:=\widehat{N}(d\xi, ds)\mathbbm{1}_{\{\abs{\xi}\leq 1\}}+{N}(d\xi, ds)\mathbbm{1}_{\{\abs{\xi}\geq 1\}},\\
    \widehat{N}(d\xi, ds)&:={N}(d\xi, ds)-ds m(d\xi). 
\end{align*}

Consider the $\Lambda$-algebra
\begin{align*}
    \mathfrak{F}_t:=\Lambda\brac{N([0,s]\times \Gamma):0\leq s\leq t, \Gamma\in \mathcal{B}(\R^d)}.
\end{align*}
Then the Poisson random measure $N$ can be viewed as a random element in the space $\mathbb{Z}(E)$ of integer-valued measures on $(E,\mathcal{B})$.

A random variable  $F$ is said to be an $L^0$-functional of $N$ if there exists a sequence of bounded measurable function $f_m:\mathbb{Z}(E)\to \R$ such that the following convergence holds in probability: 
\begin{align*}
    F=\plim_{m\to\infty}f_m(N).
\end{align*}

Next, let us introduce the field $V=(V_1,\ldots,V_d)$ whose $j$-component satisfies
\begin{align}
    \label{def_fieldV}
    V_j(t,\xi_j)=\phi (\xi_j)\psi(t),
\end{align}
where $\psi\in \mathcal{C}^\infty(\R)$ and $\phi \in \mathcal{C}^\infty\brac{\R\setminus \{0\}}$ are non-negative functions such that 
\[ \psi(t):=\begin{cases} 
      0 &\text{if } \abs{t}\geq \delta, \\
     1 &\text{if } \abs{t}\leq \delta/2,
   \end{cases}
\]
where $\delta\in [0,R]$ is a small value and is chosen as in \cite[Proof of Lemma 6]{kulik2023gradient}, and 
\begin{align*}
    \phi(\xi_j):=\abs{\xi_j}^\kappa \psi(\xi_j). 
\end{align*}

Consider the following perturbation of elements in $\mathbb{Z}(E)$:
\begin{align*}
    Q^\epsilon_k\brac{\sum_{j=1}^d \delta_{(t_j,\xi_j)}}=\sum_{j=1}^d \delta_{\brac{t_j,\xi_j+\epsilon V_k(t_j,\xi_j)e^k}},
\end{align*}
where $\{e^k:1\leq k\leq d\}$ is the canonical basis of $\R^d$. 

 For a $L^0$-functional $F$, we write
\begin{align*}
    Q^\epsilon_kF=\plim_{m\to\infty} f_m\brac{ Q^\epsilon_k(N)}, 
\end{align*}
where the existence and well-posedness of $Q^\epsilon_kF$ is shown in \cite[Lemma 2]{kulik2023gradient}. Then the Malliavin derivative of $F$ in the direction $V_ke_k$ is
\begin{align*}
   D_kF=\plim_{\epsilon\to 0}\brac{Q^\epsilon_k(F)-F}.
\end{align*}
If every direction has such a limit, then $F$ is said to be differentiable. We will also write
\begin{align*}
    \der F=\brac{D_1F,\ldots, D_dF},
\end{align*}
which is the Malliavin derivative of $F$ with respect to the field $V=(V_1,\cdots,V_d)$.

The following chain rule of $\der$ will be useful.
\begin{lemma}
    Assume $F_1,\cdots,F_m$ are differentiable $L^0$-functionals of $N$. Then for any $g\in\mathcal{C}^1_b$, $g\brac{F_1,\cdots,F_m}$ is differentiable and 
    \begin{align*}
        D_k g\brac{F_1,\cdots,F_m}=\sum_{j=1}^m \nabla_j g\brac{F_1,\cdots,F_m}D_kF_j. 
    \end{align*}
\end{lemma}

Next, we state a key result that is a Bismut-Elworthy-Li formula established by Kulik, Peszat and Priola \cite{kulik2023gradient}.

\begin{proposition}[\cite{kulik2023gradient}]
\label{prop_bismut}
Assume only \eqref{cond_pi},\eqref{cond_intkappa},\eqref{cond_int_2kappaminus2},\eqref{cond_inttwokappa} in Condition~\ref{cond_levymeasure} and \eqref{originalconddirftb} in Condition~\ref{cond_driftb}. Then for any $f\in \mathcal{C}_b(\R^d)$, there exists $G(x,t)$ such that
\begin{align*}
    \nabla P_tf(x)= \E{f(X^x_t)G(x,t)}.
\end{align*}
The random field $G(x,t)=(G_1(x,t),\ldots,G_d(x,t))$ satisfies
    \begin{align*}
    G_j(x,t)=\sum_{k=1}^d \left(A_{k,j}(x,t)B_k(t)-D_kA_{k,j}(x,t)\right),
\end{align*}
where the entries of the $\R^{d\times d}$-valued random field $A(x,t)$ and $\R^d$-valued random field $B(t)$ are given by
\begin{align*}
    A_{k,j}(x,t)&= \bracsq{(\der X^x_t)^{-1}\,\nabla X^x_t}_{k,j},\\
     B_k(t)&= -\int_0^t\int_{-R}^R\frac{\frac{d}{d\xi^k} \brac{V_k(s,\xi_k)\rho_k(\xi_k)}}{\rho_k(\xi_k)}\widehat{N}_k(ds,d\xi).
\end{align*}
Moreover, for any $t\geq 0$, $A_{k,j}(x,t)$ is $p$-integrable for $p\geq 1$ and $B_k(t)$ is $q$-integrable for $2\geq q\geq 1$. 
\end{proposition}

\begin{proof}
    Refer to Theorem~1, Lemma~3 and Lemma~5 in \cite{kulik2023gradient}. 
\end{proof}

The goal of this Appendix is to prove the following integrability result. Since we aim to obtain explicit bounds, the proof requires long and tedious calculations and is therefore postponed to Appendix~\ref{section_proofoflemmaGandnablaG}. 

\begin{lemma}
\label{lemma_GandnablaG}
 Assume Condition~\ref{cond_levymeasure} and Condition~\ref{cond_driftb}. For any $t\geq 0$, the $\R^d$-valued random variable $G(x,t)$ is differentiable in $x$ and $q$-integrable for $\tau>q\geq 1$. The same properties hold for the $d\times d$ random matrix $\nabla G(x,t)$.

 Specifically, we have for any $q_0\in (q,\tau)$
 \begin{align*} 
\sup_{x\in\R^d}\E{\abs{G(x,t)}^q}\leq 2^{q-1}\brac{\frac{q}{q_0}\mathcal{Q}^0(q_0,t)+\frac{q_0-q}{q_0} \mathcal{Q}^9\brac{\frac{qq_0}{q_0-q},t}+\sum_{k=1}^d\mathcal{Q}^{11}_k(q,t)},
    \end{align*}
and 
\begin{align*}
&\sup_{x\in \R^d}   \E{ \abs{\nabla G(x,t)}^q}\\
&\leq 2^{q-1} \brac{\sum_{m=1}^d \brac{\frac{q}{q_0}\mathcal{Q}^0(q_0,t)+\frac{q_0-q}{q_0}\mathcal{Q}^{10}_m\brac{\frac{qq_0}{q_0-q},t}+\sum_{k=1}^d\mathcal{Q}^{12}_{k,m}(q,t) }^2}^{\frac{1}{2}}, 
\end{align*}
where the factor $Q_i$'s are defined in Appendix~\ref{section_proofoflemmaGandnablaG}, with $\mathcal{Q}^0$ in Lemma~\ref{lemma_Bk} and $\mathcal{Q}^{9},\mathcal{Q}^{10}, \mathcal{Q}^{11},\mathcal{Q}^{12}$ in Lemma~\ref{lemma_AandDAand_nablaDA}. 

Furthermore, regarding dimension dependence, our upper bound on $\sup_{x\in\R^d}\E{\abs{G(x,t)}^q}$ is of the order
\begin{align*}
\mathcal{O}\left( d^{\frac{qq_0}{2(q_0-q)}}\vee d^{\frac{3}{2}+\brac{\frac{q\tau}{2(\tau-q)}\vee \tau}}\right),
\end{align*}
as $d\rightarrow\infty$, 
while our upper bound on $\sup_{x\in\R^d}\E{ \abs{\nabla G(x,t)}^q}$ is of the order 
\begin{align*}
\mathcal{O}\left(d^{\frac{4qq_0}{q_0-q}+\frac{5}{2}}\vee d^{q_0+\frac{7}{2}-\frac{q_0}{\tau}}\vee d^{\frac{q\tau}{q(\tau-q)}+\frac{5}{2}}\vee d^{\tau+\frac{5}{2}}\right), 
\end{align*}
as $d\rightarrow\infty$.
\end{lemma}


\section{Proof of Lemma~\ref{lemma_GandnablaG}}\label{section_proofoflemmaGandnablaG}

In this Appendix, we provide the proof of Lemma~\ref{lemma_GandnablaG}.

First, recall that for a $d\times d$ matrix $A$ with entries $a_{ij},1\leq i,j\leq d$, we write $\abs{A}=\sqrt{\sum_{i=1}^d \abs{a_{ij}}^2}$ which is the Frobenius norm of $A$.

One technical tool we will make frequent use in this Appendix is a maximal inequality for Poisson stochastic integrals known as Kunita's inequality (see \cite[Theorem 4.4.23]{applebaum2009levy} or \cite[Proposition 2.6.1]{kunitabook2019stochastic}). Statements of the Kunita's inequality in the aforementioned references do not include an explicit bounding constant; however as can be seen from their proofs, obtaining an explicit bounding constant does not require much extra work. We repeat the proof below for readers' convenience. 

Let $c>0$ and $t\in[0,T]$. $E$ is the open ball with radius $c$ centered at the origin. Further let $H=\brac{H^1,\ldots,H^d}$ such that each $H^i(\xi,t):E\times [0,T]\times \Omega\to\R$ is a predictable mapping and $\mathbb{P}\brac{\int_0^T\int_E\abs{H(\xi,t)}m(d\xi)dt<\infty}=1$. We can define the Poisson stochastic integrals
\begin{align*}
  {I}(t)=\brac{{I}_1(t),\ldots,{I}_d(t)}, \qquad  \hat{I}(t)=\brac{\hat{I}_1(t),\ldots,\hat{I}_d(t)},
\end{align*}
where for every $i=1,2,\ldots,d$:
\begin{align*}
   {I}_i(t)=\int_0^t\int_EH_i(\xi,s){N}(d\xi,ds),\qquad \hat{I}_i(t)=\int_0^t\int_EH_i(\xi,s)\widehat{N}(d\xi,ds). 
\end{align*}
\begin{lemma}
    \label{lemma_kunita}
    For any $p\geq 2$, it holds that
\begin{align*}
    \E{\sup_{0\leq s\leq t} \abs{\hat{I}(s)}^p}
    \leq C_6(p)\brac{\E{\brac{\int_0^t\int_E \abs{H(\xi,s)}^2m(d\xi)ds}^{\frac{p}{2}}}
    +\E{\int_0^t\int_E \abs{H(\xi,s)}^p m(d\xi)ds }},
\end{align*}
and 
\begin{align*}
    \E{\sup_{0\leq s\leq t} \abs{{I}(s)}^p}
    &\leq C_6(p)\brac{\E{\brac{\int_0^t\int_E \abs{H(\xi,s)}^2m(d\xi)ds}^{\frac{p}{2}}}
    +\E{\int_0^t\int_E \abs{H(\xi,s)}^p m(d\xi)ds }}\\
    &\qquad\qquad\qquad+\E{\brac{\int_0^t\int_E a\abs{H(\xi,s)}m(d\xi)ds}^{p}},
\end{align*}
where the constant factor is 
\begin{align*} C_6(p):= \begin{cases} 
      2 & p=2, \\
     \max\left\{ {2p^4\brac{2^{2p-3}p^2}^{\frac{p}{2(p-2)}}},
    {2^{2p-3}p^7}\right\}\frac{1}{p^5-(p-2)2^{3-2p}} & p>2.
   \end{cases}
\end{align*}

\end{lemma}
\begin{proof}
Let us first show the result for $\hat{I}(t)$. The case $p=2$ is a direct consequence of the It\^{o} isometry of Poisson stochastic integrals, so we can move on and consider the case $p>2$. We will follow the steps in \cite[Proof of Theorem 4.4.23]{applebaum2009levy}. Let us write via It\^{o}'s formula
 \begin{align*}
     \abs{\hat{I}(t)}^p=M(t)+A(t),
 \end{align*}
 where 
 \begin{align*}
     M(t):=\int_0^1\int_E \left(\abs{\hat{I}(s-)+H(\xi,s)}^p-\abs{\hat{I}(s-)}^p\right)\widehat{N}(d\xi,ds),
 \end{align*}
and 
 \begin{align*}
     A(t)&:=\int_0^t\int_E \bigg(\abs{\hat{I}(s-)+H(\xi,s)}^p\\
     &\qquad\qquad\qquad-\abs{\hat{I}(s-)}^p-p\abs{\hat{I}(s-)}^{p-2}\sum_{i=1}^d\hat{I}_i(s-)H^i(\xi,s)\bigg) m(d\xi)ds. 
 \end{align*}
$\{M(t):t\geq 0\}$ is a local martingale but without loss of generality let us assume it is a martingale, noting that we can reduce the general case to this one by constructing an appropriate sequence of stopping times. 

Let $\theta_i\in (0,1)$ for $1\leq i\leq d$ and let $J(\hat{I},H;\theta)$ be the $\R^d$-valued process whose $i$-th component has the value $\hat{I}_i(s-)+\theta_iH_i(\xi,s)$ at $s$. By Taylor's theorem, there exist $\theta_i$'s for which
\begin{align*}
    A(t)&=\int_0^t\int_E \bigg(\frac{1}{2}p(p-2)\abs{J(\hat{I},H;\theta)(s)}^{p-4}\inner{J(\hat{I},H;\theta)(s),H(\xi,s)}^2 \\
    &\qquad\qquad\qquad+p\abs{J(\hat{I},H;\theta)(s)}^{p-2}\abs{H(\xi,s)}^2\bigg)m(d\xi)ds.
\end{align*}
By Cauchy-Schwarz inequality and the fact that $\abs{a+b}^p\leq 2^p\brac{\abs{a}^p+\abs{b}^p}$, we get
\begin{align*}
    \abs{A(t)}\leq p^2 2^{p-3}\int_0^t\int_E \left(\abs{\hat{I}(s-)}^{p-2}\abs{H(\xi,s)}^2+\abs{H(\xi,s)}^p \right)m(d\xi)ds. 
\end{align*}
Then via Doob's martingale inequality, 
\begin{align}
    \E{\sup_{0\leq s\leq t} \abs{\hat{I}(s)}^p}& \leq \brac{\frac{p}{p-1}}^p p^2 2^{p-3} \E{\int_0^t\int_E \abs{\hat{I}(s-)}^{p-2}\abs{H(\xi,s)}^2 m(d\xi)ds }\nonumber\\
    &\qquad\qquad+\brac{\frac{p}{p-1}}^p p^2 2^{p-3}\E{\int_0^t\int_E \abs{H(\xi,s)}^p m(d\xi)ds}.\label{K:1:defn}   
\end{align}
Denote the first term on the right hand side in \eqref{K:1:defn} by $K_1(t)$. Assume $a$ is some constant greater than 1 to be determined later. Then via H\"{o}lder's inequality followed by Young's inequality, we can compute that
\begin{align*}
    K_1(t)
    &\leq \brac{\frac{p}{p-1}}^p p^2 2^{p-3}\E{\sup_{0\leq s\leq t}\frac{1}{a}\abs{\hat{I}(s-)}^{p-2}\int_0^t\int_E a\abs{H(\xi,s)}^2 m(d\xi)ds }\\
    &\leq \brac{\frac{p}{p-1}}^p p^2 2^{p-3} a^{2-p}\E{\brac{\sup_{0\leq s\leq t}\left|\hat{I}(s-)\right| }^p}^{p-\frac{2}{p}}\\
    &\hspace{13em}\cdot\E{\brac{\int_0^t\int_E a\abs{H(\xi,s)}^2m(d\xi)ds }^{\frac{p}{2}}}^{\frac{2}{p}}\\
    &\leq \brac{\frac{p}{p-1}}^p(p-2)p 2^{p-3} a^{2-p}\E{\sup_{0\leq s\leq t} \left|\hat{I}(s)\right|^p}+\frac{2}{p}a^{\frac{p}{2}}\\
    &\hspace{13em}\cdot\E{\brac{\int_0^t\int_E \abs{H(\xi,s)}^2m(d\xi)ds}^{\frac{p}{2}}}.
\end{align*}
This leads to
\begin{align*}
    \E{\sup_{0\leq s\leq t} \abs{\hat{I}(s)}^p}&\leq \brac{\frac{p}{p-1}}^p(p-2)p 2^{p-3} a^{2-p}\E{\sup_{0\leq s\leq t} \left|\hat{I}(s)\right|^p}\\
    &\qquad\qquad+\frac{2}{p}a^{\frac{p}{2}}\E{\brac{\int_0^t\int_E a\abs{H(\xi,s)}^2m(d\xi)ds}^{\frac{p}{2}}}\\
    &\qquad\qquad\qquad+\brac{\frac{p}{p-1}}^p p^2 2^{p-3}\E{\int_0^t\int_E \abs{H(\xi,s)}^p m(d\xi)ds }. 
\end{align*}
Now if the constant $a$ is large enough such that $\brac{\frac{p}{p-1}}^p(p-2)p 2^{p-3} a^{2-p}<1$, then the proof is complete. Note that $\frac{p}{p-1}<2$ for $p>2$, so we can  set 
\begin{align*}
    a=\brac{2^{2p-3}p^2}^{\frac{1}{p-2}}.
\end{align*}
To obtain the maximal inequality for $I(t)$, we write 
\begin{align*}
    I(t)=\hat{I}(t)+\int_0^t\int_E H(\xi,s)m(d\xi)ds,
\end{align*}
and apply the previous maximal inequality for $\hat{I}(t)$. 
The proof is complete.
\end{proof}

Beside from Kunita's inequality, we will often use the following version of Gr\"{o}nwall's inequality: assume $\alpha, \beta$ and $u$ are real valued function on $[a,b]$ such that $\beta,u$ are continuous and $\alpha$ is non-decreasing. If they satisfy for all $t\in [a,b]$
\begin{align*}
    u(t)\leq \alpha(t)+\int_a^t \beta(s) u(s)ds,
\end{align*}
then 
\begin{align}
\label{gronwall}
    u(t)\leq \alpha(t) \exp\brac{\int_0^t \beta(s)ds}. 
\end{align}
Another technical tool that we need is Jensen's inequality for Lebesgue integrals: let $\phi$ be a convex function $\R \mapsto \R$ and $f$ be a non-negative integrable function on $[a,b]$. Then
\begin{align}
\label{jensen}
    \phi\brac{\frac{1}{b-a}\int_a^b f(s)ds }\leq \frac{1}{b-a}\int_a^b \phi\brac{f(s)}ds. 
\end{align}
Furthermore, we will require Young's inequality which is 
\begin{align}
    \label{younginequality}
    xy\leq \frac{x^a}{a}+\frac{x^b}{b},
\end{align}
for $x,y\geq 0$ and $a,b>1;\frac{1}{a}+\frac{1}{b}=1$. 
Finally, we will use 
\begin{align}
\label{minkowskitype_inequality}
(x+y)^p\leq x^p+y^p,
\end{align}
for $x,y\geq 0$ and $0\leq p\leq 1$. 

Now we proceed to proving various results that will lead to the proof of Lemma~\ref{lemma_GandnablaG}. Since we impose stricter conditions compared to \cite{kulik2023gradient}, we can strengthen the result on $q$-integrability of $B_k(t)$ in Proposition~\ref{prop_bismut} from \cite{kulik2023gradient} as follows.  

\begin{lemma}
\label{lemma_Bk}
 Assume Condition~\ref{cond_levymeasure}. For any $t\geq 0$, the $\R^d$-valued random variable $B(t)$ is $q$-integrable for $1\leq q\leq \tau$.
 Specifically, we have
    \begin{align*}
       \mathcal{Q}^0(q,t)&=  \E{\abs{B(t)}^q}\\
        &\leq \sum_{k=1}^d d^{\frac{q}{2}-\frac{q}{\tau}}\bigg(2^{q-\frac{q}{\tau}}C_6(\tau)^{\frac{q}{\tau}}\Big(t^{\frac{q}{2}}M_k(2\kappa-2)^{\frac{q}{2}}+t^{\frac{q}{\tau}}M_k(\tau(\kappa-1))^{\frac{q}{\tau}}
        \\
        &\qquad\qquad
        +t^{\frac{q}{2}}\overline{M}_k(2\kappa)^{\frac{q}{2}}+t^{\frac{q}{\tau}}\overline{M}_k(\tau \kappa)^{\frac{q}{\tau}}\Big)
        +t^q \brac{ M_k(\kappa-1)^{\frac{q}{\tau}}+\overline{M}_k(\kappa)}^q\bigg),
    \end{align*}
where the factor $C_6$ is defined in Lemma~\ref{lemma_kunita} and the remaining notations are given 
in Condition~\ref{cond_levymeasure}. 

Moreover regarding dimension dependence, our upper bound on $\mathcal{Q}^0(q,t)$ is of the order $\mathcal{O}(d^{\frac{q}{2}-\frac{q}{\tau}+1})$ as $d\rightarrow\infty$. 
\end{lemma}

\begin{proof}
    It is sufficient to prove the statement for $q=\tau$. Per our assumption $\tau>2$, hence we can apply Kunita's inequality in Lemma~\ref{lemma_kunita} to get     
    \begin{align*}
        &\E{\abs{B_k(t)}^\tau}\\
        &=\E{\abs{\int_0^t\int_{-\delta}^\delta\frac{\phi'(\xi_k)\rho_k(\xi_k)+\phi(\xi_k)\rho'_k(\xi_k)}{\rho_k(\xi_k)}\widehat{N}_k(d\xi_k, ds)}^\tau}
        \\&\leq C_6(\tau)\brac{\int_0^t\int_{-\delta}^\delta \abs{\frac{\phi'(\xi_k)\rho_k(\xi_k)+\phi(\xi_k)\rho'_k(\xi_k)}{\rho_k(\xi_k)}}^2m_k(\xi_k) ds}^{\frac{\tau}{2}}\\
        &\qquad\qquad+C_6(\tau)\int_0^t\int_{-\delta}^\delta\abs{\frac{\phi'(\xi_k)\rho_k(\xi_k)+\phi(\xi_k)\rho'_k(\xi_k)}{\rho_k(\xi_k)}}^\tau m_k(\xi_k) ds\\
        &\qquad\qquad\qquad+{\brac{\int_0^t\int_{-\delta}^\delta\abs{\frac{\phi'(\xi_k)\rho_k(\xi_k)+\phi(\xi_k)\rho'_k(\xi_k)}{\rho_k(\xi_k)}} m_k(\xi_k) ds}^\tau}\\
        &\leq 2^{\tau-1}C_6(\tau)\cdot\brac{ \brac{\int_0^t\int_{-\delta}^\delta \abs{\phi'(\xi_k)}^2m_k(\xi_k) ds}^{\frac{\tau}{2}}+  \brac{\int_0^t\int_{-\delta}^\delta \abs{\phi(\xi_k)\frac{\rho_k'(\xi_k)}{\rho_k(\xi_k)}}^2m_k(\xi_k) ds}^{\frac{\tau}{2}}}\\
        &\qquad+2^{\tau-1} C_6(\tau)\cdot\brac{{\int_0^t\int_{-\delta}^\delta\abs{\phi'(\xi_k)}^\tau m_k(\xi_k) ds}+{\int_0^t\int_{-\delta}^\delta\abs{\phi(\xi_k)\frac{\rho_k'(\xi_k)}{\rho_k(\xi_k)}}^\tau m_k(\xi_k) ds}}\\
        &\qquad\qquad+{\brac{{\int_0^t\int_{-\delta}^\delta\abs{\phi'(\xi_k)} m_k(\xi_k) ds}+{\int_0^t\int_{-\delta}^\delta\abs{\phi(\xi_k)\frac{\rho_k'(\xi_k)}{\rho_k(\xi_k)}} m_k(\xi_k) ds}}}^\tau\\
        &\leq 2^{\tau-1}C_6(\tau)\brac{t^{\frac{\tau}{2}}M_k(2(\kappa-1))^{\frac{\tau}{2}}+tM_k(\tau(\kappa-1))+t^{\frac{\tau}{2}}\overline{M}_k(2\kappa)^{\frac{\tau}{2}}+t\overline{M}_k(\tau \kappa) }\\
        &\qquad\qquad\qquad+t^\tau \brac{ M_k(\kappa-1)+\overline{M}_k(\kappa)}^\tau. 
    \end{align*}
The second to last inequality is due to Minskowski's inequality and $\abs{f+g}^p\leq 2^{p-1}\brac{\abs{f}^p+\abs{g}^p}$. The last line is due to the definition of $\phi$ and Condition~\ref{cond_levymeasure}. In particular, (6) in Condition~\ref{cond_levymeasure} and Remark 2.1 imply $M_k(\tau(\kappa-1))\leq M_k(2(\kappa-1))<\infty$, while (7) and (8) in Condition~\ref{cond_levymeasure} guarantee $\overline{M}_k(2\kappa), \overline{M}_k(\tau \kappa) <\infty$. Also H\"{o}lder's inequality and (6), (7) in Condition~\ref{cond_levymeasure} imply $M_k(\kappa-1)\leq M_k(2\kappa-2)<\infty$ and $\overline{M}_k(\kappa)\leq \overline{M}_k(2\kappa)<\infty.$

Next, we apply Jensen's inequality to the function $\abs{x}\mapsto \abs{x}^{\frac{\tau}{2}}$ to get
 \begin{align*}
     \E{\abs{B(t)}^\tau}&=\E{\brac{\sum_{k=1}^d \abs{B_k(t)}^2}^{\frac{\tau}{2}}}
     \\
     &\leq d^{\frac{\tau}{2}-1}\sum_{k=1}^d \E{\abs{B_k(t)}^\tau}\\
     &\leq \sum_{k=1}^d d^{\frac{\tau}{2}-1}\bigg(2^{\tau-1}C_6(\tau)\Big(t^{\frac{\tau}{2}}M_k(2(\kappa-1))^{\frac{\tau}{2}}+tM_k(\tau(\kappa-1))
     \\
     &\qquad\quad
     +t^{\frac{\tau}{2}}\overline{M}_k(2\kappa)^{\frac{\tau}{2}}+t\overline{M}_k(\tau \kappa)\Big)
     +t^\tau \brac{ M_k(\kappa-1)+\overline{M}_k(\kappa)}^\tau\bigg). 
 \end{align*}
Hence, for $q\in [1,\tau]$, Lyapunov's inequality and \eqref{minkowskitype_inequality} imply that
 \begin{align*}
       \E{\abs{B(t)}^q}&\leq \E{\abs{B(t)}^\tau}^{q/\tau}\\
          &\leq \Bigg(\sum_{k=1}^d d^{\frac{\tau}{2}-1}\bigg(2^{\tau-1}C_6(\tau)\Big(t^{\frac{\tau}{2}}M_k(2(\kappa-1))^{\frac{\tau}{2}}+tM_k(\tau(\kappa-1))
          \\
          &\qquad\quad+t^{\frac{\tau}{2}}\overline{M}_k(2\kappa)^{\frac{\tau}{2}}+t\overline{M}_k(\tau \kappa)\Big)
          +t^\tau \brac{ M_k(\kappa-1)+\overline{M}_k(\kappa)}^\tau\bigg)\Bigg)^{\frac{q}{\tau}}\\
          &\leq \sum_{k=1}^d d^{\frac{q}{2}-\frac{q}{\tau}}\Bigg(2^{q-\frac{q}{\tau}}C_6(\tau)^{\frac{q}{\tau}}\bigg(t^{\frac{q}{2}}M_k(2\kappa-2)^{\frac{q}{2}}+t^{\frac{q}{\tau}}M_k(\tau(\kappa-1))^{\frac{q}{\tau}}
          \\
          &\quad\quad
          +t^{\frac{q}{2}}\overline{M}_k(2\kappa)^{\frac{q}{2}}+t^{\frac{q}{\tau}}\overline{M}_k(\tau \kappa)^{\frac{q}{\tau}} \bigg)
          +t^q \brac{ M_k(\kappa-1)^{\frac{q}{\tau}}+\overline{M}_k(\kappa)}^q\Bigg).
    \end{align*}
This completes the proof.
\end{proof}


\begin{lemma}
\label{lemma_nablaX}
 Assume Condition~\ref{cond_levymeasure} and Condition~\ref{cond_driftb}.   For any $t\geq 0$ and  $k,m\leq d$, it holds almost surely that
\begin{align*}
    \sup_{x\in \R^d}\abs{\nabla X^x_t}&\leq e^{\theta_1 t}, \qquad \sup_{x\in \R^d} \abs{\nabla_m \nabla X^x_t}\leq e^{\theta_2 t}, 
\end{align*}
and \begin{align*}
     \sup_{x\in \R^d} \abs{\nabla^2_{k,m} \nabla X^x_t}&\leq e^{\theta_3 t}. 
\end{align*}

\end{lemma}
\begin{proof}
    We will only estimate $\nabla X^x_t$; the remaining cases are similar. The first derivative flow satisfies: 
    \begin{align*}
        \nabla X^x_t=I_{d\times d}+ \int_0^t \nabla b(X^x_s)\nabla X^x_sds. 
    \end{align*}
Since $b$ has bounded first derivative, we may write
\begin{align*}
    \abs{\nabla X^x_t}\leq 1+\int_0^t \sup_{y\in\mathbb{R}^{d}}\abs{\nabla b(y)} \abs{\nabla X^x_s}ds.
\end{align*}
It follows from Gr\"{o}nwall's inequality \eqref{gronwall} and Condition~\ref{cond_driftb} that $\sup_{x\in \R^d}\abs{\nabla X^x_t}\leq e^{\theta_1 t} $.  
\end{proof}

\begin{lemma}
\label{lemma_DXandothers}
Assume Condition~\ref{cond_levymeasure} and Condition~\ref{cond_driftb}. For any $t\geq 0$ and $k,m\leq d$, the $d\times d$ random matrices $\der X^x_t,D_k \nabla X^x_t$ and $\nabla_m\der X^x_t$ are $p$-integrable for $p\geq 1$. 

Specifically, it holds for $p\geq 2$ that
\begin{align*}
\mathcal{Q}^1(p,t)  &=\sup_{x\in\R^d}  \E{\abs{\der X^x_t}^p} \\
  &\leq  e^{p\theta_1 t}  d^{\frac{p}{2}-1}\sum_{j=1}^d \brac{C_6(p)\brac{t^{\frac{p}{2}}M_j(2\kappa)^{\frac{p}{2}}+tM_j(p\kappa) }+t^{p} M_j(\kappa)^{p}};
\end{align*}
\begin{align*}
   \mathcal{Q}^2_k(p,t) &=\sup_{x\in\R^d}\E{\abs{D_k \nabla X^x_t}^{p}}\\ &\leq \frac{1}{2} d^{\frac{p}{2}-1} \theta_2^p e^{3p\theta_1 t}    \sum_{j=1}^d \Big(C_6(p)\brac{t^{\frac{3p}{2}}M_j(2\kappa)^{\frac{p}{2}}+t^{p+1}M_j(p\kappa) }+t^{2p} M_j(\kappa)^{p}\Big),
\end{align*}
and
\begin{align*}
 \mathcal{Q}^3_m(p,t) &= \E{\abs{\nabla_m\der X^x_t}^p}\\
    &\leq \frac{1}{2}d^{\frac{p}{2}-1}\theta_2^{p} e^{3p\theta_1 t} \sum_{j=1}^d \Big(C_6(p)\brac{t^{\frac{3p}{2}}M_j(2\kappa)^{\frac{p}{2}}+t^{p+1}M_j(p\kappa) }+t^{2p} M_j(\kappa)^{p}\Big),
\end{align*}
where the factor $C_6$ is defined in Lemma~\ref{lemma_kunita}, and the remaining notations are from Conditions~\ref{cond_levymeasure} and~\ref{cond_driftb}. 

Meanwhile for $2>p\geq 1$, we have
\begin{align*}
&\mathcal{Q}^1(p,t)=\sup_{x\in\R^d}  \E{\abs{\der X^x_t}^p} \leq \brac{e^{2\theta_1 t}  \sum_{j=1}^d \Big(C_6(2) 2t M_j(2\kappa)+t^{2} M_j(\kappa)^{2}\Big)}^{\frac{p}{2}} ;
\\   
&\mathcal{Q}^2_k(p,t) =\sup_{x\in\R^d}\E{\abs{D_k \nabla X^x_t}^{p}}\leq \brac{\frac{1}{2} \theta_2^2 e^{6\theta_1 t}    \sum_{j=1}^d \Big(2 C_6(2){t^{3}M_j(2\kappa) }+t^{4} M_j(\kappa)^{2}\Big)}^{\frac{p}{2}};
\end{align*}
and
\begin{align*}
 \mathcal{Q}^3_m(p,t) &= \E{\abs{\nabla_m\der X^x_t}^p}
    \leq \brac{\frac{1}{2}\theta_2^{2} e^{6\theta_1 t} \sum_{j=1}^d \Big(2C_6(2){t^{3}M_j(2\kappa) }+t^{4} M_j(\kappa)^{2}\Big)}^{\frac{p}{2}}.
\end{align*}

Moreover regarding dimension dependence, our upper bounds on $\mathcal{Q}^1(p,t), \mathcal{Q}^2_k(p,t)$ and $\mathcal{Q}^3_m(p,t)$ are all of the order $\mathcal{O}\left(d^{\frac{p}{2}}\right)$ as $d\rightarrow\infty$. 
\end{lemma}

\begin{proof}

First,  we have
\begin{align*}
    \der X^x_t=\int_0^t\nabla b(X^x_s) \der X^x_sds+ \der Z_t,
\end{align*}
and $\der Z_t$ is a $d\times d$ random matrix with diagonal entries
\begin{align*}
   \bracsq{ \der Z_t}_{j,j}= D_jZ^j_t=\int_0^t\int_{\R}V_j(s,\xi_j)N_j(d\xi_j,ds),
\end{align*}
and non-diagonal entries $\bracsq{\der Z_t}_{i,j}=0$
for any $i\neq j$. 
Note that $p$-integrability of $\mathbb{D}X^x_t$ has been shown in \cite[Section 4]{kulik2023gradient} without explicit constants; so we re-do it here in order to spell out the constants explicitly. It is sufficient to consider only $p\geq 2$. Kunita's inequality in Lemma~\ref{lemma_kunita} implies
\begin{align}
\label{estimate_momentDZ}
    \E{\abs{\der Z_t}^p}&=\brac{{\sum_{j=1}^d \bracsq{\der Z_t}^2_{j,j} }}^{\frac{p}{2}}\nonumber\\
    &\leq d^{\frac{p}{2}-1}\sum_{j=1}^d \E{\abs{\int_0^t\int_{\R}V_j(s,\xi_j)N_j(d\xi_j,ds)}^p}\nonumber\\
    &\leq d^{\frac{p}{2}-1} \sum_{j=1}^d C_6(p)\bigg(\brac{\int_0^t\int_{\R} \brac{\phi (\xi_j)\psi(s)}^2m_j(d\xi_j)ds}^{\frac{p}{2}}\nonumber\\
    &+\int_0^t\int_{\R} \abs{\phi (\xi_j)\psi(s)}^p m_j(d\xi_j)ds\bigg)+\brac{\int_0^t\int_{\R} \brac{\phi (\xi_j)\psi(s)}m_j(d\xi_j)ds}^{p}\nonumber \\
    &\leq d^{\frac{p}{2}-1}\sum_{j=1}^d \brac{C_6(p)\brac{t^{\frac{p}{2}}M_j(2\kappa)^{\frac{p}{2}}+tM_j(p\kappa) }+t^p M_j(\kappa)^p}.
\end{align}
The quantities $M_j(\kappa),M_j(2\kappa)$ and $M_j(p\kappa)$ are bounded per (5) in Condition~\ref{cond_levymeasure} and Remark 2.1. Furthermore, we know 
\begin{align*}
    \abs{\der X^x_t}\leq \int_0^t \theta_1 \abs{\der X^x_s}ds+\abs{\der Z_t},
\end{align*}
and thus by Gr\"{o}nwall's inequality \eqref{gronwall}, $\abs{\der X^x_t}\leq e^{\theta_1t}\abs{\der Z_t}$. Then based on \eqref{estimate_momentDZ}, we can write for $p\geq 2$
\begin{align}
\label{estimate_momentofDX}
\sup_{x\in\R^d}  \E{\abs{\der X^x_t}^p}&\leq e^{p\theta_1 t}\E{\abs{\der Z_t}^p}\nonumber\\
&\leq e^{p\theta_1 t}  d^{\frac{p}{2}-1}\sum_{j=1}^d \brac{C_6(p)\brac{t^{\frac{p}{2}}M_j(2\kappa)^{\frac{p}{2}}+tM_j(p\kappa) }+t^p M_j(\kappa)^p}.
\end{align}

Second, we study $D_k \nabla X^x_t$ which satisfies
\begin{align}
\label{equation_DknablaX}
    D_k \nabla X^x_t=\int_0^t \brac{\nabla_k\nabla b}(X^x_s)\,D_k X^x_s \,\nabla X^x_s+\nabla b(X^x_s)\,D_k  \nabla X^x_sds.
\end{align}
Condition~\ref{cond_driftb} on $b$ implies
\begin{align*}
    \abs{D_k \nabla X^x_t}\leq  \int_0^t  \theta_2\abs{\nabla X^x_s}\abs{D_k X^x_s }ds+\int_0^t \theta_1 \abs{D_k  \nabla X^x_s}ds. 
\end{align*}
By Gr\"{o}nwall's inequality \eqref{gronwall} and Lemma~\ref{lemma_nablaX}, 
\begin{align*}
    \abs{D_k \nabla X^x_t}\leq e^{\theta_1 t} \int_0^t  \theta_2\abs{\nabla X^x_s}\abs{D_k X^x_s }ds \leq\theta_2 e^{2\theta_1 t}  \int_0^t  \abs{D_k X^x_s }ds .
\end{align*}
Let $p\geq 2$. Then via Jensen's inequality \eqref{jensen} applied to the function $x\mapsto \abs{x}^p$ and the estimate at \eqref{estimate_momentofDX}, 
\begin{align*}
    &\E{\abs{D_k \nabla X^x_t}^p}\\
    &\leq  \theta_2^p e^{2p\theta_1 t}t^{p-1} \int_0^t \E{\abs{D_k X^x_s}^p}ds\\
    &\leq   \theta_2^p e^{2p\theta_1 t}t^{p-1}\cdot\int_0^t e^{p\theta_1 s}  d^{\frac{p}{2}-1}\sum_{j=1}^d \brac{C_6(p)\brac{s^{\frac{p}{2}}M_j(2\kappa)^{\frac{p}{2}}+sM_j(p\kappa) }+s^p M_j(\kappa)^p}ds\\
     &\leq \frac{1}{2} d^{\frac{p}{2}-1} \theta_2^p e^{3p\theta_1 t}    \sum_{j=1}^d \brac{C_6(p)\brac{t^{\frac{3p}{2}}M_j(2\kappa)^{\frac{p}{2}}+t^{p+1}M_j(p\kappa) }+t^{2p} M_j(\kappa)^p}. 
\end{align*}

Next we consider the $d\times d$ random matrices $\nabla_m\der X^x_t$ which satisfies
\begin{align*}
    \nabla_m\der X^x_t=\int_0^t  \left(\nabla^2 b(X^x_s)\,\nabla_m X^x_s \,\der X^x_s+\nabla b(X^x_s)\,\nabla_m \der X^x_s\right)ds.
\end{align*}
Condition~\ref{cond_driftb} on $b$ and Lemma~\ref{lemma_nablaX} implies
\begin{align*}
    \abs{\nabla_m\der X^x_t}\leq \int_0^t \left(\theta_2e^{\theta_1s}\abs{\der X^x_s}+\theta_1\abs{\nabla_m\der X^x_s}\right)ds. 
\end{align*}
By Gr\"{o}nwall's inequality \eqref{gronwall}, Jensen's inequality \eqref{jensen} applied to the function $x\mapsto \abs{x}^p$ and the estimate in \eqref{estimate_momentofDX}, we obtain:
\begin{align*}
   &\E{\abs{\nabla_m\der X^x_t}^p}\\
     &\leq  \theta_2^p e^{p\theta_1 t} t^{p-1}
     \cdot\int_0^t e^{2p\theta_1 s}  d^{\frac{p}{2}-1}\sum_{j=1}^d \brac{C_6(p)\brac{s^{\frac{p}{2}}M_j(2\kappa)^{\frac{p}{2}}+sM_j(p\kappa) }+s^p M_j(\kappa)^p}ds\\
     &\leq \frac{1}{2}d^{\frac{p}{2}-1}\theta_2^p e^{3p\theta_1 t} \sum_{j=1}^d \brac{C_6(p)\brac{t^{\frac{3p}{2}}M_j(2\kappa)^{\frac{p}{2}}+t^{p+1}M_j(p\kappa) }+t^{2p} M_j(\kappa)^p}.
\end{align*}

So far we have considered the case $p\geq 2$. Finally, the estimates for the case $2>p\geq 1$ come from the above estimates and the fact that $\E{\abs{F}^p}\leq \E{\abs{F}^2}^{\frac{p}{2}}$. 
This completes the proof.
\end{proof}

\begin{lemma}\label{lemma_DDX}
Assume Condition~\ref{cond_levymeasure} and Condition~\ref{cond_driftb}. For any $t\geq 0$ and $k\leq d$, the $d\times d$ random matrix $D_k\der X^x_t$ is $q$-integrable for $\tau\geq q\geq 1$. 
     
Specifically, we have
\begin{align*}
\mathcal{Q}^4_k(q,t) 
&= \sup_{x\in\R^d}\E{\abs{D_k \der X^x_t}^q} \\
    &\leq 2^{q-\frac{\tau}{q}}e^{q\theta_1 t} \sum_{j=1}^d\bigg(d^{\frac{q}{2}-\frac{q}{\tau}}C_6(\tau)^{\frac{q}{\tau}}\brac{t^\frac{q}{\tau}{M}_j\brac{\tau(2\kappa-1)}^{\frac{q}{\tau}}+t^{\frac{q}{2}}M_j\brac{4\kappa-2)}^{\frac{q}{2}} }
    \\
    &\qquad\qquad\qquad\qquad\qquad\qquad
    +d^{\frac{q}{2}-\frac{q}{\tau}}t^q M_j(2\kappa-1)^q\\
     &\quad
     +\brac{\frac{1}{2}}^{\frac{q}{\tau}} \theta_2^q e^{2q\theta_1 t}   d^{q-\frac{q}{\tau}} \cdot\brac{C_6(2\tau)^{\frac{q}{\tau}}\brac{t^{2q}M_j(2\kappa)^{q}+t^{q+\frac{q}{\tau}}M_j(2\tau\kappa)^{\frac{q}{\tau}} }+t^{3q} M_j(\kappa)^{2q}}\bigg), 
\end{align*}
where the factor $C_6$ is defined in Lemma~\ref{lemma_kunita}, and the remaining notations are from Conditions~\ref{cond_levymeasure} and~\ref{cond_driftb}. Moreover regarding dimension dependence, our upper bound on $\mathcal{Q}^4_k(q,t)$ is of the order $\mathcal{O}\left(d^{q-\frac{q}{\tau}+1}\right)$ as $d\rightarrow\infty$. 
\end{lemma}

\begin{proof}
 $D_k \der X^x_t$ is the solution of
\begin{align}
\label{equation_DkDX}
    D_k \der X^x_t=\int_0^t \left(\nabla^2b(X^x_s)\,D_kX^{x}_s \, \der X^x_s+\nabla b(X^x_s)\,D_k \der X^x_s\right)ds+D_k \der Z_t.
\end{align}
The last term on the right-hand side is a $d\times d$ random matrix whose entries are given by
\begin{align*}
    \bracsq{D_k \der Z_t}_{j,j}&=\int_0^t \int_{\R}V_j(s,\xi_j)\phi'(\xi_j)\psi(s)N_j(d\xi_j,ds)\\
    &=\int_0^t\int_{\R} \phi(\xi_j)\phi'(\xi_j)\psi(s)^2 N_j(d\xi_j,ds),
\end{align*}
and $\bracsq{D_k\der Z_t}_{i,j}=0$ for $i\neq j$. To see $ \bracsq{D_k\der Z_t}_{j,j}$ is $\tau$-integrable, we apply Kunita's inequality in Lemma~\ref{lemma_kunita} to get
\begin{align}
\label{estimate_momentDDZ}
 \E{\abs{\bracsq{D_k \der Z_t}_{j,j}}^\tau}&=\E{\abs{\int_0^t \int_{\R}\phi(\xi_j)\phi'(\xi_j)\psi(s)^2 N_j(d\xi_j,ds)}^\tau}\nonumber \\
 &\leq  C_6(\tau)\Bigg(\int_0^t\int_{\R} \abs{\phi(\xi_j)\phi'(\xi_j)\psi(s)^2}^\tau \rho_j(d\xi_j)ds\nonumber\\
 &\qquad\qquad+\brac{\int_0^t\int_{\R} \abs{\phi(\xi_j)\phi'(\xi_j)\psi(s)^2}^2 \rho_j(d\xi_j)ds}^{\frac{\tau}{2}}\Bigg)\nonumber\\
 &\qquad+\brac{\int_0^t\int_{\R} \abs{\phi(\xi_j)\phi'(\xi_j)\psi(s)^2} \rho_j(d\xi_j)ds}^{\tau}\nonumber\\
 &\leq C_6(\tau)\brac{t{M}_j\brac{\tau(2\kappa-1)}+t^{\frac{\tau}{2}}M_j\brac{4\kappa-2)}^{\frac{\tau}{2}} }+t^\tau M_j(2\kappa-1)^\tau. 
\end{align}
In the last line, the quantities $M_j(2\kappa-1), {M}_j\brac{\tau(2\kappa-1)}$ and $M_j\brac{4\kappa-2)}$ are both bounded as a consequence of (6) in Condition~\ref{cond_levymeasure}, Remark 2.1 in the paper and the fact that $\min \{\tau(2\kappa-1),4\kappa-2\}>2\kappa-2$. Then we can combine the above estimate and 
\begin{align*}
    \E{\abs{D_k\der Z_t}^\tau}\leq d^{\frac{\tau}{2}-1}\sum_{j=1^d}\E{\abs{\bracsq{D_k \der Z_t}_{j,j}}^\tau}. 
\end{align*}

Now by Condition~\ref{cond_driftb} on $b$ and Gr\"{o}nwall's inequality \eqref{gronwall} applied to \eqref{equation_DkDX}, we have
\begin{align*}
   \abs{D_k \der X^x_t}\leq \brac{D_k \der Z_t+\theta_2\int_0^t \abs{\der X^x_s}^2ds}e^{\theta_1 t}.
\end{align*}
 Jensen's inequality \eqref{jensen} applied to the function $x\mapsto \abs{x}^\tau$, followed by usage of the estimates at \eqref{estimate_momentofDX}, \eqref{estimate_momentDDZ} lead to
\begin{align*}
     &\E{\abs{D_k \der X^x_t}^\tau} \\
     &\leq 2^{\tau-1}e^{\tau\theta_1 t}\brac{\E{\abs{D_k\der Z_t}^\tau} + (\theta_2)^\tau t^{\tau-1}\int_0^t \E{\abs{\der X^x_s}^{2\tau}}ds }\\
     &\leq 2^{\tau-1}e^{\tau\theta_1 t} \sum_{j=1}^d\bigg(d^{\frac{\tau}{2}-1}C_6(\tau)\brac{t{M}_j\brac{\tau(2\kappa-1)}+t^{\frac{\tau}{2}}M_j\brac{4\kappa-2)}^{\frac{\tau}{2}} }\\
     &\hspace{26em}+d^{\frac{\tau}{2}-1}t^\tau M_j(2\kappa-1)^\tau 
     \\&\quad+ (\theta_2)^\tau t^{\tau-1}
     \int_0^t    e^{2\tau\theta_1 s}  d^{\tau-1} \brac{C_6(2\tau)\brac{s^{\tau}M_j(2\kappa)^{\tau}+sM_j(2\tau\kappa) }+s^{2\tau} M_j(\kappa)^{2\tau}}       ds\bigg)\\
     &\leq 2^{\tau-1}e^{\tau\theta_1 t} \sum_{j=1}^d\bigg(d^{\frac{\tau}{2}-1}C_6(\tau)\brac{t{M}_j\brac{\tau(2\kappa-1)}+t^{\frac{\tau}{2}}M_j\brac{4\kappa-2)}^{\frac{\tau}{2}} }\\
     &\hspace{26em}+d^{\frac{\tau}{2}-1}t^\tau M_j(2\kappa-1)^\tau\\
     &\quad+\frac{1}{2} (\theta_2)^\tau e^{2\tau\theta_1 t}   d^{\tau-1} \brac{C_6(2\tau)\brac{t^{2\tau}M_j(2\kappa)^{\tau}+t^{\tau+1}M_j(2\tau\kappa) }+t^{3\tau} M_j(\kappa)^{2\tau}} \bigg)    . 
\end{align*}
Consequently for $q\in [1,\tau]$, Lyapunov's inequality and \eqref{minkowskitype_inequality} imply
\begin{align*}
&\sup_{x\in\R^d}\E{\abs{D_k \der X^x_t}^q}
\\
&\leq \sup_{x\in\R^d}\E{\abs{D_k \der X^x_t}^\tau}^{q/\tau} \\
    &\leq \Bigg(2^{\tau-1}e^{\tau\theta_1 t} \sum_{j=1}^d\bigg( d^{\frac{\tau}{2}-1}C_6(\tau)\brac{t{M}_j\brac{\tau(2\kappa-1)}+t^{\frac{\tau}{2}}M_j\brac{4\kappa-2)}^{\frac{\tau}{2}} }\\
     &\hspace{27em}+d^{\frac{\tau}{2}-1}t^\tau M_j(2\kappa-1)^\tau\\
     &\qquad+\frac{1}{2} (\theta_2)^\tau e^{2\tau\theta_1 t}   d^{\tau-1} \brac{C_6(2\tau)\brac{t^{2\tau}M_j(2\kappa)^{\tau}+t^{\tau+1}M_j(2\tau\kappa) }+t^{3\tau} M_j(\kappa)^{2\tau}}\bigg)   \Bigg)^{\frac{q}{\tau}} \\
      &\leq 2^{q-\frac{\tau}{q}}e^{q\theta_1 t} \sum_{j=1}^d\bigg(d^{\frac{q}{2}-\frac{q}{\tau}}C_6(\tau)^{\frac{q}{\tau}}\brac{t^\frac{q}{\tau}{M}_j\brac{\tau(2\kappa-1)}^{\frac{q}{\tau}}+t^{\frac{q}{2}}M_j\brac{4\kappa-2)}^{\frac{q}{2}} }
      \\
     &\hspace{26em}+d^{\frac{q}{2}-\frac{q}{\tau}}t^q M_j(2\kappa-1)^q\\
     &\quad+\brac{\frac{1}{2}}^{\frac{q}{\tau}} \theta_2^q e^{2q\theta_1 t}   d^{q-\frac{q}{\tau}} \brac{C_6(2\tau)^{\frac{q}{\tau}}\brac{t^{2q}M_j(2\kappa)^{q}+t^{q+\frac{q}{\tau}}M_j(2\tau\kappa)^{\frac{q}{\tau}} }+t^{3q} M_j(\kappa)^{2q}}\bigg)   . 
\end{align*}
Bounded-ness of $M_j(2\kappa-1), {M}_j\brac{\tau(2\kappa-1)}$ and $M_j\brac{4\kappa-2)}$ has been explained below \eqref{estimate_momentDDZ}. Finally, (5) in Condition~\ref{cond_levymeasure} and Remark~\ref{remark_smalljump} imply
\begin{align*}
    M_j(2\tau\kappa)\leq M_j(2\kappa)\leq M_j(\kappa)<\infty. 
\end{align*}
The proof is complete.
\end{proof}

\begin{lemma}(\cite[Lemma 5]{kulik2023gradient})
\label{lemma_inverseDX}
  Assume Condition~\ref{cond_levymeasure} and Condition~\ref{cond_driftb}. For any $t\geq 0$, the $d\times d$ random matrix $(\der X^x_t)^{-1}$ is $p$-integrable for $p\geq 1$. 

Specifically, we have
\begin{align}
\label{estimate_inverseDX}
       \mathcal{Q}^5(p,t) &=\sup_{x\in\R^d} \E{\abs{\brac{\der X^x_t}^{-1}}^p}\nonumber\\
    &\leq e^{p\theta_1t}2^{p-1}\brac{\min \left\{ \frac{\theta_1 te^{\theta_1t}}{1-\theta_1 te^{\theta_1t} },1 \right\} \E{\abs{\brac{DZ_t}^{-1}}^p}+\E{\abs{\brac{DZ_t}^{-1}}^p}}. 
\end{align}
The quantity $\E{\abs{\brac{DZ_t}^{-1}}^p}$ is bounded for every $t>0$. Specifically, it holds that 
\begin{align*}
    E{\abs{\brac{DZ_t}^{-1}}^p}\leq Cd^{\frac{p}{2}}t^{-\frac{\kappa}{\pi}},
\end{align*}
where $C=C(p,T)$ is some constant and the parameters $\pi,\kappa$ are in Condition~\ref{cond_levymeasure}. Hence in terms of dimension dependence, our upper bound of  $\mathcal{Q}^5(p,t)$ is of the order $\mathcal{O}\left(d^{\frac{p}{2}}\right)$ as $d\rightarrow\infty$. 
\end{lemma}

\begin{proof}
$p$-integrability of $(\der X^x_t)^{-1}$ has been verified in \cite[Lemma 5]{kulik2023gradient}, below we will make the bounding constant in the $p$-th moment explicit. 

Proposition~\ref{prop_bismut} says 
\begin{align}
\label{equation_DXasA(x,t)andnablaX}
    \brac{\der X^x_t}^{-1}=A(x,t)\brac{\nabla X^x_t}^{-1}.
\end{align}
Next, let us study $p$-integrability of $A(x,t)$ for $p\geq 1$. \cite[Section 5]{kulik2023gradient} provides the representation
\begin{align}
\label{equation_A(x,t)expansion}
    A(x,t)=\brac{\der Z_t}^{-1} \brac{I_{d\times d}+Q(x,t)}^{-1}=\brac{\der Z_t}^{-1} +\brac{\der Z_t}^{-1} \sum_{n=1}^\infty (-1)^n Q(x,t)^n,
\end{align}
where 
\begin{align*}
    Q(x,t)&=\brac{\int_0^t \brac{\brac{\nabla X^x_s}^{-1}-I_{d\times d}}d{\der Z_s}}(\der Z_t)^{-1}
    \\&\leq \min \left\{\abs{\brac{\nabla X^x_t}^{-1}-I_{d\times d}},\frac{1}{2} \right\}. 
\end{align*}
Notice that for $a,b>0$, $e^a-e^b\leq e^{a+b}\abs{a-b}$ so
\begin{align*}
    \abs{\brac{\nabla X^x_t}^{-1}-I_{d\times d}}&\leq \exp \brac{-\int_0^t \nabla b(X_s)ds}-\exp (0)\leq \theta_1 te^{\theta_1t}. 
\end{align*}
The last line is a consequence of Condition~\ref{cond_driftb}. This implies 
\begin{align*}
    \abs{\sum_{n=1}^\infty (-1)^n Q(x,t)^n }\leq \sum_{n=1}^\infty \brac{\min \left\{\theta_1 te^{\theta_1t}, \frac{1}{2} \right\}}^n
    =\min \left\{ \frac{\theta_1 te^{\theta_1t}}{1-\theta_1 te^{\theta_1t} },1 \right\}. 
\end{align*}
Therefore, we deduce from Equation \eqref{equation_A(x,t)expansion} that
\begin{align*}
    \E{\abs{A(x,t)}^p}\leq 2^{p-1}\min \left\{ \frac{\theta_1 te^{\theta_1t}}{1-\theta_1 te^{\theta_1t} },1 \right\} \E{\abs{\brac{DZ_t}^{-1}}^p}+2^{p-1}\E{\abs{\brac{DZ_t}^{-1}}^p}. 
\end{align*}
The above estimate and Equation \eqref{equation_DXasA(x,t)andnablaX} lead to 
\begin{align*}
    &\E{\abs{\brac{\der X^x_t}^{-1}}^p}\leq \abs{\brac{\nabla X^x_t}^{-1}}^p \E{\abs{A(x,t) }^p}\\
    &\qquad\leq e^{p\theta_1t}2^{p-1}\brac{\min \left\{ \frac{\theta_1 te^{\theta_1t}}{1-\theta_1 te^{\theta_1t} },1 \right\} \E{\abs{\brac{DZ_t}^{-1}}^p}+\E{\abs{\brac{DZ_t}^{-1}}^p}}. 
\end{align*}
It has been shown in \cite[part i of Lemma 6 and Section 7]{kulik2023gradient} that the quantity $\E{\abs{\brac{DZ_t}^{-1}}^p}$ can be bounded as follows: 
\begin{align*}
    \E{\abs{\brac{DZ_t}^{-1}}^p}&=\E{\brac{\sum_{j=1}^d\frac{1}{\bracsq{DZ_t}^2_{j,j}} }^{\frac{p}{2}}}\\
    &\leq d^{\frac{p}{2}-1}\sum_{j=1}^d \E{\frac{1}{\bracsq{DZ_t}^p_{j,j}}}\\
    &=\frac{d^{\frac{p}{2}-1}}{\Gamma(p)}\sum_{j=1}^d\int_0^\infty s^{p-1}\exp\brac{-t\int_{\R}\brac{1-e^{-sV_j(\xi,s)}} m_j(d\xi_j)}ds.
\end{align*}
When $t\in (0,T]$ for some $T>0$, \cite[Lemma 7]{kulik2023gradient}, (4) in Condition~\ref{cond_levymeasure} and the above inequality imply
\begin{align*}
    \E{\abs{\brac{DZ_t}^{-1}}^p}\leq Cd^{\frac{p}{2}}t^{-\frac{\kappa}{\pi}},
\end{align*}
for some constant $C=C(p,T)$.  
This completes the proof.
\end{proof}

\begin{lemma}
\label{lemma_nablaDDX}
   Assume Condition~\ref{cond_levymeasure} and Condition~\ref{cond_driftb}. For any $t\geq 0$ and $m,k\leq d$, the $d\times d$ random matrix $\nabla_m D_k\mathbb{D}X^x_t$ is well-defined and $q$-integrable for $\tau>q\geq 1$. 

Specifically, we have
\begin{align*}
    &\mathcal{Q}^6_{m,k}(q,t)=\sup_{x\in\R^d}\E{\abs{ \nabla_m D_k\mathbb{D}X^x_t}^q}\\
    &\leq \frac{1}{2}\theta_3^q e^{4q\theta_1 t}  d^{q-1}4^{q-1}\sum_{j=1}^d \brac{C_6(2q)\brac{t^{2q}M_j(2\kappa)^{q}+tM_j(2q\kappa) }+t^{3q} M_j(\kappa)^{2q}}\\
    &+\frac{1}{4}d^{q-1}4^{q-1}\theta_2^{3q} e^{5q\theta_1 t} \sum_{j=1}^d \brac{C_6(2q)\brac{t^{4q}M_j(2\kappa)^{q}+t^{3q+1}M_j(2q\kappa) }+t^{5q} M_j(\kappa)^{2q}}\\
    &+ \frac{\theta_2^q}{2}e^{3q\theta_1 t}  d^{q-1}4^{q-1}\sum_{j=1}^d \brac{C_6(2q)\brac{t^{2q}M_j(2\kappa)^{q}+t^q M_j(2q\kappa) }+t^{3q} M_j(\kappa)^{2q}}\\
       &\quad+  2^{q-\frac{\tau}{q}+\tau-3}4^{q-1}e^{(3q+\tau)\theta_1 t} \\&\hspace{3em}\cdot\sum_{j=1}^d\bigg(d^{\frac{q}{2}-\frac{q}{\tau}}C_6(\tau)^{\frac{q}{\tau}}\brac{t^{\frac{q}{\tau}+q}{M}_j\brac{\tau(2\kappa-1)}^{\frac{q}{\tau}}+t^{\frac{q}{2}+q}M_j\brac{4\kappa-2)}^{\frac{q}{2}} }\\
     &\qquad\qquad\qquad+d^{\frac{q}{2}-\frac{q}{\tau}}t^{2q} M_j(2\kappa-1)^q\\
     &+\brac{\frac{1}{2}}^{\frac{q}{\tau}} \theta_2^q e^{2q\theta_1 t}   d^{q-\frac{q}{\tau}} \Big(C_6(2\tau)^{\frac{q}{\tau}}\brac{t^{3q}M_j(2\kappa)^{q}+t^{2q+\frac{q}{\tau}}M_j(2\tau\kappa)^{\frac{q}{\tau}} }+t^{4q} M_j(\kappa)^{2q}\Big)\bigg),    \end{align*}
where the factor $C_6$ is defined in Lemma~\ref{lemma_kunita}, and the remaining notations are from Conditions~\ref{cond_levymeasure} and~\ref{cond_driftb}. Moreover regarding dimension dependence, our upper bound on $\mathcal{Q}^6_{m,k}(q,t)$ is of the order $\mathcal{O}\left(d^{q+1-\frac{q}{\tau}}\right)$ as $d\rightarrow\infty$. 
\end{lemma}

\begin{proof}
    It follows from \eqref{equation_DkDX} that
    \begin{align*}
         \nabla_m D_k\mathbb{D}X^x_t
         =&\int_0^t  \bigg(\nabla^3 b(X^x_s) \,\nabla_m X^x_s \,D_kX^{x}_s \,\der X^x_s + \nabla^2b(X^x_s) \,\nabla_m D_kX^{x}_s \,\der X^x_s\\
         &\qquad\qquad+\nabla^2b(X^x_s) \,D_kX^{x}_s \,\nabla_m\mathbb{D}X^x_s+\nabla b(X^x_s)\,\nabla_m X^x_s\, D_k\mathbb{D}X^x_s\\
         &\qquad\qquad\qquad\qquad+\nabla b(X^x_s)\,\nabla_m D_k\mathbb{D}X^x_s \bigg)ds.
    \end{align*}
Condition~\ref{cond_driftb} on $b$ implies
\begin{align*}
  \abs{ \nabla_m D_k\mathbb{D}X^x_t}
         =&\int_0^t\bigg(\theta_3\,e^{\theta_1 s} \,\abs{D_kX^{x}_s} \,\abs{\der X^x_s} + \theta_2 \,\abs{\nabla_m D_kX^{x}_s} \,\abs{\der X^x_s}\\
         &\qquad\qquad+\theta_2 \,\abs{D_kX^{x}_s }\,\abs{\nabla_m\mathbb{D}X^x_s}+\theta_2\,e^{\theta_1 s}\, \abs{D_k\mathbb{D}X^x_s}\\
         &\qquad\qquad\qquad\qquad\qquad+\theta_1\,\abs{\nabla_m D_k\mathbb{D}X^x_s} \bigg)ds.
\end{align*}
Assume $q\in [1,\tau)$. By Gr\"{o}nwall's inequality \eqref{gronwall} and Jensen's inequality \eqref{jensen} applied to the function $\abs{x}\mapsto \abs{x}^q$,
\begin{align*}
  \E{\abs{ \nabla_m D_k\mathbb{D}X^x_t}^q}
   &\leq e^{q\theta_1 t} t^{q-1}4^{q-1}\int_0^t\bigg(\theta_3^q\,e^{q\theta_1 s} \,\E{\abs{D_kX^{x}_s}^q \,\abs{\der X^x_s}^q} \\
    &\qquad\qquad\qquad\qquad\qquad\qquad+\theta_2^q \,\E{\abs{\nabla_m D_kX^{x}_s}^q \,\abs{\der X^x_s}^q}
    \\
    &\qquad+\theta_2^q \,\E{\abs{D_kX^{x}_s }^q\,\abs{\nabla_m\mathbb{D}X^x_s}^q}+\theta_2^q\,e^{q\theta_1 s}\, \E{\abs{D_k\mathbb{D}X^x_s}^q}\bigg)ds\\
    &\leq e^{q\theta_1 t} t^{q-1}4^{q-1}\bigg( \theta_3^qe^{q\theta_1t}\int_0^t \E{\abs{D_kX^{x}_s}^q \,\abs{\der X^x_s}^q}ds \\
    &\quad+\theta_2^q \int_0^t \E{\abs{\nabla_m D_kX^{x}_s}^q \,\abs{\der X^x_s}^q} ds+\theta_2^q \int_0^t \E{\abs{D_kX^{x}_s }^q\,\abs{\nabla_m\mathbb{D}X^x_s}^q}ds\\
    &\qquad\qquad\qquad\qquad+\theta_2^q e^{q\theta_1 t}\int_0^t \E{\abs{D_k\mathbb{D}X^x_s}^q}ds\bigg).
\end{align*}
Let us study terms on the right hand side. Lemma~\ref{lemma_DXandothers} (keeping in mind $2q\geq 2$) imply
\begin{align*}
    &\theta_3^qe^{q\theta_1t}\int_0^t \E{\abs{\abs{\der X^x_s}^{2q}}}ds\\
    &= \theta_3^qe^{q\theta_1t}\int_0^t \mathcal{Q}^1(2q,s) ds\\
   & \leq \frac{1}{2}\theta_3^q e^{3q\theta_1 t}  d^{q-1}\sum_{j=1}^d \brac{C_6(2q)\brac{t^{q+1}M_j(2\kappa)^{q}+tM_j(2q\kappa) }+t^{2q+1} M_j(\kappa)^{2q}}.
\end{align*}

Similarly, 
\begin{align*}
    &\theta_2^q\int_0^t \E{\abs{\nabla_m D_kX^{x}_s}^q \,\abs{\der X^x_s}^q} ds\\
    &\leq \frac{\theta_2^q}{2}\int_0^t \left(\E{\abs{\nabla_m D_kX^{x}_s}^{2q}}+ \E{\abs{\der X^x_s}^{2q}}\right)ds\\
    &\leq \frac{\theta_2^q}{2}\int_0^t \left(\mathcal{Q}^2_m(2q,s)+\mathcal{Q}^1(2q,s) \right)ds\\
    &\leq \frac{1}{8}d^{q-1}\theta_2^{3q} e^{4q\theta_1 t} \sum_{j=1}^d \brac{C_6(2q)\brac{t^{3q+1}M_j(2\kappa)^{q}+t^{2q+2}M_j(2q\kappa) }+t^{4q+1} M_j(\kappa)^{2q}}\\
    &\qquad+ \frac{\theta_2^q}{4}e^{2q\theta_1 t}  d^{q-1}\sum_{j=1}^d \brac{C_6(2q)\brac{t^{q+1}M_j(2\kappa)^{q}+tM_j(2q\kappa) }+t^{2q+1} M_j(\kappa)^{2q}}.
\end{align*}

The term $\theta_2^q\int_0^t\E{\abs{D_kX^{x}_s }^q\,\abs{\nabla_m\mathbb{D}X^x_s}^q}ds$
has the same bound as the previous one. Finally,
\begin{align*}
   &\theta_2^q e^{q\theta_1 t} \int_0^t \E{\abs{D_k\mathbb{D}X^x_s}^q}ds
   \\
   &=\theta_2^q e^{q\theta_1 t}\int_0^t \mathcal{Q}^4_k(q,s)ds\\
    &\leq 2^{q-\frac{\tau}{q}+\tau-3}e^{(2q+\tau)\theta_1 t}\\
    &\quad\cdot\sum_{j=1}^d\bigg(d^{\frac{q}{2}-\frac{q}{\tau}}C_6(\tau)^{\frac{q}{\tau}}\brac{t^{\frac{q}{\tau}+1}{M}_j\brac{\tau(2\kappa-1)}^{\frac{q}{\tau}}+t^{\frac{q}{2}+1}M_j\brac{4\kappa-2)}^{\frac{q}{2}} }\\
     &\qquad\qquad\qquad\qquad+d^{\frac{q}{2}-\frac{q}{\tau}}t^{q+1} M_j(2\kappa-1)^q+\brac{\frac{1}{2}}^{\frac{q}{\tau}} \theta_2^q e^{2q\theta_1 t}   d^{q-\frac{q}{\tau}}\\
     & \qquad\qquad\cdot\Big(C_6(2\tau)^{\frac{q}{\tau}}\brac{t^{2q+1}M_j(2\kappa)^{q}+t^{q+\frac{q}{\tau}+1}M_j(2\tau\kappa)^{\frac{q}{\tau}} }+t^{3q+1} M_j(\kappa)^{2q}\Big)\bigg)   .    
\end{align*}

A combination of the previous calculations will yield the desired bound on \\$\mathcal{Q}^6_{m,k}(q,t)=\sup_{x\in\R^d}\E{\abs{ \nabla_m D_k\mathbb{D}X^x_t}^q}$. 

Finally, let us consider the dimension dependence of our upper bound on $\mathcal{Q}^6_{m,k}(q,t)$. Between $d^q$ and $d^{q+1-\frac{q}{\tau}}$, the latter is the dominating quantity since $q/\tau<1$, hence we conclude the upper bound on $\mathcal{Q}^6_{m,k}(q,t)$ is of the order $\mathcal{O}\left(d^{q+1-\frac{q}{\tau}}\right)$ as $d\rightarrow\infty$. 
This completes the proof.
\end{proof}

\begin{lemma}
\label{lemma_nablaDnablaX}
Assume Condition~\ref{cond_levymeasure} and Condition~\ref{cond_driftb}. For any $t\geq 0$ and $m,k\leq d$, the $d\times d$ random matrix $\nabla_m D_k\nabla X^x_t$ is $p$-integrable for any $p\geq 1$. 
 
Specifically, we have for $p\geq 2$
\begin{align*}
     \mathcal{Q}^7_{m,k}(p,t)&=\sup_{x\in\R^d}\E{\abs{\nabla_m D_k\nabla X^x_t}^p}\\
    &\leq \brac{e^{4p\theta_1t}\theta_3^p+e^{p(2\theta_1+\theta_2)t}\theta_2^p }4^{p-1}   d^{\frac{p}{2}-1}\\
    &\hspace{2em}\cdot\sum_{j=1}^d \brac{C_6(p)\brac{t^{\frac{3p}{2}}M_j(2\kappa)^{\frac{p}{2}}+t^{p+1}M_j(p\kappa) }+t^{2p} M_j(\kappa)^{p}}\\
    &\quad+2d^{\frac{p}{2}-1}\theta_2^{2p} 4^{d-1} e^{5p\theta_1 t} \sum_{j=1}^d \Big(C_6(p)\brac{t^{\frac{5p}{2}}M_j(2\kappa)^{\frac{p}{2}}+t^{2p+1}M_j(p\kappa) }+t^{3p} M_j(\kappa)^{p}\Big),
\end{align*}
and when $2>p\geq 1$
\begin{align*}
    \mathcal{Q}^7_{m,k}(p,t)
    &=\sup_{x\in\R^d}\E{\abs{\nabla_m D_k\nabla X^x_t}^p}\\
    &\leq \Bigg(\brac{e^{8\theta_1t}\theta_3^2+e^{2(2\theta_1+\theta_2)t}\theta_2^2 }4 \sum_{j=1}^d {2C_6(2)\brac{t^{3}M_j(2\kappa) }+t^{4} M_j(\kappa)^{2}}\\
    &\hspace{5em}+2\theta_2^{4} 4 e^{10\theta_1 t} \sum_{j=1}^d \Big(C_6(2){2t^{5}M_j(2\kappa) }+t^{6} M_j(\kappa)^{2}\Big)\Bigg)^{\frac{p}{2}},
\end{align*}
where the factor $C_6$ is defined in Lemma~\ref{lemma_kunita}, and the remaining notations are from Conditions~\ref{cond_levymeasure} and~\ref{cond_driftb}. Moreover regarding dimension dependence, our upper bound on $\mathcal{Q}^7_{m,k}(p,t)$ is of the order $\mathcal{O}\left(d^{\frac{p}{2}}\right)$ as $d\rightarrow\infty$. 
\end{lemma}
\begin{proof}

Based on Equation \eqref{equation_DknablaX}, we can write
\begin{align*}
    \nabla_m D_k\nabla X^x_t
    &=\int_0^t \bigg(\nabla^3 b(X^x_s)\, \nabla_m X^x_s\, D_kX^x_s \, \nabla X^x_s+\nabla^2 b(X_s)\, \nabla_m D_kX^x_s\, \nabla X^x_s\\
    &\qquad+\nabla^2 b(X^x_s)\, D_kX^x_s\, \nabla_m\nabla X^x_s +\nabla^2 b(X^x_s)\, \nabla_m X^x_s \, D_k\nabla X^x_s
    \\
    &\qquad\qquad\qquad\qquad\qquad+\nabla b(X^x_s)\, (\nabla_m D_k \nabla X^x_s) \bigg)ds. 
\end{align*}
Assume $p\geq 2$. Condition~\ref{cond_driftb} and an application of Gr\"{o}nwall's inequality \eqref{gronwall}, followed by Jensen's inequality \eqref{jensen} with $\phi(x)=\abs{x}^{p}$ lead to
\begin{align*}
    &\E{\abs{\nabla_m D_k\nabla X^x_t}^p}\\
    &\leq e^{p\theta_1t}4^{p-1} t^{p-1}\int_0^t \bigg(\theta_3^pe^{2p\theta_1 s}\E{\abs{\der X^x_s}^p}+\theta_2^pe^{p\theta_1s}\E{\abs{\nabla_m \der X_s}^p}\\
    &\hspace{13em}+\theta_2^p e^{p\theta_2 s}\E{\abs{\der X^x_s}^p}+\theta_2^p e^{p\theta_1 s}\E{\abs{D_k\nabla X^x_s}^p}\bigg)ds\\
    &\leq \brac{e^{3p\theta_1t}\theta_3^p+e^{p(\theta_1+\theta_2)t}\theta_2^p }4^{p-1}t^{p-1}\int_0^t \mathcal{Q}^1(p,s)ds\\
    &\hspace{5em}+e^{2p\theta_1t}\theta_2^p4^{p-1}t^{p-1}\int_0^t\mathcal{Q}^3_m(p,s)ds+e^{2p\theta_1 t}\theta_2^p4^{p-1}t^{p-1}\int_0^t\mathcal{Q}^2_k(p,s)ds\\
    &\leq \brac{e^{4p\theta_1t}\theta_3^p+e^{p(2\theta_1+\theta_2)t}\theta_2^p }4^{p-1}   d^{\frac{p}{2}-1}\\
    &\hspace{10em}\cdot\sum_{j=1}^d \brac{C_6(p)\brac{t^{\frac{3p}{2}}M_j(2\kappa)^{\frac{p}{2}}+t^{p+1}M_j(p\kappa) }+t^{2p} M_j(\kappa)^{p}}\\
    &\qquad+2d^{\frac{p}{2}-1}\theta_2^{2p} 4^{d-1} e^{5p\theta_1 t} \\
    &\hspace{8em}\cdot\sum_{j=1}^d \Big(C_6(p)\brac{t^{\frac{5p}{2}}M_j(2\kappa)^{\frac{p}{2}}+t^{2p+1}M_j(p\kappa) }+t^{3p} M_j(\kappa)^{p}\Big).
\end{align*}
To reach the last line, we have used Lemma~\ref{lemma_DXandothers}. 

 The estimate for the case $2>p\geq 1$ come from the above estimate and the fact that $\E{\abs{F}^p}\leq \E{\abs{F}^2}^{\frac{p}{2}}$. 
\end{proof}


\begin{lemma}
\label{lemma_nablainverseDX_anddifferentiablity}
Assume Condition~\ref{cond_levymeasure} and Condition~\ref{cond_driftb}. For any $t\geq 0$ and $m\leq d$, the $d\times d$ random matrix $\nabla_m(\der X^x_t)^{-1}$ is well-defined and $p$-integrable for any $p\geq 1$. Moreover, the $d\times d$ random matrix $D_m(\der X^x_t)^{-1}$ is well-defined and $q$-integrable for $\tau>q\geq 1$. 

Specifically for any $p\geq 1$,
\begin{align*}
    \E{\abs{\nabla_m (\der X^x_t)^{-1}}^p}
    &\leq \mathcal{Q}^5(4p,t)+\mathcal{Q}^3_m(2p,t), 
\end{align*}
and for any $q$ such that $\tau>q\geq 1$,
\begin{align*}
     \E{\abs{D_m (\der X^x_t)^{-1}}^q}&\leq \frac{\tau-q}{\tau}\mathcal{Q}^5\brac{\frac{2\tau q}{\tau-q},t}+\frac{q}{\tau}\mathcal{Q}^4_m(\tau,t). 
\end{align*}
The terms $\mathcal{Q}^3, \mathcal{Q}^4$ and $\mathcal{Q}^5$ are  respectively defined in Lemmas~\ref{lemma_DXandothers},~\ref{lemma_DDX} and~\ref{lemma_inverseDX}. 

\end{lemma}
\begin{proof}
Malliavin differentiablity of $(\der X^x_t)^{-1}$ has been shown in \cite[Lemma 5]{kulik2023gradient}; therefore what remains to show is that $\nabla_m(\der X^x_t)^{-1}$ is differentiable in $x$. Assume $t\geq 0$ and for any $x\in\R^d$, $\{x_n:n\in\N\}$ is a sequence converging to $x$. We have
\begin{align*}
    (\der X^{x_n}_t)^{-1}\,\der X^{x_n}_t-(\der X^x_t)^{-1}\,\der X^x_t=I_{d\times d}-I_{d\times d}=0. 
\end{align*}
This is equivalent to
\begin{align*}
    \brac{(\der X^{x_n}_t)^{-1}-(\der X^{x}_t)^{-1}}\,\der X^{x_n}_t-(\der X^x_t)^{-1}\,\brac{\der X^{x_n}_t-\der X^x_t}=0,
\end{align*}
and 
\begin{align*}
    (\der X^{x_n}_t)^{-1}-(\der X^{x}_t)^{-1}=(\der X^x_t)^{-1}\,\brac{\der X^{x_n}_t-\der X^x_t}\,(\der X^{x_n}_t)^{-1}.
\end{align*}
The last equation implies
\begin{align*}
    \nabla_m (\der X^x_t)^{-1}=(\der X^x_t)^{-1}\,\nabla_m \der X^x_t \,(\der X^x_t)^{-1}.
\end{align*}
Then by Young's inequality \eqref{younginequality} with $a=b=2$ and Lemmas~\ref{lemma_DXandothers},~\ref{lemma_inverseDX} 
\begin{align*}
    \E{\abs{\nabla_m (\der X^x_t)^{-1}}^p}&\leq \frac{1}{2}\brac{\E{\abs{(\der X^x_t)^{-1}}^{4p}}+ \E{\abs{\nabla_m \der X^x_t}^{2p}}}\\
    &\leq \mathcal{Q}^5(4p,t)+\mathcal{Q}^3_m(2p,t). 
\end{align*}

Next we have
\begin{align*}
    D_m(\der X^x_t)^{-1}=(\der X^x_t)^{-1}\,D_m\der X^x_t\,(\der X^x_t)^{-1}.
\end{align*}
Assume $q \in [1,\tau)$. Via Young's inequality \eqref{younginequality} with $a=\frac{\tau}{q},b=\frac{\tau}{\tau-q}$ and Lemmas~\ref{lemma_DDX},~\ref{lemma_inverseDX}, we deduce that
\begin{align*}
     \E{\abs{D_m (\der X^x_t)^{-1}}^q}&\leq \frac{\tau-q}{\tau}\E{\abs{(\der X^x_t)^{-1}}^{\frac{2\tau q}{\tau-q}}} +\frac{q}{\tau}\E{\abs{D_m \der X^x_t}^{\tau}}\\
    &\leq \frac{\tau-q}{\tau}\mathcal{Q}^5\brac{\frac{2\tau q}{\tau-q},t}+\frac{q}{\tau}\mathcal{Q}^4_m(\tau,t). 
\end{align*}

The proof is complete.
\end{proof}

\begin{lemma}
\label{lemma_nablaDinverseDX}
Assume Condition~\ref{cond_levymeasure} and Condition~\ref{cond_driftb}. For any $t\geq 0$ and $k,m\leq d$, the $d\times d$ random matrix  $\nabla_m D_k(\mathbb{D}X^x_t)^{-1}$ is well-defined and $q$-integrable for $\tau> q\geq 1$. 

Specifically, let $q_0$ be any constant such that $q<q_0<\tau$ then
\begin{align*}
\mathcal{Q}^8_{m,k}(q,t) 
&=\sup_{x\in \R^d}\E{\abs{\nabla_m D_k(\mathbb{D}X^x_t)^{-1}}^q}\\
&\leq 2\frac{3^{q-1}q}{q_0}\mathcal{Q}^4_k(q_0,t)+2\frac{3^{q-1}(q_0-q)}{2q_0}\bigg(\mathcal{Q}^5\brac{\frac{8qq_0}{q_0-q},t }+\mathcal{Q}^3_m\brac{\frac{4qq_0}{q_0-q},t}\\
&+\mathcal{Q}^5\brac{\frac{2qq_0}{q_0-q},t } \bigg)+ \frac{3^{q-1}q}{q_0}\mathcal{Q}^6_{m,k}(q_0,t)+\frac{3^{q-1}(q_0-q)}{q_0}\mathcal{Q}^5\brac{\frac{2qq_0}{q_0-q},t}. 
\end{align*}
Moreover regarding dimension dependence, our upper bound on the quantity $\mathcal{Q}^8_{m,k}(q,t) $ is of the order $\mathcal{O}\left(d^{\frac{4qq_0}{q_0-q}\vee \brac{q_0+1-\frac{q_0}{\tau}}}\right)$ as $d\rightarrow\infty$. 
\end{lemma}

\begin{proof}
To show $D_k(\mathbb{D}X^x_t)^{-1}$ is differentiable in $x$, one can follow a very similar argument in the proof of Lemma~\ref{lemma_nablainverseDX_anddifferentiablity} for differentiability of $(\mathbb{D}X^x_t)^{-1}$. Thus, we can write 
    \begin{align*}
        \nabla_m D_k(\mathbb{D}X^x_t)^{-1}&=\nabla_m\brac{(\mathbb{D}X^x_t)^{-1}\,D_k\mathbb{D}X^x_t\,(\mathbb{D}X^x_t)^{-1}}\\
        &=\nabla_m(\mathbb{D}X^x_t)^{-1}\,D_k\mathbb{D}X^x_t\,(\mathbb{D}X^x_t)^{-1}+(\mathbb{D}X^x_t)^{-1}\,\nabla_m D_k\mathbb{D}X^x_t\,(\mathbb{D}X^x_t)^{-1}\\
        &\qquad\qquad\qquad+(\mathbb{D}X^x_t)^{-1}\,D_k\mathbb{D}X^x_t\,\nabla_m(\mathbb{D}X^x_t)^{-1}.
    \end{align*}

Then for $q\in [1,\tau)$,
\begin{align}
    \E{\abs{D_k(\mathbb{D}X^x_t)^{-1}}^q}
    &\leq 3^{q-1}\Big(\E{\abs{\nabla_m(\mathbb{D}X^x_t)^{-1}}^q\,\abs{D_k\mathbb{D}X^x_t}^q\,\abs{(\mathbb{D}X^x_t)^{-1}}^q}\nonumber\\
    &\qquad\qquad+\E{\abs{(\mathbb{D}X^x_t)^{-1}}^q\,\abs{\nabla_m D_k\mathbb{D}X^x_t}^q\,\abs{(\mathbb{D}X^x_t)^{-1}}^q}\nonumber\\
&\qquad\qquad\qquad+\E{\abs{(\mathbb{D}X^x_t)^{-1}}^q\,\abs{D_k\mathbb{D}X^x_t}^q\,\abs{\nabla_m(\mathbb{D}X^x_t)^{-1}}^q}\Big).\label{bound:three:terms}
\end{align}

Let us bound each term on the right hand side of \eqref{bound:three:terms} separately. 

First, let us bound the first term on the right hand side of \eqref{bound:three:terms}, noting that it is the same as the third term on the right hand side of \eqref{bound:three:terms}.
Assume $q_0$ is a constant in $(q,\tau)$. 
By applying Young's inequality \eqref{younginequality} twice, first with $a=\frac{q_0}{q},b=\frac{q_0}{q_0-q}$ and then with $a=b=2$, we arrive at
\begin{align*}
&\E{\abs{\nabla_m(\mathbb{D}X^x_t)^{-1}}^q\,\abs{D_k\mathbb{D}X^x_t}^q\,\abs{(\mathbb{D}X^x_t)^{-1}}^q}\\
&\leq \frac{q}{q_0}\E{\abs{D_k\der X^x_t}^{q_0}}+\frac{q_0-q}{2q_0}\brac{\E{\abs{\nabla_m (\der X^x_t)^{-1}}^{\frac{2qq_0}{q_0-q}}} +\E{\abs{(\der X^x_t)^{-1}}^{\frac{2qq_0}{q_0-q}}}}\\
&\leq \frac{q}{q_0}\mathcal{Q}^4_k(q_0,t)+\frac{q_0-q}{2q_0}\bigg(\mathcal{Q}^5\brac{\frac{8qq_0}{q_0-q},t }+\mathcal{Q}^3_m\brac{\frac{4qq_0}{q_0-q},t}+\mathcal{Q}^5\brac{\frac{2qq_0}{q_0-q},t } \bigg). 
\end{align*}
The last line is due to Lemmas~\ref{lemma_DDX},~\ref{lemma_inverseDX},~\ref{lemma_nablainverseDX_anddifferentiablity}. 

Next, let us bound the second term on the right hand side of \eqref{bound:three:terms}.
Similarly, assume $q_0$ as above. Then, via Young's inequality \eqref{younginequality} with $a=\frac{q_0}{q},b=\frac{q_0}{q_0-q}$ and Lemmas~\ref{lemma_inverseDX},~\ref{lemma_nablaDDX},
\begin{align*}
    &\E{\abs{(\mathbb{D}X^x_t)^{-1}}^q\,\abs{\nabla_m D_k\mathbb{D}X^x_t}^q\,\abs{(\mathbb{D}X^x_t)^{-1}}^q}\\
    &\leq \frac{q}{q_0}\E{\abs{\nabla_m D_k\mathbb{D}X^x_t}^{q_0}}+\frac{q_0-q}{q_0}\E{\abs{(\der X^x_t)^{-1}}^{\frac{2qq_0}{q_0-q}}}\\
    &\leq \frac{q}{q_0}\mathcal{Q}^6_{n,k}(q_0,t)+\frac{q_0-q}{q_0}\mathcal{Q}^5\brac{\frac{2qq_0}{q_0-q},t}. 
\end{align*}

Combining the previous calculations yields the desired bound on 
\begin{equation*}
\mathcal{Q}^8_{m,k}(q,t)=\sup_{x\in\R^d}\E{\abs{D_k(\mathbb{D}X^x_t)^{-1}}^q}.
\end{equation*}

Next we deal with the dimension dependence of our upper bound on $\mathcal{Q}^8_{m,k}(q,t) $. Based on previous lemmas, the contribution to the upper bound on $\mathcal{Q}^8_{m,k}(q,t) $ from 
\begin{align*}
\mathcal{Q}^4_k(q_0,t),\quad
\mathcal{Q}^5\brac{\frac{8qq_0}{q_0-q},t },\quad
\mathcal{Q}^3_m\brac{\frac{4qq_0}{q_0-q},t}
\end{align*}
and 
\begin{align*}
    \mathcal{Q}^5\brac{\frac{2qq_0}{q_0-q},t },\quad
\mathcal{Q}^6_{m,k}(q_0,t)
\end{align*}
are respectively of the order $\mathcal{O}\left(d^{q_0+1-\frac{q_0}{\tau}}\right)$, $\mathcal{O}\left(d^{\frac{4qq_0}{q_0-q}}\right)$, $\mathcal{O}\left(d^{\frac{2qq_0}{q_0-q}}\right)$, $\mathcal{O}\left(d^{\frac{qq_0}{q_0-q}}\right)$, $\mathcal{O}\left(d^{q_0+1-\frac{q_0}{\tau}}\right)$ as $d\rightarrow\infty$. Thus, the upper bound on $\mathcal{Q}^8_{m,k}(q,t) $ is of the order $\mathcal{O}\left(d^{\frac{4qq_0}{q_0-q}\vee \brac{q_0+1-\frac{q_0}{\tau}}}\right)$ as $d\rightarrow\infty$. 
\end{proof}


\begin{lemma}
    \label{lemma_AandDAand_nablaDA}
    Assume Condition~\ref{cond_levymeasure} and Condition~\ref{cond_driftb}. Recall the $d\times d$ random matrix $A(x,t)$ in Proposition~\ref{prop_bismut}.  Then for any $t\geq 0$ and $m\leq d$, the $d\times d$ random matrices $A(x,t)$ and $\nabla_m A(x,t)$ are $p$-integrable for $p\geq 1$. Moreover for $k,m\leq d$, the $d\times d$ random matrices  $D_kA(x,t)$ and $\nabla_m D_kA(x,t)$ are $q$-integrable for $\tau> q\geq 1$.

    Specifically, we have for $p\geq 1$
    \begin{align*}
        \mathcal{Q}^9(p,t)&= \sup_{x\in\R^d}\E{\abs{ A(x,t)}^p}\leq e^{p\theta_1 t}\mathcal{Q}^5(p,t),
    \end{align*}
and    
    \begin{align*}
    \mathcal{Q}^{10}_m(p,t)
    &=\sup_{x\in\R^d}\E{\abs{\nabla_m A(x,t)}^p}\\
    &\leq 2^{p-2}\brac{\mathcal{Q}^5(8p,t)+\mathcal{Q}^3_m(4p,t)+\mathcal{Q}^5(2p,t)+e^{2p\theta_1 t}+e^{2p\theta_2t} }. 
\end{align*}
Meanwhile for $q\in [1,\tau)$ and any $q_0$ such that $q<q_0<\tau$, 
\begin{align*}
   \mathcal{Q}^{11}_k(q,t)&=\sup_{x\in\R^d}\E{\abs{D_k A(x,t)}^q}\\
    &\leq 2^{q-1}e^{q\theta_1 t} \brac{\frac{\tau-q}{\tau}\mathcal{Q}^5\brac{\frac{q\tau}{\tau-q},t}+\frac{q}{\tau}\mathcal{Q}^4_k(\tau,t)} +2^{q-2}\mathcal{Q}^5(2q,t)+2^{q-2}\mathcal{Q}^2_k(2q,t),
\end{align*}
and
\begin{align*}
    \mathcal{Q}^{12}_{m,k}(q,t)&=\sup_{x\in\R^d}\E{\abs{\nabla_m D_kA(x,t)}^q}\\
    &\hspace{1em}\leq  4^{q-1}e^{q\theta_1 t} \mathcal{Q}^8_{m,k}(q,t)+4^{q-1}e^{q\theta_2 t} \brac{\frac{\tau-q}{\tau}\mathcal{Q}^5\brac{\frac{q\tau}{\tau-q},t}+\frac{q}{\tau}\mathcal{Q}^4_m(\tau,t)}\\
    &\qquad\qquad+\frac{4^{q-1}}{2}\brac{\frac{\tau-q}{\tau}\mathcal{Q}^5(8q,t)+\frac{q}{\tau}\mathcal{Q}^3_m(4q,t) +\mathcal{Q}^2_k(2q,t)}\\
    &\qquad\qquad\qquad+4^{q-1}\brac{{\frac{1}{q_0}\mathcal{Q}^7_{m,k}(q_0,t)+\frac{q_0-q}{q_0}\mathcal{Q}^5\brac{\frac{qq_0}{q_0-q},t}}}. 
\end{align*}

Regarding dimension dependence, our upper bounds on $\mathcal{Q}^9(p,t), \mathcal{Q}^{10}_m(p,t), \mathcal{Q}^{11}_k(q,t)$ and $\mathcal{Q}^{12}_{m,k}(q,t)$ are respectively of the order $\mathcal{O}\left(d^{\frac{p}{2}}\right), \mathcal{O}\left(d^{4p}\right), \mathcal{O}\left(d^{\frac{q\tau}{2(\tau-q)}\vee \tau}\right)$ and 
\begin{align*}
\mathcal{O}\left(d^{\frac{4qq_0}{q_0-q}}\vee d^{q_0+1-\frac{q_0}{\tau}}\vee d^{\frac{q\tau}{q(\tau-q)}}\vee d^\tau\right), 
\end{align*}
as $d\rightarrow\infty$
\end{lemma}

\begin{proof}
The estimate on $\mathcal{Q}^9(p,t)=\sup_{x\in\R^d}\E{\abs{ A(x,t)}^p}$ is a direct consequence of Lemmas~\ref{lemma_nablaX} and~\ref{lemma_inverseDX}. 

Next, we have
    \begin{align*}
     \nabla_m A(x,t)&=\nabla_m \brac{(\mathbb{D}X^x_t)^{-1}\,\nabla X^x_t}\\
     &={\nabla_m(\mathbb{D}X^x_t)^{-1}\,\nabla X^x_t}+{(\mathbb{D}X^x_t)^{-1}\,\nabla_m\nabla X^x_t}.
\end{align*}
By Young's inequality \eqref{younginequality} with $a=b=2$ and Lemmas~\ref{lemma_nablaX},~\ref{lemma_inverseDX},~\ref{lemma_nablainverseDX_anddifferentiablity}, 
\begin{align*}
    \mathcal{Q}^{10}_m(p,t)
    &=\sup_{x\in\R^d}\E{\abs{\nabla_m A(x,t)}^p}\\
    &\leq 2^{p-2}\bigg(\sup_{x\in\R^d}\E{\abs{\nabla_m(\mathbb{D}X^x_t)^{-1}}^{2p}}+\sup_{x\in\R^d}\E{\abs{\nabla X^x_t}^{2p}}
    \\
    &\qquad\qquad+\sup_{x\in\R^d}\E{\abs{(\mathbb{D}X^x_t)^{-1}}^{2p}}+\sup_{x\in\R^d}\E{\abs{\nabla_m\nabla X^x_t}^{2p}}\bigg)\\
    &\leq 2^{p-2}\brac{\mathcal{Q}^5(8p,t)+\mathcal{Q}^3_m(4p,t)+\mathcal{Q}^5(2p,t)+e^{2p\theta_1 t}+e^{2p\theta_2t} }. 
\end{align*}
In terms of dimension dependence, the quantities $\mathcal{Q}^5(8p,t),\mathcal{Q}^3_m(4p,t),
\mathcal{Q}^5(2p,t)$ are respectively of the order $\mathcal{O}(d^{4p}), \mathcal{O}(d^{2p}), \mathcal{O}(d^p)$, and thus our upper bound on $\mathcal{Q}^{10}_m(p,t)$ is of the order $\mathcal{O}(d^{4p})$ as $d\rightarrow\infty$.

Next, let us consider
\begin{align*}
     D_k A(x,t)&= D_k \brac{(\mathbb{D}X^x_t)^{-1}\,\nabla X^x_t}\\
     &={ D_k(\mathbb{D}X^x_t)^{-1}\,\nabla X^x_t}+{(\mathbb{D}X^x_t)^{-1}\, D_k\nabla X^x_t}.
\end{align*}
We assume $q\in [1,\tau)$. By Young's inequality and Lemmas~\ref{lemma_DXandothers},~\ref{lemma_inverseDX},~\ref{lemma_nablainverseDX_anddifferentiablity},  
\begin{align*}
    &\mathcal{Q}^{11}_k(q,t)=\sup_{x\in\R^d}\E{\abs{D_k A(x,t)}^q}\\
    &\leq 2^{q-1}\brac{\sup_{x\in\R^d}\E{\abs{D_k(\der X^x_t)^{-1} \nabla X^x_t}^q }+\sup_{x\in\R^d}\E{\abs{(\der X^x_t)^{-1} D_k\nabla X^x_t}^q} }\\
    &\leq 2^{q-1}e^{q\theta_1 t}\sup_{x\in\R^d}\E{\abs{D_k(\der X^x_t)^{-1}}^q}+2^{q-2}\sup_{x\in\R^d}\E{\abs{(\der X^x_t)^{-1}}^{2q}}\\
    &\hspace{17em}+2^{q-2}\sup_{x\in\R^d}\E{\abs{D_k\nabla X^x_t}^{2q}}\\
    &\leq 2^{q-1}e^{q\theta_1 t} \brac{\frac{\tau-q}{\tau}\mathcal{Q}^5\brac{\frac{q\tau}{\tau-q},t}+\frac{q}{\tau}\mathcal{Q}^4_k(\tau,t)} \\
    &\hspace{17em}+2^{q-2}\mathcal{Q}^5(2q,t)+2^{q-2}\mathcal{Q}^2_k(2q,t).
\end{align*}
Regarding the dimension dependence, the quantities 
\begin{align*}
    \mathcal{Q}^5\brac{\frac{q\tau}{\tau-q},t},\quad
\mathcal{Q}^2_k(2q,t),\quad
\mathcal{Q}^5(2q,t),\quad\text{and}\quad
\mathcal{Q}^4_k(\tau,t),
\end{align*}
are respectively of the order $\mathcal{O}\left(d^{\frac{q\tau}{2(\tau-q)}}\right)$, $\mathcal{O}(d^\tau)$, $\mathcal{O}(d^q)$  
and $\mathcal{O}(d^q)$ as $d\rightarrow\infty$.
Then our upper bound on 
$\mathcal{Q}^{11}_k(q,t)$ is of the order $\mathcal{O}\left(d^{\frac{q\tau}{2(\tau-q)}\vee \tau}\right)$ as $d\rightarrow\infty$. 

The last thing to study is
\begin{align*}
    \nabla_m D_kA(x,t)&=\nabla_m D_k\brac{(\mathbb{D}X^x_t)^{-1}\,\nabla X^x_t}\\
    &=\nabla_mD_k(\mathbb{D}X^x_t)^{-1}\,\nabla X^x_t +D_k(\mathbb{D}X^x_t)^{-1}\,\nabla_m\nabla X^x_t \\
    &\qquad\qquad+\nabla_m (\mathbb{D}X^x_t)^{-1}\,D_k\nabla X^x_t + (\mathbb{D}X^x_t)^{-1}\,\nabla_m D_k\nabla X^x_t .
\end{align*}
Again let us assume $q\in [1,\tau)$ and $q_0$ is another constant such that $q<q_0<\tau$. Via Young's inequality and Lemmas~\ref{lemma_nablaX},~\ref{lemma_DXandothers},~\ref{lemma_nablaDnablaX},~\ref{lemma_nablainverseDX_anddifferentiablity} and~\ref{lemma_nablaDinverseDX}, we can compute that
\begin{align*}
    \mathcal{Q}^{12}_{m,k}(q,t)
    &=\sup_{x\in\R^d}\E{\abs{\nabla_m D_kA(x,t)}^q}\\
    &\leq 4^{q-1}e^{q\theta_1 t}\sup_{x\in\R^d}\E{\abs{\nabla_m D_k(\der X^x_t)^{-1}}^q}+4^{q-1}e^{q\theta_2 t}\sup_{x\in\R^d}\E{\abs{D_k(\der X^x_t)^{-1}}^q}\\
    &\qquad+\frac{4^{q-1}}{2}\brac{\sup_{x\in\R^d}\E{\abs{\nabla_m(\der X^x_t)^{-1}}^{2q}}+\sup_{x\in\R^d}\E{\abs{D_k\nabla X^x_t}^{2q}} }\\
    &\qquad\qquad+4^{q-1}\brac{\frac{q}{q_0}\sup_{x\in\R^d}\E{\abs{\nabla_mD_k\nabla X^x_t}^{q_0}}+\frac{q_0-q}{q_0}\sup_{x\in\R^d}\E{\abs{(\der X^x_t)^{-1}}^{\frac{qq_0}{q_0-q}}} }\\
    &\leq 4^{q-1}e^{q\theta_1 t} \mathcal{Q}^8_{m,k}(q,t)+4^{q-1}e^{q\theta_2 t} \brac{\frac{\tau-q}{\tau}\mathcal{Q}^5\brac{\frac{q\tau}{\tau-q},t}+\frac{q}{\tau}\mathcal{Q}^4_m(\tau,t)}\\
    &\qquad+\frac{4^{q-1}}{2}\brac{\frac{\tau-q}{\tau}\mathcal{Q}^5(8q,t)+\frac{q}{\tau}\mathcal{Q}^3_m(4q,t) +\mathcal{Q}^2_k(2q,t)}\\
    &\qquad\qquad+4^{q-1}\brac{{\frac{1}{q_0}\mathcal{Q}^7_{m,k}(q_0,t)+\frac{q_0-q}{q_0}\mathcal{Q}^5\brac{\frac{qq_0}{q_0-q},t}}}. 
\end{align*}
Regarding dimension dependence, one can compute that the quantities 
\begin{align*}
&\mathcal{Q}^8_{m,k}(q,t),\mathcal{Q}^5\brac{\frac{q\tau}{\tau-q},t}, \mathcal{Q}^4_m(\tau,t),\mathcal{Q}^5(8q,t),\\
&\mathcal{Q}^3_m(4q,t) ,\mathcal{Q}^2_k(2q,t), \mathcal{Q}^7_{m,k}(q_0,t), \mathcal{Q}^5\brac{\frac{qq_0}{q_0-q},t},
\end{align*}
are respectively of the order 
\begin{align*}
&\mathcal{O}\left(d^{\frac{4qq_0}{q_0-q}\vee \brac{q_0+1-\frac{q_0}{\tau}}}\right), \mathcal{O}\left(d^{\frac{q\tau}{2(\tau-q)}}\right),\mathcal{O}\left(d^{q_0+1-\frac{q_0}{\tau}}\right), \mathcal{O}\left(d^{4q}\right), 
\\
&\mathcal{O}\left(d^{2q}\right), \mathcal{O}(d^q), \mathcal{O}\left(d^{\frac{q_0}{2}}\right), \mathcal{O}\left(d^{\frac{qq_0}{2(q_0-q)}}\right), 
\end{align*}
as $d\rightarrow\infty$.
Since $\tau>q_0>q\geq 1$, we can conclude that our upper bound of $\mathcal{Q}^{12}_{m,k}(q,t)$ is of the order
\begin{align*}
\mathcal{O}\left(d^{\frac{4qq_0}{q_0-q}}\vee d^{q_0+1-\frac{q_0}{\tau}}\vee d^{\frac{q\tau}{q(\tau-q)}}\vee d^\tau\right),
\end{align*}
as $d\rightarrow\infty$. This completes the proof.
\end{proof}

\begin{proof}[Proof of Lemma~\ref{lemma_GandnablaG}]
Note that $q$-integrability of $G(x,t)$ when $q=1$ has been proved in \cite[Theorem 1]{kulik2023gradient}, and here under more stringent conditions, we show an extension to the case $\tau>q\geq 1$. 

Per Proposition~\ref{prop_bismut}, 
    \begin{align*}
    G_j(x,t)=\sum_{k=1}^d \left(A_{k,j}(x,t)B_k(t)-D_kA_{k,j}(x,t)\right),
\end{align*}
so that
\begin{align}\label{G:A:B:D}
    \abs{G(x,t)}\leq \abs{A(x,t)}\abs{B(t)}+\sqrt{d}\sum_{k=1}^d\abs{D_kA(x,t)}. 
\end{align}
To see this, notice that $G(x,t)=R(x,t)-S(x,t)$, 
where 
\begin{align*}
    R(x,t)=(R_{1}(x,t),\ldots,R_{d}(x,t))^{T},\quad S(x,t)=(S_{1}(x,t),\ldots,S_{d}(x,t))^{T}
\end{align*}
with
\begin{equation*}
R_{j}(x,t):=\sum_{k=1}^{d}A_{k,j}(x,t)B_k(t),
\qquad
S_{j}(x,t):=\sum_{k=1}^{d}D_kA_{k,j}(x,t).
\end{equation*}
We can compute that
$|G(x,t)|\leq|R(x,t)|+|S(x,t)|$, 
and moreover, by Cauchy-Schwarz inequality,
\begin{align*}
|R(x,t)|
&=\left(\sum_{j=1}^{d}\left(\sum_{k=1}^{d}A_{k,j}(x,t)B_k(t)\right)^{2}\right)^{1/2}
\\
&\leq
\left(\sum_{j=1}^{d}\sum_{k=1}^{d}|A_{k,j}(x,t)|^{2}\sum_{k=1}^{d}|B_k(t)|^{2}\right)^{1/2}
=\abs{A(x,t)}\abs{B(t)},
\end{align*}
and furthermore,
\begin{align*}
|S(x,t)|
&=\left(\sum_{j=1}^{d}\left(\sum_{k=1}^{d}D_kA_{k,j}(x,t)\right)^{2}\right)^{1/2}
\leq\sum_{j=1}^{d}\left|\sum_{k=1}^{d}D_kA_{k,j}(x,t)\right|
\\
&\leq\sum_{k=1}^{d}\sum_{j=1}^{d}|D_kA_{k,j}(x,t)|
\leq\sqrt{d}\sum_{k=1}^{d}\left(\sum_{j=1}^{d}|D_kA_{k,j}(x,t)|^{2}\right)^{1/2}
\\
&\leq\sqrt{d}\sum_{k=1}^{d}\left(\sum_{j=1}^{d}\sum_{\ell=1}^{d}|D_kA_{\ell,j}(x,t)|^{2}\right)^{1/2}
=\sqrt{d}\sum_{k=1}^d\abs{D_kA(x,t)}. 
\end{align*}
Hence, we proved \eqref{G:A:B:D}.

Now let $q_0\in (q,\tau)$. Via \eqref{G:A:B:D}, Young inequality \eqref{younginequality} with $a=\frac{q_0}{q},b=\frac{q_0}{q_0-q}$ and Lemmas~\ref{lemma_Bk},~\ref{lemma_AandDAand_nablaDA},
\begin{align*}
\sup_{x\in\R^d}\E{\abs{G(x,t)}^q}
&\leq 2^{q-1}\brac{\sup_{x\in\R^d}\E{\abs{A(x,t)}^q\abs{B(t)}^q} +\sqrt{d}\sum_{k=1}^d\sup_{x\in\R^d}\E{\abs{D_kA(x,t)}^q} }\\
&\leq 2^{q-1}\Bigg(\frac{q}{q_0}\sup_{x\in\R^d}\E{\abs{B(t)}^{q_0}}+\frac{q_0-q}{q_0}\sup_{x\in\R^d}\E{\abs{A(x,t)}^{\frac{qq_0}{q_0-q}}}\\
&\hspace{15em}+\sqrt{d}\sum_{k=1}^d\sup_{x\in\R^d}\E{\abs{D_kA(x,t)}^q}  \Bigg)\\
&\leq 2^{q-1}\brac{\frac{q}{q_0}\mathcal{Q}^0(q_0,t)+\frac{q_0-q}{q_0} \mathcal{Q}^9\brac{\frac{qq_0}{q_0-q},t}+\sqrt{d}\sum_{k=1}^d\mathcal{Q}^{11}_k(q,t) }. 
    \end{align*}
In terms of the dimension dependence, one can compute that 
the quantities 
\begin{equation*}
\mathcal{Q}^0(q_0,t), \quad \mathcal{Q}^9\brac{\frac{qq_0}{q_0-q},t}
\quad\text{and}\quad \sqrt{d}
\sum_{k=1}^d\mathcal{Q}^{11}_k(q,t), 
\end{equation*}
are respectively of the order $\mathcal{O}\left(d^{\frac{q}{2}+1-\frac{q}{\tau}}\right), \mathcal{O}\left(d^{\frac{qq_0}{2(q_0-q)}}\right)$ and $\mathcal{O}\left(d^{\frac{3}{2}+\brac{\frac{q\tau}{2(\tau-q)}\vee \tau}}\right)$. Since $q/2+1\geq 3/2$ for $q\geq 1$, our upper bound on $\sup_{x\in\R^d}\E{\abs{G(x,t)}^q}$ is of the order
\begin{align*}
\mathcal{O}\left( d^{\frac{qq_0}{2(q_0-q)}}\vee d^{\frac{3}{2}+\brac{\frac{q\tau}{2(\tau-q)}\vee \tau}}\right),
\end{align*}
as $d\rightarrow\infty$.

Next, to see that $G_j(x,t)$ is differentiable in $x$, we need to check that $A_{k,j}(x,t)$ and $D_kA_{k,j}(x,t)$ are differentiable in $x$. The former has been done in the proof of Lemma~\ref{lemma_nablainverseDX_anddifferentiablity}, and the latter can be shown in a similar way. Therefore, we have
\begin{align*}
    \nabla  G_j(x,t)=\sum_{k=1}^d  \left(\nabla A_{k,j}(x,t)\,B_k(t)- \nabla D_kA_{k,j}(x,t)\right). 
\end{align*} 


Next, let us prove $q$-integrability of $ \nabla  G(x,t)$ for $\tau>q\geq 1$. Note that we have
\begin{align*}
    \abs{\nabla G(x,t)}=\brac{\sum_{m=1}^d \abs{\nabla_m G(x,t)}^2}^{\frac{1}{2}},
\end{align*}
and similar as in \eqref{G:A:B:D}, we can show that
\begin{align*}
    \abs{\nabla_m G(x,t)}\leq \abs{\nabla_m A(x,t)}\abs{B(t)}+\sqrt{d}\sum_{k=1}^d\abs{\nabla_m D_k A(x,t)}. 
\end{align*}
As before, we assume that $q_0\in (q,\tau)$. By Young's inequality with $a=\frac{q_0}{q},b=\frac{q_0}{q_0-q}$ and Lemmas~\ref{lemma_Bk},~\ref{lemma_AandDAand_nablaDA}, 
\begin{align*}
   & \sup_{x\in\R^d}\E{\abs{\nabla_m G(x,t)}^q}\\
   &\leq 2^{q-1}\brac{\E{\abs{B(t)}^q \abs{\nabla_m A(x,t)}^q}+\sqrt{d} \sum_{k=1}^d\sup_{x\in\R^d}\E{\abs{\nabla_m D_k A(x,t)}^q} }\\
   &\leq 2^{p-1}\Bigg(\frac{q}{q_0}\E{\abs{B(t)}^{q_0}}+\frac{q_0-q}{q_0}\sup_{x\in\R^d}\E{\abs{\nabla_m A(x,t)}^{\frac{qq_0}{q_0-q}}}\\
   &\hspace{15em}+\sqrt{d}\sum_{k=1}^d\sup_{x\in\R^d}\E{\abs{\nabla_m D_k A(x,t)}^q} \Bigg)\\
   &\leq 2^{q-1}\brac{\frac{q}{q_0}\mathcal{Q}^0(q_0,t)+\frac{q_0-q}{q_0}\mathcal{Q}^{10}_m\brac{\frac{qq_0}{q_0-q},t}+\sqrt{d}\sum_{k=1}^d\mathcal{Q}^{12}_{k,m}(q,t) }.
\end{align*}
Therefore,
\begin{align*}
   &\sup_{x\in\R^d}\E{ \abs{\nabla G(x,t)}^q}
   \\
   &\leq 2^{q-1} \brac{\sum_{m=1}^d \brac{\frac{q}{q_0}\mathcal{Q}^0(q_0,t)+\frac{q_0-q}{q_0}\mathcal{Q}^{10}_m\brac{\frac{qq_0}{q_0-q},t}+\sqrt{d}\sum_{k=1}^d\mathcal{Q}^{12}_{k,m}(q,t) }^2}^{\frac{1}{2}}. 
\end{align*}
Finally, let us calculate the dimension dependence. One can compute that the quantities 
\begin{align*}
    \mathcal{Q}^0(q_0,t),\quad \mathcal{Q}^{10}_m\brac{\frac{qq_0}{q_0-q},t}
    \quad\text{and}\quad
    \sqrt{d}\sum_{k=1}^d\mathcal{Q}^{12}_{k,m}(q,t),
\end{align*}
are respectively of the order (in terms of dimension dependence) $\mathcal{O}\left(d^{\frac{q}{2}+\frac{3}{2}-\frac{q}{\tau}}\right), \\\mathcal{O}\left(d^{\frac{4qq_0}{q_0-q}+\frac{1}{2}}\right)$ and 
\begin{align*}
\mathcal{O}\left(d^{\frac{4qq_0}{q_0-q}+\frac{3}{2}}\vee d^{q_0+\frac{5}{2}-\frac{q_0}{\tau}}\vee d^{\frac{q\tau}{q(\tau-q)}+\frac{3}{2}}\vee d^{\tau+\frac{3}{2}}\right), 
\end{align*}
as $d\rightarrow\infty$.
Therefore, our upper bound on $\sup_{x\in\R^d}\E{ \abs{\nabla G(x,t)}^q}$ is of the order 
\begin{align*}
\mathcal{O}\left(d^{\frac{4qq_0}{q_0-q}+\frac{5}{2}}\vee d^{q_0+\frac{7}{2}-\frac{q_0}{\tau}}\vee d^{\frac{q\tau}{q(\tau-q)}+\frac{5}{2}}\vee d^{\tau+\frac{5}{2}}\right), 
\end{align*}
as $d\rightarrow\infty$.
The proof is complete.
\end{proof}

\section{The Ornstein-Uhlenbeck process}\label{sec:OU}

In this section, $\{L^\alpha_t:t\geq 0\}$ will denote a one-dimensional $\alpha$-stable L\'{e}vy process where $1< \alpha< 2$ and the associated L\'{e}vy measure is ${p_\alpha}/{\abs{z}^{\alpha+1}}$, with ${p_\alpha}$ defined in \eqref{p:alpha}. Via a direct calculation, we will show that the rate $\eta$ in Theorem \ref{theorem_eulerscheme} can be achieved for the Euler scheme of a one-dimensional Ornstein-Uhlenbeck process driven by $L^\alpha$.

Let us consider the process
\begin{align*}
    dX_t=-X_tdt+dL^\alpha_t,
\end{align*}
and its Euler discretization
\begin{align*}
Y_{m+1}&=Y_m+\eta Y_m+\xi_m, \quad Y_0=x. 
\end{align*}
Here $\xi_m:= L^\alpha_{(m+1)\eta}-L^\alpha_{m\eta}$ $\eta>0$ is the stepsize, and $\xi_m,m\in \N$ is a family of i.i.d. stable random variables such that $\xi_{m}\sim\xi$ in distribution for every $m\in\mathbb{N}$, 
where $\xi$ has the with characteristic function $\mathbb{E}[e^{iu\xi}]=\exp \brac{\abs{u}^\alpha}$
for any $u\in\mathbb{R}$.

Denote $X_\infty$ and $Y_\infty$ as the random variables that are distributed as the invariant measures of the process $X_t,t\geq 0$ and the Markov chain $Y_m,m\in\N$ respectively. \cite[Lemma 3]{raj2023algorithmic} tells us that
\begin{align*}
    \E{\exp(iuX_\infty)}=\exp\brac{-\frac{1}{\alpha}\abs{u}^\alpha},
\end{align*}
which implies that $X_\infty\sim \brac{\frac{1}{\alpha}}^{1/\alpha}\xi$ in distribution. Meanwhile, \cite[Corrolary 11]{raj2023algorithmic} and $\sum_{j=0}^\infty (1-\eta)^{\alpha j}=\frac{1}{1-(1-\eta)^\alpha}$ imply that
\begin{align*}
    \E{\exp\brac{iuY_\infty }}= \exp\brac{-\frac{\eta}{1-(1-\eta)^\alpha} \abs{u}^\alpha},
\end{align*}
and hence $Y_\infty\sim \brac{\frac{\eta}{1-(1-\eta)^\alpha} }^{1/\alpha} \xi$ in distribution. 

Therefore, by the definition of Wasserstein distance, we have 
\begin{align*}
    d_{\operatorname{Wass}}\brac{\operatorname{Law}(X_\infty),\operatorname{Law}(Y_\infty)}\leq \E{\abs{\xi}}\abs{\brac{\frac{\eta}{1-(1-\eta)^\alpha} }^{1/\alpha}-\brac{\frac{1}{\alpha}}^{1/\alpha} }. 
\end{align*}
Next, we define $P(\alpha):=\brac{\frac{\eta}{1-(1-\eta)^\alpha} }^{1/\alpha}-\brac{\frac{1}{\alpha}}^{1/\alpha}$
and consider its dependence on $\eta$ as $\eta\rightarrow 0$. Note that we have
\begin{align*}
    (1-\eta)^\alpha=1-\alpha\eta +\frac{\alpha(\alpha-1)}{2}\eta^2+\mathcal{O}\brac{\eta^3},
\end{align*}
so that 
\begin{align*}
    \frac{\eta}{1-(1-\eta)^\alpha}&=\frac{\eta}{\alpha\eta -\frac{\alpha(\alpha-1)}{2}\eta^2-\mathcal{O}\brac{\eta^3}}=\frac{1}{\alpha}\brac{1+\frac{\alpha+1}{2}\eta+\mathcal{O}\brac{\eta^2}},
\end{align*}
as $\eta\rightarrow 0$. 
This infers that
\begin{align*}
    P(\alpha)= \brac{\frac{1}{\alpha}}^{1/\alpha}\brac{\brac{ 1+\frac{\alpha+1}{2}\eta+\mathcal{O}\brac{\eta^2}
 }^{1/\alpha}-1 }=\mathcal{O}\brac{\eta},
\end{align*}
and consequently
\begin{align*}
     d_{\operatorname{Wass}}\brac{\operatorname{Law}(X_\infty),\operatorname{Law}(Y_\infty)}\leq \mathcal{O}\brac{\eta}. 
\end{align*}
Hence, the discretization error in Wasserstein distance has the linear dependence on the stepsize $\eta$, 
which is the same as in Theorem~\ref{theorem_eulerscheme} in the main paper.

 \end{document}